\documentclass[11pt,reqno,draft]{amsart}
\usepackage{amsmath}
\usepackage{bm}
\usepackage[applemac]{inputenc}
\usepackage{amsmath}
\usepackage{amssymb}
\usepackage{euscript}
\usepackage{mathrsfs}
\usepackage{wasysym}
\usepackage[all]{xy}
\usepackage{color}
\usepackage[colorlinks=true,linktocpage=true,pagebackref=false, urlcolor=black, citecolor=black,linkcolor=black]{hyperref}
\usepackage{tikz}

\newtheorem{thm}{Theorem}[section]

\newtheorem{lem}[thm]{Lemma}
\newtheorem{prop}[thm]{Proposition}

\newtheorem{cor}[thm]{Corollary}

\theoremstyle{definition}

\theoremstyle{remark}
\newtheorem{remark}[thm]{Remark}

\newtheorem{examples}[thm]{Examples}

\numberwithin{equation}{section}
\renewcommand\thesubsection{\thesection.\Alph{subsection}}
\renewcommand{\theparagraph}{\thesubsection.\arabic{paragraph}}
\renewcommand{\theequation}{\thesection.\arabic{equation}}
 \newcommand{\R}{{\mathbb R}}

 \newcommand{\Cont}{{\mathcal C}}

\newcommand{\gtp}{{\mathfrak p}} \newcommand{\gtq}{{\mathfrak q}}
\newcommand{\gtm}{{\mathfrak m}} \newcommand{\gtn}{{\mathfrak n}}
\newcommand{\gta}{{\mathfrak a}} \newcommand{\gtb}{{\mathfrak b}}
\newcommand{\gtP}{{\mathfrak P}} \newcommand{\gtQ}{{\mathfrak Q}}
 
 \newcommand{\gtA}{{\mathfrak A}}

\newcommand{\Dd}{{\EuScript D}}
\newcommand{\Zz}{{\EuScript Z}}

\newcommand{\bs}{{\EuScript B}}

\newcommand{\Ff}{{\EuScript F}}

\newcommand{\Hom}{\operatorname{Hom}}
\newcommand{\im}{\operatorname{im}}
\newcommand{\qf}{\operatorname{qf}}

\newcommand{\dist}{\operatorname{dist}}

\newcommand{\Spec}{\operatorname{Spec}}
\newcommand{\Specm}{\operatorname{\Spec_{\rm max}}}

\newcommand{\tr}{\operatorname{tr}}

\newcommand{\id}{\operatorname{id}}

\newcommand{\cl}{\operatorname{Cl}}

\newcommand{\lc}{\operatorname{lc}}
\newcommand{\diam}{{\text{\tiny$\displaystyle\diamond$}}}
\newcommand{\gtmd}{\operatorname{\gtm^{\diam}\hspace{-1.5mm}}}

\newcommand{\Specs}{\operatorname{Spec_s}}
\newcommand{\Speca}{\operatorname{Spec_s^*}}
\newcommand{\Specd}{\operatorname{Spec_s^{\diam}}}
\newcommand{\betas}{\operatorname{\beta_s\!}}
\newcommand{\betaa}{\operatorname{\beta_s^*\!\!}}
\newcommand{\betad}{\operatorname{\beta_s^{\diam}\!\!}}

\newcommand{\x}{{\tt x}}  
 \renewcommand{\t}{{\tt t}}

\newcommand{\ol }{\overline}
\newcommand{\qq}[1]{\langle{#1}\rangle}

\numberwithin{equation}{section}

\renewcommand\thesubsection{\thesection.\Alph{subsection}}
\renewcommand{\theequation}{\thesection.\arabic{equation}}

\begin{document}

\title[{On the size of the fibers of spectral maps}]{On the size of the fibers of spectral maps\\ induced by semialgebraic embeddings}
\author{Jos\'e F. Fernando}
\address{Departamento de \'Algebra, Facultad de Ciencias Matem\'aticas, Universidad Complutense de Madrid, 28040 MADRID (SPAIN)}
\email{josefer@mat.ucm.es}

\date{23/11/2015}

\thanks{Author supported by Spanish GR MTM2011-22435 and Spanish MTM2014-55565. %A great part of this article has been performed in the course of a researching stay of the author in the Department of Mathematics of the University of Pisa.
}
\keywords{Semialgebraic set, semialgebraic function, Zariski spectrum, spectral map, {\tt sa}-tuple, suitably arranged {\tt sa}-tuple, singleton fiber, finite fiber, infinite fiber.\\}

\begin{abstract}
Let ${\mathcal S}(M)$ be the ring of (continuous) semialgebraic functions on a semialgebraic set $M$ and ${\mathcal S}^*(M)$ its subring of bounded semialgebraic functions. In this work we compute the size of the fibers of the spectral maps $\Spec({\tt j})_1:\Spec({\mathcal S}(N))\to\Spec({\mathcal S}(M))$ and $\Spec({\tt j})_2:\Spec({\mathcal S}^*(N))\to\Spec({\mathcal S}^*(M))$ induced by the inclusion ${\tt j}:N\hookrightarrow M$ of a semialgebraic subset $N$ of $M$. The ring ${\mathcal S}(M)$ can be understood as the localization of ${\mathcal S}^*(M)$ at the multiplicative subset ${\mathcal W}_M$ of those bounded semialgebraic functions on $M$ with empty zero set. This provides a natural inclusion ${\mathfrak i}_M:\Spec({\mathcal S}(M))\hookrightarrow\Spec({\mathcal S}^*(M))$ that reduces both problems above to an analysis of the fibers of the spectral map $\Spec({\tt j})_2:\Spec({\mathcal S}^*(N))\to\Spec({\mathcal S}^*(M))$. If we denote $Z:=\cl_{\Spec({\mathcal S}^*(M))}(M\setminus N)$, it holds that the restriction map $\Spec({\tt j})_2|:\Spec({\mathcal S}^*(N))\setminus\Spec({\tt j})_2^{-1}(Z)\to\Spec({\mathcal S}^*(M))\setminus Z$ is a homeomorphism. 

Our problem concentrates on the computation of the size of the fibers of $\Spec({\tt j})_2$ at the points of $Z$. The size of the fibers of prime ideals `close' to the complement $Y:=M\setminus N$ provides valuable information concerning how $N$ is immersed inside $M$. If $N$ is dense in $M$, the map $\Spec({\tt j})_2$ is surjective and the generic fiber of a prime ideal $\gtp\in Z$ contains infinitely many elements. However, finite fibers may also appear and we provide a criterium to decide when the fiber $\Spec({\tt j})_2^{-1}(\gtp)$ is a finite set for $\gtp\in Z$. If such is the case, our procedure allows us to compute the size $s$ of $\Spec({\tt j})_2^{-1}(\gtp)$. If in addition $N$ is locally compact and $M$ is pure dimensional, $s$ coincides with the number of minimal prime ideals contained in $\gtp$.
\end{abstract}

\maketitle

\section{Introduction}\label{s1}
\renewcommand{\theequation}{I.\arabic{equation}}

This paper is part of a larger project of studying semialgebraic sets via their ring of continuous semialgebraic functions. There is an extensive classical literature on the ring ${\mathcal C}(X)$ of continuous functions on a Hausdorff space $X$ and a large part of it is collected in the celebrated book \cite{gj}. The space $X$ is canonically embedded into the spectrum of that ring and much of its topology/geometry can be recovered from the structure of the full spectrum. The tame behavior of semialgebraic functions adds some extra structure and finiteness properties. In particular, much more can be said in the non-locally compact case by carefully reducing to the ring of bounded semialgebraic functions, which is also one of the technical points in this paper.

\subsection{Motivation and preliminary notations} 
A semialgebraic set $M\subset\R^m$ is a boolean combination of sets defined by polynomial equations and inequalities. A continuous map $f:M\to\R^n$ is \emph{semialgebraic} if its graph is a semialgebraic subset of $\R^{m+n}$. As usual $f$ is a \em semialgebraic function \em when $n=1$ and $Z(f)$ denotes its zero set. The sum and product of functions defined pointwise endow the set ${\mathcal S}(M)$ of semialgebraic functions on $M$ with a structure of a unital commutative ring. In fact ${\mathcal S}(M)$ is an $\R$-algebra and the subset ${\mathcal S}^*(M)$ of bounded semialgebraic functions on $M$ is an $\R$-subalgebra of ${\mathcal S}(M)$. In this article $M$ denotes a semialgebraic subset of $\R^m$ and we write ${\mathcal S}^{\diam}(M)$ when referring to both rings ${\mathcal S}(M)$ and ${\mathcal S}^*(M)$ indistinctly. To simplify notation we write $\Specd(M):=\Spec({\mathcal S}^{\diam}(M))$ and $\betad M:=\Specm({\mathcal S}^{\diam}(M))$ to respectively denote the Zariski and the maximal spectra of ${\mathcal S}^{\diam}(M)$. In addition, $\partial M:=\betas^*M\setminus M$ is the \em remainder of $M$\em. Given a semialgebraic map ${\tt h}:M_1\to M_2$, we denote the ring homomorphism induced by ${\tt h}$ with ${\tt h}^{\diam,\ast}:{\mathcal S}^{\diam}(M_2)\to{\mathcal S}^{\diam}(M_1),\ f\mapsto f\circ{\tt h}$. This ring homomorphism is injective if and only if ${\tt h}(M_1)$ is dense in $M_2$. The \em spectral map \em induced by ${\tt h}$ is $\Specd({\tt h}):\Specd(M_1)\to\Specd(M_2),\ \gtp\mapsto ({\tt h}^{\diam,\ast})^{-1}(\gtp)$. As it is continuous, it maps $\Specd(M_1)$ into $\cl_{\Specd(M_2)}(\cl_{M_2}({\tt h}(M_1)))\cong\Specd(\cl_{M_2}({\tt h}(M_1)))$ (see \ref{closure}(ii)), so the fiber of each prime ideal belonging to $\Specd(M_2)\setminus\cl_{\Specd(M_2)}(\cl_{M_2}({\tt h}(M_1)))$ is empty. In addition the map $\Speca({\tt h}):\Speca(M_1)\to\Speca(M_2)$ maps $\betaa M_1$ into $\betaa M_2$ and we write $\betaa\,{\tt h}:=\Speca({\tt h})|_{{\betaa M_1}}:{\betaa M_1}\to{\betaa M_2}$. 

The rings ${\mathcal S}^{\diam}(M)$ are particular cases of the so-called \em real closed rings \em introduced by Schwartz in the '80s of the last century \cite{s0}. The theory of real closed rings has been deeply developed until now in a fruitful attempt to establish new foundations for semi-algebraic geometry with relevant interconnections to model theory, see the results of Cherlin-Dickmann \cite{cd1,cd2}, Schwartz \cite{s0,s1,s2,s3}, Schwartz with Prestel, Madden and Tressl \cite{ps,sm,scht} and Tressl \cite{t0,t1,t2}. We refer the reader to \cite{s1} for a ring theoretic analysis of the concept of real closed ring. Moreover, this theory, which vastly generalizes the classical techniques concerning the semi-algebraic spaces of Delfs-Knebusch \cite{dk2}, provides a powerful machinery to approach problems concerning rings of real valued functions, like: (1) real closed fields; (2) rings of real-valued continuous functions on Tychonoff spaces; (3) rings of semi-algebraic functions on semi-algebraic subsets of $\R^n$; and more generally (4) rings of definable continuous functions on definable sets in o-minimal expansions of real closed fields. In addition, the theory of real closed rings contributes to achieve a better understanding of the algebraic properties of such rings and the topological properties of their spectra. 

It is natural to wonder whether the ring ${\mathcal S}(M)$ determines the semialgebraic set $M$. Given another semialgebraic set $N$, the natural map 
$$
(\cdot)^*:{\mathcal S}(M,N)\to\Hom_{\R\text{-alg}}({\mathcal S}(N),{\mathcal S}(M)),\ {\tt h}\mapsto {\tt h}^*
$$ 
where ${\tt h}^*:{\mathcal S}(N)\to{\mathcal S}(M),\ f\mapsto f\circ{\tt h}$ is a bijection. Consequently, \em $M$ and $N$ are semialgebraically homeomorphic if and only if the rings ${\mathcal S}(M)$ and ${\mathcal S}(N)$ are isomorphic\em. This argument goes back to the pioneer work of Schwartz \cite{s0,s1,s2}. Consequently, the category of semialgebraic sets is faithfully reflected in the full subcategory of real closed rings consisting of all $\R$-algebras of the form ${\mathcal S}(M)$. Next, one wonders whether the ring ${\mathcal S}^*(M)$ determines the semialgebraic set $M$. A point $p\in M$ is an \em endpoint of $M$ \em if it has an open neighborhood $U\subset M$ equipped with a semialgebraic homeomorphism $f:U\to[0,1)$ that maps $p$ onto $0$. We denote $\eta(M)$ the set of endpoints of $M$. In \cite[\S11]{t1} it is shown that for every real closed ring $A$ there exists a largest real closed ring $B$ such that $A$ is convex in $B$. In \cite{s3} it is shown how the spectrum of a real closed ring lies in the spectrum of any convex subring. Schwartz proved in \cite[\S5]{s4} that ${\mathcal S}(M\setminus\eta(M))$ is the convex closure of the real closed ring ${\mathcal S}^*(M)={\mathcal S}^*(M\setminus\eta(M))$. If ${\mathcal S}^*(N)$ and ${\mathcal S}^*(M)$ are isomorphic as $\R$-algebras, then their convex closures ${\mathcal S}(N\setminus\eta(N))$ and ${\mathcal S}(M\setminus\eta(M))$ are also isomorphic as $\R$-algebras. Consequently, the semialgebraic sets $M\setminus\eta(M)$ and $N\setminus\eta(N)$ are semialgebraically homeomorphic. 

The study of the fibers of the spectral map $\varphi^*:\Spec B\to\Spec A$ associated to a ring homomorphism $\varphi:A\to B$ is a recurrent topic in algebraic and analytic geometry. A morphism between schemes is \em quasi-finite \em if it is of finite type and its fibers are finite. The cardinal of the fibers of $\varphi^*$ is upperly bounded by the rank of $B$ as $A$-module. Chevalley's theorem states the semicontinuity of the dimension of the fibers of morphisms of schemes that are locally of finite type \cite[13.1.3]{gd}. This result is true for the spectral morphisms induced by rational maps between complex algebraic varieties. 

In the analytic setting recall Grauert--Remmert's theorem \cite{grm}: \em an analytic map between analytic spaces $f:X\to Y$ is open if its fibers have pure dimension equal to $\dim(X)-\dim(Y)$\em. A deeper study of the fibers of analytic mappings of real or complex spaces is presented in \cite{bf}. 

Our point of view concerning the cardinality of fibers of spectral maps is closer to the latter case. The first steps in this direction are due to Brumfiel \cite{br}, Bochnak-Coste-Roy \cite[\S7]{bcr} and has a further precedent devised by Schwartz \cite{s5}. In addition \cite[Appendix A]{dk} explains the relationship between morphisms of semialgebraic spaces and their \em abstractions\em, which have the spectra of rings of semialgebraic functions as basic building blocks. Both latter articles show that the behavior of the fibers of the induced spectral mappings provides also geometric information. This appears also in \cite{fg6} where it is shown: \em a continuous semialgebraic map $h: N\to M$ is open, proper and surjective if and only if the induced spectral map is open, proper and surjective and $\betaa\,{h}(\partial N)=\partial M$\em. 

Roughly speaking our goal is the study of the spectral map associated to a suitable inclusion of semialgebraic sets $N\hookrightarrow M$ and to characterize when the fibers of this spectral map are finite. Even though the fibers of the inclusion $N\hookrightarrow M$ are either empty or singletons, this is not longer true in general for the associated spectral map and the size of its fibers provides geometric information about the embedding of $N$ in $M$. If $M$ is a simplicial complex and $N\subset M$ is obtained from $M$ by deleting some of its faces, the size of the fibers of the spectral map associated to the inclusion $N\hookrightarrow M$ provides information concerning the `nature of the deleted faces'. As all semialgebraic sets are triangulable, the previous fact will allow us to understand the semialgebraic case.

Although the rings ${\mathcal S}^{\diam}(M)$ are neither noetherian nor enjoy primary decomposition properties, they are closer to polynomial rings than to classical rings of continuous functions. For example, the Lebesgue dimension of $\R$ is $1$ (see Problem 16F in \cite{gj}) while the Krull dimension of the ring $\Cont(\R)$ of real valued continuous functions on $\R$ is infinite, see Problem 14I in \cite{gj}. Carral and Coste proved the equality $\dim({\mathcal S}(M))=\dim(M)$ for a locally closed semialgebraic set $M$ in \cite{cc} (see also \cite{g,s2,s4}) by proving that the real spectrum of ${\mathcal S}(M)$ is homeomorphic to the constructible subset $\widetilde{M}$ of the real spectrum of the ring of polynomial functions on $\R^m$ associated to $M$ (see \cite[Chapter 7]{bcr} for the technicalities concerning the real spectrum). Gamboa-Ruiz extended this equality to an arbitrary semialgebraic set in \cite{gr} using strong properties of the real spectrum of excellent rings and some crucial results of the theory of real closed rings \cite{s2}. 

In \cite{fg2} we provide an elementary geometric proof of the fact that the Krull's dimension of the ring ${\mathcal S}^{\diam}(M)$ coincides with the dimension of $M$ but without involving the sophisticated machinery of real spectra. We compute the Krull dimension of the ring ${\mathcal S}^{\diam}(M)$ by comparing it with the Krull dimensions of the rings ${\mathcal S}(X)={\mathcal S}^*(X)$ for suitable semialgebraic compactifications $X$ of $M$; recall that a pair $(X,{\tt j})$ is a \em semialgebraic compactification \em of $M$ if ${\tt j}:M\hookrightarrow X$ is a semialgebraic embedding such that $X$ is a compact semialgebraic set and $\cl({\tt j}(M))=X$. In addition, the ring ${\mathcal S}^*(M)$ is the direct limit of the family constituted by the rings ${\mathcal S}(X)$ where $(X,{\tt j})$ runs on the semialgebraic compactifications of $M$. 

We prove in \cite{fe1} that semialgebraic compactifications provide further information to study chains of prime ideals in rings of semialgebraic functions by comparing the spectra $\Specs(M)$ and $\Specs(X)$ where $X$ is a suitable semialgebraic compactification of $M$. The main purpose of \cite{fe1} is to understand the structure of non refinable chains of prime ideals of the ring ${\mathcal S}^*(M)$ for an arbitrary semialgebraic set $M$ (not necessarily locally closed). The article \cite{fe1} somehow completes the work already began in \cite{fg3}, in which we studied some algebraic, topological and functorial properties of the Zariski and maximal spectra of the rings ${\mathcal S}^{\diam}(M)$ for an arbitrary semialgebraic set $M$. Moreover, our results generalize some similar already known ones for the o-minimal context in the exponentially bounded and polynomially bounded cases that are developed under the assumption of local closedness \cite{t0}. Recall that ${\mathcal S}^*(M)$ can be understood as the \em ring of holomorphy \em of the real closed ring ${\mathcal S}(M)$ in the sense of \cite[p.40]{t1}. This provides some valuable information in relation with the chains of prime ideals containing a given ideal of ${\mathcal S}^*(M)$. To that end, one can use Gelfand-Kolgomorov's Theorem for rings with normal spectrum and the related results concerning rings of holomorphy \cite[\S10]{t1} applied to the pair of rings ${\mathcal S}^*(M)\subset{\mathcal S}(M)$.

Locally compact semialgebraic spaces (and in particular the compact ones) have an advantageous geometrical behavior \cite{bcr,cc,dk2}. The reason is that a locally compact semialgebraic set $M$ is an open subset of each Hausdorff compactification of $M$. Apart from semialgebraic compactifications, another important source of valuable information to understand $\Specs(M)$ when $M$ is non-locally compact arises from the spectrum $\Spec({\mathcal S}^{\diam}(M_{\lc}))$ where $M_{\lc}$ denotes the (semialgebraic) subset of those points of $M$ that have a compact neighborhood in $M$ (see \ref{1B}). This provides a new evidence of the importance of locally compact semialgebraic sets in semialgebraic geometry. Both types of embeddings $M\hookrightarrow X$, where $X$ is a semialgebraic compactification of $M$, and $M_{\lc}\hookrightarrow M$ share many properties and a general study of the induced spectral maps appears in \cite{fg3}. In this framework the study of the fibers of spectral maps induced by general semialgebraic embeddings plays also an important role and this is the main goal of this work.

\paragraph{} Let us fix a semialgebraic set $N$ contained in $M$. If $N$ is dense in $M$, the inclusion induces a surjective map from the Zariski spectrum of ${\mathcal S}^*(N)$ to the Zariski spectrum of ${\mathcal S}^*(M)$. This map is almost everywhere one-to-one except for what happens `close' to the complement $Y:=M\setminus N$. The size of the fibers of prime ideals `close' to the complement $Y:=M\setminus N$ provides valuable information concerning how $N$ is immersed inside $M$. The existence of infinite fibers is in some sense related to the existence of infinitely many semialgebraic ways to tend to $Y$ inside $N$ and one understands that this always occurs if $Y$ has local codimension $\geq2$ in $M$. The preceding presentation is of course very vague and, as quoted before, one main purpose of this paper is to determine the cases when the fibers are finite. More precisely, we are interested in determining the size of the fibers of the spectral maps 
$$
\Specs({\tt j}):\Specs(N)\to\Specs(M)\quad\text{and}\quad\Speca({\tt j}):\Speca(N)\to\Speca(M)
$$ 
induced by an inclusion ${\tt j}:N\hookrightarrow M$ of semialgebraic sets such that $N$ is dense in $M$. To systematize notations a $5$-tuple $(M,N,Y,{\tt j},{\tt i})$ where
\begin{itemize}
\item[(i)] $N\subset M$ is a dense semialgebraic subset of $M$,
\item[(ii)] $Y:=M\setminus N$ and ${\tt j}:N\hookrightarrow M$ and ${\tt i}:Y\hookrightarrow M$ are the inclusion maps
\end{itemize}
is called a \em semialgebraic tuple \em or a \em {\tt sa}-tuple\em. Of course, the pair $(M,N)$ determines the full tuple. As $N$ is dense in $M$, the ring homomorphism ${\tt j}^{\diam,\ast}:{\mathcal S}^{\diam}(M)\to{\mathcal S}^{\diam}(N)$ is injective and we understand ${\mathcal S}^{\diam}(M)$ as a subring of ${\mathcal S}^{\diam}(N)$.

Observe that $Y=\cl_M(N)\setminus N$ is a semialgebraic subset of $M$ whose dimension is by \cite[2.8.13]{bcr} strictly smaller than $\dim(N)=\dim(M)$. A special relevant type of {\tt sa}-tuple $(M,N,Y,{\tt j},{\tt i})$ arises when $N$ is locally closed; we call it \em suitably arranged {\tt sa}-tuple \em \cite[\S5]{fg3}. Notice that if such is the case, $N$ is open in $M$, so $Y=M\setminus N$ is closed in $M$. 

\subsection{Main results}
To ease the presentation and the ulterior proofs of our main results we collect them in a lemma and three theorems that we state here in the Introduction. We denote the \em local dimension of $M$ at a point $p\in\R^m$ \em with $\dim_p(M)$, see \cite[2.8.11]{bcr} for further details. Fix a {\tt sa}-tuple $(M,N,Y,{\tt j},{\tt i})$ and denote $Z:=\cl_{\Speca(M)}(Y)$. Let ${\mathcal W}_N$ be the multiplicative subset of all bounded semialgebraic functions on $N$ with empty zero set. 

\begin{lem}[Reduction to the ring of bounded semialgebraic functions]\label{main0}
We have:
\begin{itemize}
\item[(i)] The image of $\Specs({\tt j}):\Specs(N)\to\Specs(M)$ is $\{\gtp\in\Specs(M):\ \gtp\cap{\mathcal W}_N=\varnothing\}$. 
\item[(ii)] If $\gtp\cap{\mathcal W}_N=\varnothing$, then $\gtq\cap{\mathcal W}_N=\varnothing$ for all $\gtq\in\Speca({\tt j})^{-1}(\gtp\cap{\mathcal S}^*(M))$ and $\Specs({\tt j})^{-1}(\gtp)=\{\gtq{\mathcal S}(N):\ \gtq\in\Speca({\tt j})^{-1}(\gtp\cap{\mathcal S}^*(M))\}$. In particular, the fibers $\Specs({\tt j})^{-1}(\gtp)$ and $\Speca({\tt j})^{-1}(\gtp\cap{\mathcal S}^*(M))$ have the same size.
\item[(iii)] Let $h\in{\mathcal S}(\cl_{\R^m}(M))$ be such that $Z(h)=\cl_{\R^m}(\cl_{\R^m}(M)\setminus N)$ and denote $S:=\{\gtp\in\Specs(M):\ h\in\gtp\}$. Then the restriction map 
$$
\Specs({\tt j})|:\Specs(N)\setminus\Specs({\tt j})^{-1}(S)\to\Speca(M)\setminus S
$$ 
is a homeomorphism.
\end{itemize}
\end{lem}

\begin{thm}\label{main}
We have:
\begin{itemize}
\item[(i)] The map $\Speca({\tt j}):\Speca(N)\to\Speca(M)$ is surjective.
\item[(ii)] For each closed semialgebraic subset $C$ of $N$, it holds
\begin{align*}
\Speca({\tt j})(\cl_{\Speca(N)}(C))&=\cl_{\Speca(M)}(C),\\
\Speca({\tt j})^{-1}(\cl_{\Speca(M)}(C)\setminus Z)&=\cl_{\Speca(N)}(C)\setminus\Speca({\tt j})^{-1}(Z).
\end{align*}
\item[(iii)] The restriction map $\Speca({\tt j})|:\Speca(N)\setminus\Speca({\tt j})^{-1}(Z)\to\Speca(M)\setminus Z$ is a homeomorphism. 
\item[(iv)] Analogous statements hold for $\betaa\,{\tt j}$ if we substitute $\Speca$ by $\betaa$\,. 
\end{itemize}
\end{thm}

\begin{thm}\label{main2}
We have:
\begin{itemize}
\item[(i)] If the local dimension $\dim_p(M)\geq 2$ for all $p\in Y$, then $Z$ is the smallest closed subset $T$ of $\Speca(M)$ such that the restriction map 
$$
\Speca({\tt j})|:\Speca(N)\setminus\Speca({\tt j})^{-1}(T)\to\Speca(M)\setminus T
$$ 
is a homeomorphism. 
\item[(ii)] If the local dimension $\dim_p(Y)\leq\dim_p(M)-2$ for all $p\in Y$, then $Z$ is the smallest subset $T$ of $\Speca(M)$ such that the restriction map
$$
\Speca({\tt j})|:\Speca(N)\setminus\Speca({\tt j})^{-1}(T)\to\Speca(M)\setminus T
$$ 
is a homeomorphism. If such is the case, given $\gtp\in\Speca(M)$, the fiber
$$
\Speca({\tt j})^{-1}(\gtp)\text{ is }
\left\{
\begin{array}{ll}
\!\!\text{a singleton}&\text{if $\gtp\in\Speca(M)\setminus Z$,}\\[4pt]
\!\!\text{an infinite set}&\text{if $\gtp\in Z$.}
\end{array}
\right.
$$
\item[(iii)] If $\dim(M)=1$, both $\cl_{\Speca(M)}(Y)=Y$ and $\Speca({\tt j})^{-1}(Y)\subset\betaa N\setminus N$ are finite sets.
\item[(iv)] Analogous statements hold for $\betaa\,{\tt j}$ if we substitute $\Speca$ by $\betaa$\,. 
\end{itemize}
\end{thm}

As $M_{\lc}$ is dense in $M$, the tuple $(M,M_{\lc},\rho_1(M):=M\setminus M_{\lc},{\tt j},{\tt i})$ is a suitably arranged {\tt sa}-tuple. It holds by Corollary \ref{rho2} that $\dim_p(\rho_1(M))\leq\dim_p(M)-2$ for all $p\in\rho_1(M)$. Thus, Theorem \ref{main2} applies and provides the size of all fibers of $\Speca({\tt j}):\Speca(M_{\lc})\to\Speca(M)$. 

Our next purpose is to compute the size of the fibers of the spectral map induced by a general {\tt sa}-tuple. As we will see in Section \ref{s4}, we initially reduce this problem to compute the size of the fibers of the spectral map $\Speca({\tt j}):\Speca(N)\to\Speca(M)$ induced by a suitably arranged {\tt sa}-tuple $(M,N,Y,{\tt j},{\tt i})$ where $M$ is pure dimensional.

\subsubsection{Finite fibers and threshold of a prime ideal}
Let $(M,N,Y,{\tt j},{\tt i})$ be a suitably arranged {\tt sa}-tuple such that $M$ is pure dimensional of dimension $d$. Observe that $N\subset M_{\lc}$ because $N$ is locally compact and dense in $M$; in particular, $\rho_1(M):=M\setminus M_{\lc}\subset Y$. Consider the auxiliary suitably arranged {\tt sa}-tuples $(M,M_{\lc},\rho_1(M),{\tt j}_1,{\tt i}_1)$ and $(M_{\lc},N,Y_2:=M_{\lc}\setminus N,{\tt j}_2,{\tt i}_2)$. By Theorem \ref{main2}(ii) we know that if $\gtp\in\cl_{\Speca(M)}(\rho_1(M))$, the fiber $\Speca({\tt j}_1)^{-1}(\gtp)$ is an infinite set. As ${\tt j}={\tt j}_1\circ{\tt j}_2$, also the fiber $\Speca({\tt j})^{-1}(\gtp)=\Speca({\tt j}_2)^{-1}(\Speca({\tt j}_1)^{-1}(\gtp))$ is an infinite set. Thus, it only remains to determine what happens for a prime ideal $\gtp\in\cl_{\Speca(M)}(Y)\setminus\cl_{\Speca(M)}(\rho_1(M))$. 

Let ${\mathcal W}_M$ be the multiplicative set of those $f\in{\mathcal S}^*(M)$ such that $Z(f)=\varnothing$ and define ${\mathcal E}_M$ as the multiplicative set of those $f\in{\mathcal S}(M)$ such that $Z(f)=M\setminus M_{\lc}$. Let $\gtp\not\in\cl_{\Speca(M)}(\rho_1(M))$ be a prime ideal of ${\mathcal S}^*(M)$. As we will see in \ref{spectracomp} there exists a unique maximal ideal $\gtm^*$ of ${\mathcal S}^*(M)$ that contains $\gtp$. Let $\gtm$ be the unique maximal ideal of ${\mathcal S}(M)$ such that $\gtm\cap{\mathcal S}^*(M)\subset\gtm^*$, see \ref{spectracomp}. On the other hand, let $\ol{\gtp}$ be any prime ideal of ${\mathcal S}^*(M)$ contained in $\gtp$ such that $\ol{\gtp}\cap{\mathcal E}_M=\varnothing$ but $\gtq\cap{\mathcal E}_M\neq\varnothing$ for each prime ideal $\gtq$ of ${\mathcal S}^*(M)$ that strictly contains $\ol{\gtp}$. Notice that such a prime ideal $\ol{\gtp}$ exists because by Theorem \ref{minimalnlc} no minimal prime ideal of ${\mathcal S}^*(M)$ intersects ${\mathcal E}_M$. Consider the prime ideal
\begin{equation}\label{gtQ}
\widehat{\gtp}:=\begin{cases}
\ol{\gtp}&\text{ if $\gtp\cap{\mathcal W}_M=\varnothing$},\\
\gtm\cap{\mathcal S}^*(M)&\text{ if $\gtp\cap{\mathcal W}_M\neq\varnothing$}.
\end{cases}
\end{equation}
By \ref{chainp} it holds $\widehat{\gtp}\subset\gtp$, so $\widehat{\gtp}\not\in\cl_{\Speca(M)}(\rho_1(M))$. As we see in Lemma \ref{e}, $\widehat{\gtp}$ is univocally determined by $\gtp$ and if $C$ is a closed subset of $M$ such that $\gtp\in\cl_{\Speca(M)}(C)$, then $\widehat{\gtp}\in\cl_{\Speca(M)}(C)$. In addition every non-refinable chain of prime ideals of ${\mathcal S}^*(M)$ through $\gtp$ contains also $\widehat{\gtp}$. In particular, the minimal prime ideals of ${\mathcal S}^*(M)$ contained in $\widehat{\gtp}$ are the same as those contained in $\gtp$. We call $\widehat{\gtp}$ the \em threshold of $\gtp$ in ${\mathcal S}^*(M)$\em. 

\begin{thm}[Finite fibers]\label{fiberspec} 
The fiber $\Speca({\tt j})^{-1}(\gtp)$ is finite if and only if 
$$
{\tt d}_M(\widehat{\gtp}{\mathcal S}(M)):=\min\{\dim(Z(f)):\,f\in\widehat{\gtp}{\mathcal S}(M)\}=d-1. 
$$
Moreover, if such is the case, the size of $\Speca({\tt j})^{-1}(\gtp)$ coincides with the (finite) number of minimal prime ideals of ${\mathcal S}^*(M)$ contained in $\widehat{\gtp}$ (or equivalently in $\gtp$).
\end{thm}
\begin{remark}
Notice that if $\dim(M)=1$, the previous result can be translated as: \em Assume $p\in Y$. Then the (finite) size of the fiber $\Speca({\tt j})^{-1}(\gtm_p^*)$ equals the number of semialgebraic half-branches of the germ $M_p$ \em (use \cite[7.3]{fe1}).
\end{remark}

The proof of Theorem \ref{fiberspec} relies on the following two lemmas that we prove in Section \ref{s5}.

\begin{lem}\label{pmchain}
Let $\gtm\in\cl_{\Specs(M)}(Y)$ be a maximal ideal of ${\mathcal S}(M)$ such that ${\tt d}_M(\gtm)=d-1$. Let $\gtp_1:=\gtm\cap{\mathcal S}^*(M)\subsetneq\cdots\subsetneq\gtp_r=\gtm^*$ be the collection of all prime ideals of ${\mathcal S}^*(M)$ that contain $\gtp_1$. Let $\gtq_1$ be a prime ideal of ${\mathcal S}^*(N)$ such that $\Speca({\tt j})(\gtq_1)=\gtp_1$ and let $\gtq_1\subsetneq\cdots\subsetneq\gtq_s$ be the collection of all prime ideals of ${\mathcal S}^*(N)$ that contain $\gtq_1$. Then $s=r$.
\end{lem}

\begin{lem}\label{fiberspec3}
Let $\gtp\in\cl_{\Speca(M)}(Y)$ be a prime ideal of ${\mathcal S}^*(M)$ that contains only one minimal prime ideal $\gta$ of ${\mathcal S}^*(M)$ and satisfies $\gtp\cap{\mathcal W}_M=\varnothing$ and ${\tt d}_M(\gtp{\mathcal S}(M))=d-1$. Then the fiber $\Speca({\tt j})^{-1}(\gtp)$ is a singleton. 
\end{lem}

\subsection{Structure of the article} 
In Section \ref{s2} we present the preliminary results used in Section \ref{s4} to prove Lemma \ref{main0} and Theorems \ref{main} and \ref{main2}. The reading can be started directly in Section \ref{s3} and referred to the Preliminaries only when needed. In Section \ref{s3} we provide some Examples \ref{excrucial} to illustrate some of the results stated in the Introduction. Next, we reduce the computation of the size of the fibers of spectral maps induced by semialgebraic embeddings to the case of a pure dimensional suitably arranged {\tt sa}-tuple in Section \ref{s4}. The specific computation of the size of those fibers is the aim of Theorem \ref{fiberspec}. The proof of this result and those of Lemmas \ref{pmchain} and \ref{fiberspec3} are conducted in Section \ref{s5}.

\subsection*{Acknowledgements}
The author is strongly indebted to the anonymous referees for their careful reading and comments that have improved the exposition and clarified some imprecise arguments. He is very grateful to Prof. Gamboa for helpful discussions and subtle comments during the preparation of this paper. The author is also indebted to S. Schramm for a careful reading of the final version and for the suggestions to refine its redaction.

\section{Preliminaries on semialgebraic sets and functions}\label{s2}
\renewcommand{\thethm}{\thesection.\arabic{thm}}
\renewcommand{\theequation}{\thesection.\arabic{equation}}

In the following $M\subset\R^m$ denotes a semialgebraic set. For each function $f\in{\mathcal S}^{\diam}(M)$ and each semialgebraic subset $S\subset M$ we denote $Z_S(f):=\{x\in S:\, f(x)=0\}$ and $D_S(f):=S\setminus Z_S(f)$. If $S=M$, we say that $Z(f):=Z_M(f)$ is the \em zero set \em of $f$ and we write $D(f):=D_M(f)$. Sometimes it will be useful to assume that the semialgebraic set $M$ we are working with is bounded. Such assumption can be done without loss of generality because the semialgebraic homeomorphism
$$
{\tt h}:\{x\in\R^m:\ \|x\|<1\}\to\R^m,\ x\mapsto\frac{x}{\sqrt{1-\|x\|^2}}
$$
induces a ring isomorphism ${\mathcal S}(M)\to {\mathcal S}({\tt h}^{-1}(M)),\ f\mapsto f\circ{\tt h}$ that maps ${\mathcal S}^\diam(M)$ onto ${\mathcal S}^\diam({\tt h}^{-1}(M))$.

A crucial fact when dealing with ${\mathcal S}^{\diam}(M)$ is that every closed semialgebraic subset $Z$ of $M$ is the zero-set $Z(h)$ of the (bounded) semialgebraic function $h:=\min\{1,\dist(\cdot,Z)\}$ on $M$. We will use that the difference $\cl_{\R^m}(S)\setminus S$ has by \cite[2.8.13]{bcr} dimension strictly smaller than $S$ for each semialgebraic set $S\subset\R^m$.

\subsection{Bricks of a semialgebraic set}\label{bricks}
Recall the following decomposition of $M$ as an irredundant finite union of closed pure dimensional semialgebraic subsets of $M$ as well as some of its main properties. \em There exists a unique finite family $\{M_1,\ldots,M_r\}$ of semialgebraic subsets of $M$ satisfying the following properties: 
\begin{itemize}
\item[(i)] Each $M_i$ is the closure in $M$ of the set of points of $M$ whose local dimension is equal to some fixed value. In particular, $M_i$ is pure dimensional and closed in $M$. 
\item[(ii)] $M=\bigcup_{i=1}^rM_i$.
\item[(iii)] $M_i\setminus\bigcup_{j\neq i}M_j$ is dense in $M_i$.
\item[(iv)] $\dim(M_i)>\dim(M_{i+1})$ for $i=1,\ldots,r-1$. In particular, $\dim(M_1)=\dim(M)$.
\end{itemize}\em

We call the sets $M_i$ the \em bricks \em of $M$ and denote the \em family of bricks of $M$ \em with $\bs_M:=\{\bs_i(M):=M_i\}_{i=1}^r$. Moreover, \em if $N\subset M$ is a dense semialgebraic subset of $M$, the families $\bs_N$ and $\bs_M$ of bricks of $N$ and $M$ satisfy the following relations:
\begin{itemize}
\item[(1)] $\bs_M:=\{\bs_i(M)=\cl_M(\bs_i(N))\}_i$,
\item[(2)] $\bs_N:=\{\bs_i(N)=\bs_i(M)\cap N\}_i$.
\end{itemize}\em

\subsection{Locally closed semialgebraic sets}\label{1B}
Local closedness has been revealed as an important property for the validity of results that are in the core of semialgebraic geometry. This property is the key assumption to guarantee a Hilbert's Nullstellensatz for the ring ${\mathcal S}(M)$ and consequently to assure that the radical ideals of ${\mathcal S}(M)$ coincide with the zero ideals of ${\mathcal S}(M)$ (commonly named as $z$-ideals). The presence of non units with empty zero set in ${\mathcal S}^*(M)$ requires a more sophisticated Nullstellensatz for this ring \cite{fg1}. Locally closed semialgebraic subsets of $\R^n$ coincide with locally compact ones because the sets $\cl_{\R^m}(M)$ and $U:=\R^m\setminus(\cl_{\R^m}(M)\setminus M)$ are semialgebraic. If $M$ is locally compact, $U$ is open in $\R^m$ and $M$ is the intersection of a closed and an open semialgebraic subset of $\R^m$. Let us recall some of the main properties of the largest locally compact and dense subset $M_{\lc}$ of a semialgebraic set $M$. Its construction is the main goal of \cite[9.14-9.21]{dk2}.

\begin{prop}\label{rho} 
Define $\rho_0(M):=\cl_{\R^m}(M)\setminus M$ and 
$$
\rho_1(M):=\rho_0(\rho_0(M))=\cl_{\R^m}(\rho_0(M))\cap M. 
$$
Then the semialgebraic set $M_{\lc}:=M\setminus\rho_1(M)=\cl_{\R^m}(M)\setminus\cl_{\R^m}(\rho_0(M))$ is the largest locally compact and dense subset of $M$ and coincides with the set of points of $M$ that have a compact neighborhood in $M$.
\end{prop}

\begin{cor}\label{rho2}
Suppose that $\rho_1(M)\neq\varnothing$. Then the local dimension $\dim_p(M)\geq 2$ and $\dim_p(\rho_1(M))\leq\dim_p(M)-2$ for each point $p\in\rho_1(M)$. 
\end{cor}
\begin{proof}[\sl Proof]
Let $p\in\rho_1(M)$ and suppose by contradiction that $\dim_p(M)\leq 1$. Let $U$ be an open neighborhood of $p$ in $\R^m$ such that $d:=\dim(M\cap U)=\dim_p(M)\leq 1$. As $\rho_0(M\cap U)=\cl_{\R^m}(M\cap U)\setminus (M\cap U)$ has dimension $\leq d-1\leq 0$, it is either empty or a finite set. Hence, $\rho_0(M\cap U)$ is a closed set in $\R^m$. Therefore
$$
p\in\rho_1(M)\cap U=\rho_1(M\cap U)=\cl_{\R^m}(\rho_0(M\cap U))\setminus \rho_0(M\cap U)=\varnothing,
$$
which is a contradiction. Thus, $\dim_p(M)\geq 2$. Let $V$ be an open neighborhood of $p$ in $\R^m$ such that $\dim(M\cap V)=\dim_p(M)$ and $\dim(\rho_1(M)\cap V)=\dim_p(\rho_1(M))$. Then
\begin{multline*}
\dim_p(\rho_1(M))=\dim(\rho_1(M)\cap V)=\dim(\rho_1(M\cap V))\\
=\dim(\rho_0(\rho_0(M\cap V)))\leq\dim(M\cap V)-2=\dim_p(M)-2,
\end{multline*}
as wanted.
\end{proof}

\subsection{Zariski and maximal spectra of rings of semialgebraic functions.}\label{zar}\setcounter{paragraph}{0}
We summarize some results concerning the Zariski and maximal spectra of rings of semialgebraic and bounded semialgebraic functions on a semialgebraic set \cite[\S3-\S6]{fg3}. 

The \em Zariski spectrum \em $\Specd(M):=\Spec({\mathcal S}^{\diam}(M))$ of ${\mathcal S}^{\diam}(M)$ is the collection of all prime ideals of ${\mathcal S}^{\diam}(M)$ endowed with the Zariski topology, which has the family of sets $\Dd_{\Specd(M)}(f):=\{\gtp\in\Specd(M):\,f\not\in\gtp\}$ as a basis of open sets and where $f\in{\mathcal S}^{\diam}(M)$. We write $\Zz_{\Specd(M)}(f):=\Specd(M)\setminus\Dd_{\Specd(M)}(f)$. 

We denote the maximal ideal of all functions in ${\mathcal S}^{\diam}(M)$ vanishing at a point $p\in M$ with $\gtmd_p$. If $M$ is endowed with the Euclidean topology, the map $\phi:M\to\Specd(M),\ p\mapsto\gtmd_p$ is an embedding, so we identify $M$ with $\phi(M)$. Those maximal ideals of ${\mathcal S}^{\diam}(M)$, which are not of the form $\gtmd_p$, are called \em free \em and $\partial M:=\betaa M\setminus M$ is the set of all free maximal ideals of ${\mathcal S}^*(M)$.

\paragraph{}\label{closure}
Each semialgebraic map ${\tt h}:M_1\to M_2$ induces a homomorphism 
$$
{\tt h}^{\diam,\ast}:\mathcal{S}^{\diam}(M_2)\to\mathcal{S}^{\diam}(M_1),\ f\mapsto f\circ{\tt h}.
$$ 
The map $\Specd({\tt h}):\Specd(M_1)\to\Specd(M_2),\gtp\to({\tt h}^{\diam,\ast})^{-1}(\gtp)$ is the unique continuous extension of ${\tt h}$ to $\Specd(M_1)$. The operator $\Specd$ behaves in the expected functorial way. Let us recall some of its immediate properties \cite[4.3-6]{fg3}. Let $C,C_1,C_2,N\subset M$ be semialgebraic sets such that $C,C_1,C_2$ are closed in $M$. Then
\begin{itemize}
\item[(i)] A prime ideal $\gtp\in\Specd(M)$ belongs to $\cl_{\Specd(M)}(N)$ if and only if it contains the kernel of the restriction homomorphism $\phi:{\mathcal S}^{\diam}(M)\to{\mathcal S}^{\diam}(N),\ f\to f|_N$. If $\gtp$ is in addition a $z$-ideal, it is enough to determine if there exists $g\in\gtp$ such that $Z(g)=Y$.
\item[(ii)] $\Specd(C)\cong\cl_{\Specd(M)}(C)$ via $\Specd({\tt j})$ where ${\tt j}:C\hookrightarrow M$ is the inclusion map.
\item[(iii)] $\cl_{\Specd(M)}(C_1\cap C_2)=\cl_{\Specd(M)}(C_1)\cap\cl_{\Specd(M)}(C_2)$.
\item[(iv)] If $M_1,\ldots,M_k$ are the connected components of $M$, their closures $\cl_{\Specd(M)}(M_i)\cong\Specd(M_i)$ are the connected components of $\Specd(M)$.
\end{itemize}

Next we summarize some results obtained in \cite[\S4-5]{fg3} that will be crucial for our purposes. Let us denote the set of minimal prime ideals of ${\mathcal S}^{\diam}(M)$ with ${\rm Min}({\mathcal S}^{\diam}(M))$. 

\begin{thm}\label{mins}
Let $(M,N,Y,{\tt j},{\tt i})$ be a suitably arranged {\tt sa}-tuple and let ${\mathcal L}_Y:=\{\gtp\in\Specs(M):\ \exists\,f\in\gtp,\ Z(f)=Y\}$. Then the map $\Specs({\tt j}):\Specs(N)\to\Specs(M)\setminus{\mathcal L}_Y$ is a homeomorphism whose inverse map is $\Specs({\tt j})^{-1}:\Specs(M)\setminus{\mathcal L}_Y\to\Specs(N),\ \gtp\mapsto\gtp{\mathcal S}(N)$.
\end{thm}
\begin{remark}\label{mins-r}
Let $h\in{\mathcal S}(\cl_{\R^m}(M))$ be such that $Z(h)=\cl_{\R^m}(M)\setminus N$. Then ${\mathcal L}_Y=\Zz_{\Specs(M)}(h)$.

Indeed, the inclusion $\Zz_{\Specs(M)}(h)\subset{\mathcal L}_Y$ is clear, so we only prove the converse one. Let $\gtp\in{\mathcal L}_Y$ and $f\in\gtp$ be such that $Z(f)=Y$. By \cite[2.6.4]{bcr} there exist an integer $k\geq1$ and $g\in{\mathcal S}(\cl_{\R^m}(M))$ such that $g|_N=\frac{h^k}{f}$ and $g|_{\cl_{\R^m}(M)\setminus N}=0$. Thus, $h^k=gf\in\gtp$, so $h\in\gtp$. Therefore $\gtp\in\Zz_{\Specs(M)}(h)$, as required.
\end{remark}

\begin{thm}\label{minq}
Let $(M,N,Y,{\tt j},{\tt i})$ be a suitably arranged {\tt sa}-tuple. Let $\gtp$ be a prime ideal of ${\mathcal S}^*(M)$ and denote $Z:=\cl_{\Speca(M)}(Y)$. We have
\begin{itemize}
\item[(i)] If $\gtp$ is a minimal prime ideal of ${\mathcal S}^*(M)$, then $\gtp\not\in Z$ and $\Speca({\tt j})^{-1}(\gtp)=\{\gtq\}$ where $\gtq:=\gtp{\mathcal S}(N)\cap{\mathcal S}^*(N)$ is a minimal prime ideal of ${\mathcal S}^*(N)$. 
\item[(ii)] If $\gtp\not\in Z$ and $\gtp_0\subset\gtp$ is a minimal prime ideal of ${\mathcal S}^*(M)$, the fiber $\Speca({\tt j})^{-1}(\gtp)$ is a singleton and its unique element is $\sqrt{\gtp{\mathcal S}^*(N)+\gtp_0{\mathcal S}(N)\cap{\mathcal S}^*(N)}$. 

\item[(iii)] $\Speca({\tt j}):\Speca(N)\to\Speca(M)$ is surjective and the restriction map
$$
\Speca({\tt j})|:\Speca(N)\setminus\Speca({\tt j})^{-1}(Z)\to\Speca(M)\setminus Z
$$ 
is a homeomorphism. In particular, the restriction map
$$
\Speca({\tt j})|:{\rm Min}({\mathcal S}^*(N))\to{\rm Min}({\mathcal S}^*(M))
$$ 
is also a homeomorphism.
\item[(iv)] The homomorphism ${\mathcal S}^*(M)\hookrightarrow{\mathcal S}^*(N),\ f\mapsto f|_N$ enjoys the going up property.
 \end{itemize}
\end{thm}

\paragraph{}\label{inequality}
It is well-known that ${\mathcal S}(M)={\mathcal S}^*(M)_{{\mathcal W}_M}$ where ${\mathcal W}_M$ is the multiplicative set of those functions $f\in{\mathcal S}^*(M)$ such that $Z(f)=\varnothing$ because each $f\in{\mathcal S}(M)$ can be written as $f=\frac{\frac{f}{1+|f|}}{\frac{1}{1+|f|}}$. Denote the set of prime ideals of ${\mathcal S}^*(M)$ that do not intersect ${\mathcal W}_M$ with ${\mathfrak S}(M)$. The Zariski spectrum of ${\mathcal S}(M)$ is homeomorphic to ${\mathfrak S}(M)$ via the homeomorphisms ${\mathfrak i}_M:\,\Specs(M)\to{\mathfrak S}(M),\ \gtp\mapsto\gtp\cap{\mathcal S}^*(M)$ and ${\mathfrak i}_M^{-1}:\,{\mathfrak S}(M)\to\Specs(M),\ \gtq\mapsto\gtq{\mathcal S}(M)$. The previous homeomorphism ${\mathfrak i}_M$ maps ${\rm Min}({\mathcal S}(M))$ (bijectively) onto ${\rm Min}({\mathcal S}^*(M))$ (see \cite[4.3]{fe1}).

\paragraph{}\label{zideal3}
An ideal $\gta$ of ${\mathcal S}(M)$ is a \em $z$-ideal \em if every $g\in{\mathcal S}(M)$ satisfying $Z(f)\subset Z(g)$ for some $f\in\gta$ belongs to $\gta$. Each $z$-ideal is a radical ideal because $Z(f)=Z(f^k)$ for each $f\in{\mathcal S}(M)$ and each $k\geq 1$. The operator $\Specs$ preserves prime $z$-ideals: \em if ${\tt h}:M_1\to M_2$ is a semialgebraic map, $\Specs({\tt h})(\gtp)$ is a prime $z$-ideal of ${\mathcal S}(M_2)$ for each prime $z$-ideal $\gtp$ of ${\mathcal S}(M_1)$\em. 

Two relevant examples of $z$-ideals of ${\mathcal S}(M)$ are maximal and minimal prime ideals \cite[4.7, 4.14]{fe1}. Minimal prime ideals have been characterized geometrically in \cite[4.1]{fe1} as follows.

\begin{thm}[Minimal prime ideals]\label{minimalnlc}
Let $\gtp$ be a prime ideal of ${\mathcal S}^{\diam}(M)$. Then $\gtp$ is a minimal prime ideal of ${\mathcal S}^{\diam}(M)$ if and only if the zero set of each $f\in\gtp$ has a non-empty interior in $M$. 
\end{thm}

\paragraph{Maximal spectra.}\label{spectracomp}
Denote the collection of all maximal ideals of ${\mathcal S}^{\diam}(M)$ with $\betad M$ and consider in $\betad M$ the topology induced by the Zariski topology of $\Specd(M)$. Given $f\in{\mathcal S}^{\diam}(M)$, we denote
$$
{\mathcal D}_{\betad M}(f):=\Dd_{\Specd(M)}(f)\cap\betad M\quad\text{and}\quad{\mathcal Z}_{\betad M}(f):=\betad M\setminus{\mathcal D}_{\betad M}(f)=\Zz_{\Specd(M)}(f)\cap\betad M.
$$
As for rings of continuous functions \cite[\S7]{gj}, the maximal spectra $\betas M$ and $\betaa M$ of ${\mathcal S}(M)$ and ${\mathcal S}^*(M)$ are homeomorphic (\cite[\S10]{t1}, \cite[3.5]{fg5}). The map $\Phi:\betas M\to \betaa M,\ \gtm\mapsto \gtm^*$, which maps each maximal ideal $\gtm$ of ${\mathcal S}(M)$ to the unique maximal ideal $\gtm^*$ of ${\mathcal S}^*(M)$ that contains $\gtm \cap{\mathcal S}^*(M)$, is a homeomorphism. In particular, $\Phi(\gtm_p)=\gtm_p^*$ for all $p\in M$. We denote the maximal ideals of ${\mathcal S}^*(M)$ with $\gtm^*$ and the unique maximal ideal $\gtn$ of ${\mathcal S}(M)$ such that $\gtn\cap{\mathcal S}^*(M)\subset \gtm^*$ with $\gtm$.

If ${\tt h}:M_1\to M_2$ is a semialgebraic map, $\Speca({\tt h}):\Speca(M_1)\to\Speca(M_2)$ maps $\betaa M_1$ into $\betaa M_2$ by \cite[5.9]{fg3}. We denote the restriction of $\Speca({\tt h})$ to $\betaa M_1$ with $\betaa {\tt h}:\betaa M_1\to\betaa M_2$.

\paragraph{}\label{chainp}
It is well-known that \em the set of prime ideals of ${\mathcal S}^{\diamond}(M)$ containing a prime ideal $\gtp$ form a chain\em. In \cite[2.11 \& 5.1-2]{fe1} we study the behavior of those chains of prime ideals in ${\mathcal S}^*(M)$ that do not admit a refinement. \em Let $\gtp_0\subsetneq\cdots\subsetneq\gtp_r=\gtm^*$ be a non-refinable chain of prime ideals in the ring ${\mathcal S}^*(M)$. We have:
\begin{itemize}
\item[(i)] There exists $0\leq k\leq r$ such that $\gtp_k=\gtm\cap{\mathcal S}^*(M)$ where $\gtm$ is the unique maximal ideal of ${\mathcal S}(M)$ such that $\gtm\cap{\mathcal S}^*(M)\subset\gtm^*$. In particular, $\gtp_\ell\cap{\mathcal W}_M=\varnothing$ if and only if $\ell\leq k$.
\item[(ii)] The subchain $\gtp_k=\gtm\cap{\mathcal S}^*(M)\subsetneq\cdots\subsetneq\gtp_r=\gtm^*$ is the same for every non-refinable chain of prime ideals in ${\mathcal S}^*(M)$ ending at $\gtm^*$. 
\item[(iii)] If $C$ is a closed semialgebraic subset of $M$ and $\gtp_j\in\cl_{\Speca(M)}(C)$ for $j=0,\ldots,r$, then the maximal ideal $\gtm\in\cl_{\Specs(M)}(C)$ and $\gtp_k=\gtm\cap{\mathcal S}^*(M)\in\cl_{\Speca(M)}(C)$.
\end{itemize}
\em

\paragraph{Semialgebraic depth}\label{sd}
The \em semialgebraic depth \em of a prime ideal $\gtp$ of ${\mathcal S}(M)$ is 
$$
{\tt d}_M(\gtp):=\min\{\dim(Z(f)):\,f\in\gtp\}. 
$$
Some remarkable properties of this invariant collected in \cite{fg2} and \cite[\S4]{fe1} are the following: 
\begin{itemize}
\em\item[(i)] Let $\gtp$ be a prime ideal of ${\mathcal S}(M)$. Then there exists a unique prime $z$-ideal $\gtp^z$ of ${\mathcal S}(M)$ such that $\gtp\subset\gtp^z$ and ${\tt d}_M(\gtp)={\tt d}_M(\gtp^z)$.
\item[(ii)] Let $\gtp,\gtq$ be two prime $z$-ideals of ${\mathcal S}(M)$ such that $\gtq\subsetneq\gtp$. Then ${\tt d}_M(\gtp)<{\tt d}_M(\gtq)$. If additionally ${\tt d}_M(\gtp)={\tt d}_M(\gtq)+1$, there exists no prime ideal between $\gtp$ and $\gtq$.
\item[(iii)] If $\gtp$ is a prime $z$-ideal, then ${\tt d}_M(\gtp)=\tr\deg_\R(\qf({\mathcal S}(M)/\gtp))$.
\end{itemize}

\subsection{Semialgebraic compactifications of a semialgebraic set}\label{mpsc}\setcounter{paragraph}{0}

A \em semialgebraic compactification of a semialgebraic set $M\subset\R^m$ \em is a pair $(X,{\tt k})$ constituted of a compact semialgebraic set $X\subset\R^n$ and a semialgebraic embedding ${\tt k}:M\hookrightarrow X$ whose image is dense in $X$. Of course, it holds ${\mathcal S}(X)={\mathcal S}^*(X)$. The following properties shown in \cite[\S1]{fg2} are decisive.

\paragraph{}\label{a}$\hspace{-1.5mm}$
{\em
For each finite family $\Ff:=\{f_1,\ldots,f_r\}\subset{\mathcal S}^*(M)$ there exist a semialgebraic compacti\-fication $(X,{\tt k}_{\Ff})$ of $M$ and semialgebraic functions $F_1,\ldots,F_r\in{\mathcal S}(X)$ such that $f_i=F_i\circ{\tt k}_{\Ff}$. 
\em}

\vspace{2mm}
Indeed, we may assume that $M$ is bounded. Now consider $X:=\cl({\rm graph}(f_1,\ldots,f_r))$, ${\tt k}_\Ff:M\hookrightarrow X,\ x\mapsto(x,f_1(x),\ldots,f_r(x))$ and $F_i:=\pi_{m+i}|_{X}$ where $\pi_{m+i}:\R^{m+r}\to\R,\ x:=(x_1,\ldots,x_{m+r})\mapsto x_{m+i}$ for $i=1,\ldots,r$. 

\paragraph{}\label{longitud}$\hspace{-1.5mm}$
\em
Given a chain of prime ideals $\gtp_0\subsetneq\cdots\subsetneq \gtp_r$ of ${\mathcal S}^*(M)$, there exists a semialgebraic compactification $(X,{\tt k})$ of $M$ such that the prime ideals $\gtq_i:=\gtp_i\cap{\mathcal S}(X)$ constitute a chain $\gtq_0\subsetneq\cdots\subsetneq\gtq_r$ in ${\mathcal S}(X)$. 
\em

\vspace{2mm}
Indeed, it is enough to pick $f_i\in\gtp_i\setminus \gtp_{i-1}$ for $1\leq i \leq r$ and to consider the semialgebraic compactification of $M$ provided for the family $\Ff:=\{f_1,\ldots,f_r\}$ by \ref{a}.

\paragraph{}\label{remainder}
Let ${\mathfrak F}_M$ be the collection of all semialgebraic compactifications of $M$. Given two of them $(X_1,{\tt k}_1)$ and $(X_2,{\tt k}_2)$, we say that $(X_1,{\tt k}_1)\preccurlyeq(X_2,{\rm j}_2)$ if there exists a (unique) continuous (surjective) map $\rho:=\rho_{X_1,X_2}:X_2\to X_1$ such that $\rho\circ{\tt k}_2={\tt k}_1$; the uniqueness of $\rho$ follows because $\rho|_{{\tt k}_2(M)}={\tt k}_1\circ({\tt k}_2|_M)^{-1}$ and ${\tt k}_2(M)$ is dense in $X_2$. It holds: \em $\rho^{-1}(X_1\setminus {\tt k}_1(M))=X_2\setminus {\tt k}_2(M)$ and $\rho(X_2\setminus {\tt k}_2(M))=X_1\setminus {\tt k}_1(M)$\em.
\begin{proof}[\sl Proof]
Let us see $X_2\setminus {\tt k}_2(M)\subset\rho^{-1}(X_1\setminus {\tt k}_1(M))$ first. Let $x_2\in X_2\setminus{\tt k}_2(M)$. Since ${\tt k}_2(M)$ is dense in $X_2$, by the curve selection lemma \cite[2.5.5]{bcr} there exists a semialgebraic path $\alpha:[0,1]\to\R^m$ such that $\alpha((0,1])\subset M$ and ${\tt k}_2(\alpha(0))=x_2$. Note that $\rho(x_2)=\rho({\tt k}_2(\alpha(0)))=\lim_{\t\to0^+}{\tt k}_1(\alpha(\t))$. If this point occurs in ${\tt k}_1(M)$, then $\alpha(0)\in M$, so $x_2={\tt k}_2(\alpha(0))\in {\tt k}_2(M)$, which is a contradiction. Conversely, suppose there exists $x_2\in\rho^{-1}(X_1\setminus {\tt k}_1(M))\cap {\tt k}_2(M)$. Then $\rho(x_2)\not\in {\tt k}_1(M)$, but $x_2={\tt k}_2(y)$ for some $y\in M$. This implies $\rho(x_2)=\rho({\tt k}_2(y))={\tt k}_1(y)\in {\tt k}_1(M)$, which is a contradiction. Finally, since $\rho$ is surjective and $\rho^{-1}(X_1\setminus {\tt k}_1(M))=X_2\setminus {\tt k}_2(M)$, we conclude $\rho(X_2\setminus {\tt k}_2(M))=X_1\setminus {\tt k}_1(M)$.
\end{proof} 
 
\paragraph{}\label{directed}$\hspace{-1.5mm}$
\em $({\mathfrak F}_M,\preccurlyeq)$ is an up-directed set \em and we have a collection of rings $\{{\mathcal S}(X)\}_{(X,{\tt k})\in{\mathfrak F}_M}$ and $\R$-mono\-mor\-phisms $\rho_{X_1,X_2}^*:{\mathcal S}(X_1)\to{\mathcal S}(X_2),\ f\mapsto f\circ\rho_{X_1,X_2}$ for $(X_1,{\tt k}_1)\preccurlyeq(X_2,{\tt k}_2)$ such that 
\begin{itemize}
\item $\rho_{X_1,X_1}^*=\id$ and
\item $\rho_{X_1,X_3}^*=\rho_{X_2,X_3}^*\circ\rho_{X_1,X_2}^*$ if $(X_1,{\tt k}_1)\preccurlyeq(X_2,{\tt k}_2)\preccurlyeq(X_3,{\tt k}_3)$.
\end{itemize}

We conclude: \em The ring ${\mathcal S}^*(M)$ is the direct limit of the up-directed system $\qq{{\mathcal S}(X),\rho_{X_1,X_2}^*}$ together with the homomorphisms ${\tt k}^*:{\mathcal S}(X)\hookrightarrow{\mathcal S}^*(M)$ where $(X,{\tt k})\in{\mathfrak F}_M$\em. 

\paragraph{}\label{brimming}$\hspace{-1.5mm}$ \em Let $\gtp$ be a prime ideal of ${\mathcal S}^{\diam}(M)$. Then there exists \em by \cite[\S2]{fg2} \em a semialgebraic compactification $(X,{\tt k})$ of $M$ such that 
$$
\qf({\mathcal S}(X)/(\gtp\cap{\mathcal S}(X)))=\qf({\mathcal S}^{\diam}(M)/\gtp).
$$
\em We refer to $(X,{\tt k})$ as a \em brimming semialgebraic compactification of $M$ for $\gtp$\em. Of course, if $\gtp_1,\ldots,\gtp_r$ are finitely many prime ideals of ${\mathcal S}^{\diam}(M)$, there exists by \ref{directed} a (common) brimming semialgebraic compactification $(X,{\tt k})$ of $M$ for $\gtp_1,\ldots,\gtp_r$, that is, $\qf({\mathcal S}^{\diam}(M)/\gtp_i)=\qf({\mathcal S}(X)/(\gtp_i\cap{\mathcal S}(X)))$ for $i=1,\ldots,r$.

\subsection{Separation of prime $z$-ideals}\label{s2e} 
We finish this section showing how the prime $z$-ideals of ${\mathcal S}(M)$ admit a nice behavior with respect to `separation'.

\begin{lem}\label{intcomp0}
Let $\gtp_1,\gtp_2$ be prime $z$-ideals of ${\mathcal S}(M)$ such that $\gtp_i\not\subset\gtp_j$ if $i\neq j$ and let $g\in\gtp_1\cap\gtp_2$. Then there exist $f_i\in\gtp_i\setminus\gtp_j$ for $i\neq j$ such that
\begin{itemize} 
\item[(i)] $Z(f_i)$ is pure dimensional, $Z(f_i)\subset Z(g)$ and $\dim(Z(f_i))={\tt d}_M(\gtp_i)$.
\item[(ii)] $\dim(Z(f_1^2+f_2^2))<\min\{{\tt d}_M(\gtp_1),{\tt d}_M(\gtp_2)\}$.
\end{itemize}
\end{lem}
\begin{proof}[\sl Proof]
Let $g_i\in\gtp_i\setminus\gtp_j$ for $i\neq j$. We may assume $\dim(Z(g_i))={\tt d}_M(\gtp_i)$ and $Z(g_i)\subset Z(g)$ by substituting $g_i$ with $g_i^2+a_i^2+g^2$ where $a_i\in\gtp_i$ and $\dim(Z(a_i))={\tt d}_M(\gtp_i)$.

Let $f_i\in{\mathcal S}(M)$ be such that $Z(f_i)=\cl_M(Z(g_i)\setminus Z(g_j))$ if $i\neq j$. As $Z(g_i)\subset Z(f_ig_j)$ and $\gtp_i$ is a prime $z$-ideal, $f_ig_j\in\gtp_i$ and since $g_j\in\gtp_j\setminus\gtp_i$, we deduce $f_i\in\gtp_i$. Notice
\begin{multline*}
Z(f_1)\cap Z(f_2)\subset Z(f_1)\cap Z(g_2)=\cl_M(Z(g_1)\setminus Z(g_2))\cap Z(g_2)\\
=(\cl_M(Z(g_1)\setminus Z(g_2))\setminus(Z(g_1)\setminus Z(g_2)))\cap Z(g_2);
\end{multline*}
hence, 
\begin{multline*}
\dim(Z(f_1)\cap Z(f_2))\leq\dim(\cl_M(Z(g_1)\setminus Z(g_2))\setminus(Z(g_1)\setminus Z(g_2)))\\
<\dim(Z(g_1)\setminus Z(g_2))\leq\dim(Z(g_1))={\tt d}_M(\gtp_1).
\end{multline*}
Analogously, $\dim(Z(f_1)\cap Z(f_2))<{\tt d}_M(\gtp_2)$, so 
$$
\dim(Z(f_1^2+f_2^2))<\min\{{\tt d}_M(\gtp_1),{\tt d}_M(\gtp_2)\}.
$$
Notice in addition that as $Z(f_i)\subset Z(g_i)$,
$$
{\tt d}_M(\gtp_i)\leq\dim(Z(f_i))\leq\dim(Z(g_i))={\tt d}_M(\gtp_i).
$$
To finish we may assume that $Z(f_i)$ is pure dimensional. To that end use the decomposition of $Z(f_i)$ as a union of (closed) bricks, the fact that each brick is the zero set of a semialgebraic function on $M$ and that $\gtp_i$ is a prime $z$-ideal.
\end{proof}

\begin{lem}\label{intcomp1}
Let $\gtp_1,\gtp_2,\gtq$ be prime $z$-ideals of ${\mathcal S}(M)$ such that $\gtp_i\subset\gtq$ and $\gtp_i\not\subset\gtp_j$ if $i\neq j$. Assume ${\tt d}_M(\gtp_i)={\tt d}_M(\gtq)+1$ and let $g\in\gtq$ be such that $\dim(Z(g))={\tt d}_M(\gtq)$. Then there exist $f_i\in\gtp_i\setminus\gtp_j$ if $i\neq j$ such that $Z(f_1^2+f_2^2)\subset Z(g)$ and $Z(f_i)$ is pure dimensional.
\end{lem}
\begin{proof}[\sl Proof]
The pure dimensionality of $Z(f_i)$ will be reached once the other requirements are fulfilled. To achieve it one proceeds similarly to the final part of the proof of the previous lemma. By Lemma \ref{intcomp0} there exist $h_i\in\gtp_i\setminus\gtp_j$ for $i\neq j$ such that $\dim(Z(h_i))={\tt d}_M(\gtp_i)$ and 
$$
\dim(Z(h_1^2+h_2^2))<\min\{{\tt d}_M(\gtp_1),{\tt d}_M(\gtp_2)\}={\tt d}_M(\gtq)+1.
$$ 
Thus, if $Z_i:=Z(h_i)$, we have ${\tt d}_M(\gtq)\leq\dim(Z_1\cap Z_2)\leq{\tt d}_M(\gtq)=\dim(Z(g))$. After substituting $g$ with $g^2+h_1^2+h_2^2$, we may assume $Z(g)\subset Z_1\cap Z_2$.

If $Z_1\cap Z_2\subset Z(g)$, it is enough to choose $f_i:=h_i$. Assume next $Z_1\cap Z_2\not\subset Z(g)$ and let $C_1:=Z(g)$ and $C_2:=\cl_M((Z_1\cap Z_2)\setminus C_1)$. Let $b\in{\mathcal S}(M\setminus(C_1\cap C_2))$ be such that $b^{-1}(\{-1\})=C_1\setminus(C_1\cap C_2)$ and $b^{-1}(\{1\})=C_2\setminus(C_1\cap C_2)$. Consider the closed semialgebraic subsets of $M$ 
$$
T_1:=\cl_M(b^{-1}((-\infty,0]))\quad\text{ and }\quad T_2:=b^{-1}([0,+\infty))\cup(C_1\cap C_2). 
$$
Let $b_i\in{\mathcal S}(M)$ be such that $Z(b_i)=T_i$. Note that
$$
C_1\cap C_2=C_1\cap\cl_M((Z_1\cap Z_2)\setminus C_1)=C_1\cap(\cl_M((Z_1\cap Z_2)\setminus C_1)\setminus((Z_1\cap Z_2)\setminus C_1)).
$$
Therefore 
\begin{multline}\label{ultr}
\dim(C_1\cap C_2)\leq\dim(\cl_M((Z_1\cap Z_2)\setminus C_1)\setminus((Z_1\cap Z_2)\setminus C_1))\\
<\dim((Z_1\cap Z_2)\setminus C_1)\leq\dim(Z_1\cap Z_2)={\tt d}_M(\gtq).
\end{multline}
Moreover, as 
$$
b^{-1}([0,+\infty))\cap Z(g)\subset b^{-1}([0,+\infty))\cap(b^{-1}(\{-1\})\cup(C_1\cap C_2))\subset C_1\cap C_2,
$$
we conclude
$$
Z(b_2^2+g^2)=T_2\cap Z(g)=(b^{-1}([0,+\infty))\cup(C_1\cap C_2))\cap Z(g)\subset C_1\cap C_2.
$$
Thus, $b_2^2+g^2\not\in\gtq$ because by \eqref{ultr} $\dim(Z(b_2^2+g^2))<{\tt d}_M(\gtq)$. Since $g\in\gtq$, we get $b_2\not\in\gtq$. As $b_1b_2=0$ and $\gtp_i\subset\gtq$, we deduce $b_1\in\gtp_i$. Therefore, $f_i:=h_i^2+b_1^2\in\gtp_i$ and since $b_1\in\gtp_i$ and $h_i\in\gtp_i\setminus\gtp_j$ if $i\neq j$, we have $f_i\in\gtp_i\setminus\gtp_j$ if $i\neq j$. To finish we show $Z(f_1^2+f_2^2)\subset Z(g)$.

Indeed, notice that
$$
(Z_1\cap Z_2)\setminus(C_1\cap C_2)\subset(C_1\cup C_2)\setminus(C_1\cap C_2)\subset b^{-1}(\{-1\})\cup b^{-1}(\{1\});
$$
hence,
$$
((Z_1\cap Z_2)\setminus(C_1\cap C_2))\cap(b^{-1}((-\infty,0])\subset b^{-1}(\{-1\})=C_1\setminus(C_1\cap C_2).
$$
Therefore
\begin{multline*}
Z(f_1^2+f_2^2)=Z(h_1)\cap Z(h_2)\cap Z(b_1)=Z_1\cap Z_2\cap\cl_M(b^{-1}((-\infty,0]))\\
\subset Z_1\cap Z_2\cap(b^{-1}((-\infty,0])\cup(C_1\cap C_2))\subset C_1=Z(g),
\end{multline*}
as required. 
\end{proof}

\begin{remark}
The previous lemma applies for instance if $\gtp_1,\gtp_2$ are minimal prime ideals contained in a prime $z$-ideal $\gtq$.
\end{remark}

\section{Examples}\label{s3}
\renewcommand{\theparagraph}{\thesection.\arabic{paragraph}}
\setcounter{secnumdepth}{4}

Before providing in the next section the rather technical proofs of Theorem \ref{fiberspec} and Lemmas \ref{pmchain} and \ref{fiberspec3} we present here some enlightening examples to illustrate them. We develop them in full detail for the sake of the reader and we explain the irregularities that present the size of the fibers of the spectral maps associated to some apparently innocuous geometric embeddings. Fix a positive integer $m$ and denote $X_m:=[0,1]^m$.\setcounter{paragraph}{0}

\begin{examples}\label{excrucial}
(i) Let $\gtm_0$ be the maximal ideal constituted by all semialgebraic functions on $X_m$ vanishing at the origin. We construct: \em a prime ideal $\gtp\subset\gtm_0$ of ${\mathcal S}(X_m)$ such that ${\tt d}_{X_m}(\gtq)=m$\em.

Define $\gtq$ as the set of all semialgebraic functions $f\in{\mathcal S}(X_m)$ satisfying: \em for each semialgebraic triangulation $(K,\Phi)$ of $X_m$ compatible with $Z(f)$ it holds $\Phi(\sigma)\subset Z(f)$ where:\em

\paragraph{}\label{propcrucial0}$\hspace{-1.5mm}$
\em $\sigma\in K$ is an $m$-dimensional simplex such that for each $d=0,\ldots,m$ there exists a $d$-dimensional face $\tau_d$ of $\sigma$ such that $\Phi(\tau_d)\subset\{\x_{d+1}=0,\ldots,\x_m=0\}$\em. 

\vspace{2mm}
Using recursively the straightforward property \ref{propcrucial} stated below, one shows that $\sigma$ is uniquely determined by \ref{propcrucial0}. We call $\sigma$ the \em indicator simplex for $(K,\Phi)$\em. 

\paragraph{}\label{propcrucial}$\hspace{-1.5mm}$
\em Let $\tau\subset\R^d$ be a simplex of dimension $d$ and $\eta_1,\eta_2$ two simplices contained in $\R^d\times[0,\infty)$ that have $\tau$ as a common face. Then $\eta_1^0\cap\eta_2^0\neq\varnothing$\em.

\paragraph{}\label{propcrucial2}$\hspace{-1.5mm}$
It holds: \em $\gtq$ is a prime ideal of ${\mathcal S}(X_m)$ and, as $\dim(\sigma)=m$, it is clear that ${\tt d}_{X_m}(\gtq)=m$\em.

\vspace{2mm}
Only the primality of $\gtq$ requires a comment. Indeed, let $f_1,f_2\in{\mathcal S}(X_m)$ be such that $f_1f_2\in\gtq$ and let $(K,\Phi)$ be a semialgebraic triangulation of $X_m$ compatible with $Z(f_1)$ and $Z(f_2)$. Let $\sigma$ be an indicator simplex for $(K,\Phi)$. Since $\Phi(\sigma)\subset Z(f_1f_2)$ and $(K,\Phi)$ is compatible with $Z(f_i)$, we may assume $\Phi(\sigma^0)\subset Z(f_1)$; hence, $\Phi(\sigma)\subset Z(f_1)$. Thus, $f_1\in\gtq$ and we conclude that $\gtq$ is a prime ideal. 

(ii) We claim: \em There is a chain of prime ideals $\gtq_0\subsetneq\cdots\subsetneq\gtq_m:=\gtm_0$ in ${\mathcal S}(X_m)$ such that ${\tt d}_{X_m}(\gtq_k)=m-k$ for $k=0,\ldots,m$\em.

For each $k=1,\ldots,m$ define $X_k:=[0,1]^k\times\{0\}\subset\R^m$. Clearly, $\{0\}\subsetneq X_1\subsetneq\cdots\subsetneq X_m$ is a chain of closed subsets of $X_m$. The restriction homomorphism $\varphi_k:{\mathcal S}(X_m)\to{\mathcal S}(X_k),\ f\mapsto f|_{X_k}$ is by \cite{dk} surjective, so the ideal $\gtq_k$ constructed in (i) for $X_k$ provides a prime ideal $\gtq_{m-k}:=\varphi_k^{-1}(\gtq_k)$ such that ${\tt d}_{X_m}(\gtq_{m-k})={\tt d}_{X_k}(\gtq_k)=k$. Now, by the definition of the ideals $\gtq_k$, it is clear that $\gtq_0\subsetneq\cdots\subsetneq\gtq_m:=\gtm_0$.

(iii) We present next an inclusion of semialgebraic sets ${\tt j}:N\hookrightarrow M$ and a non-definable chain of prime ideals $\gtp_0\subsetneq\cdots\subsetneq\gtp_m$ in ${\mathcal S}(M)$ such that \em each prime ideal $\gtp_k\in\cl_{\Speca(M)}(Y)$ where $Y:=M\setminus N$ and each fiber $\Speca({\tt j})^{-1}(\gtp_k)$ is a singleton \em (see Theorem \ref{fiberspec}).

Let $M:=X_m\setminus\{\x_{m-1}=0,\x_m=0\}$ and consider the inclusion ${\tt k}:M\hookrightarrow X_m$. The map $\Speca({\tt k}):\Speca(M)\to\Speca(X_m)$ is surjective and by Theorem \ref{minq}(iv) there exists a chain of prime ideals $\gtp_0\subsetneq\cdots\subsetneq\gtp_m$ in $\Speca(M)$ such that $\Speca({\tt k})(\gtp_k)=\gtq_k$ for $k=0,\ldots,m$. By \cite[Thm. 1]{fg2} we know that $\dim({\mathcal S}(M))=\dim({\mathcal S}^*(M))=\dim(M)=m$. Thus, the chain of prime ideals $\gtp_0\subsetneq\cdots\subsetneq\gtp_m$ has maximal length and does not admit any refinement; hence, $\gtp_m=\gtm^*$ is a maximal ideal. Notice that for each $k\geq2$ there exists $f_k\in\gtp_k$ such that $Z(f_k)\cap M=\varnothing$. In addition the zero set of each $f\in\gtp_1$ intersects $M$. Thus, $\gtp_k\cap{\mathcal W}_M=\varnothing$ if and only if $k=0,1$.

By \ref{chainp} we conclude $\gtp_1=\gtm\cap{\mathcal S}^*(M)$ where $\gtm$ is the unique maximal ideal of ${\mathcal S}(M)$ such that $\gtm\cap{\mathcal S}^*(M)\subset\gtm^*$. Notice ${\tt d}_M(\gtm)=m-1$ and that $\gtq_0$ is by Theorem \ref{minimalnlc} and property \ref{propcrucial} the unique minimal prime ideal of ${\mathcal S}(X_m)$ contained in $\gtq_1$. By Theorem \ref{minq}(iii) $\gtp_0$ is the unique minimal prime ideal contained in $\gtp_1$. 

\paragraph{}Let $N:=M\setminus\{\x_m=0\}$ and $Y:=M\setminus N=M\cap\{\x_m=0\}$. Denote the inclusion with ${\tt j}:N\hookrightarrow M$ and observe $\gtp_1=\gtm\cap{\mathcal S}^*(M)\in\cl_{\Speca(M)}(Y)$, so also $\gtp_k\in\cl_{\Speca(M)}(Y)$ for $k=2,\ldots,m$. Moreover, following the notation of \eqref{gtQ}, \em it holds $\widehat{\gtp}_k=\gtp_1$ for $k=1,\ldots,m$, $\gtp_1{\mathcal S}(M)=\gtm$ and ${\tt d}_M(\gtm)=m-1$\em. By Theorem \ref{fiberspec} \em the fiber $\Speca({\tt j})^{-1}(\gtp_k)$ is a singleton $\{\gtp_k'\}$ for $k=1,\ldots,m$\em. 

(iv) \em If $m\geq2$, one can find infinitely maximal ideals with singleton fibers with respect to $\Speca({\tt j})$ contained in $\cl_{\Speca(M)}(Y)$\em. 

To that end, fix $k\geq 1$ and $0\leq\ell< k$. Consider the inclusions 
$$
{\tt h}_{k,\ell}:X_m\to X_m,\ x:=(x_1,\ldots,x_m)\mapsto\tfrac{1}{k}(x_1,\ldots,x_m)+(\tfrac{\ell}{k},0,\ldots,0).
$$
Denote $Z_{\ell,k}:=\im({\tt h}_{k,\ell})$, which is a closed semialgebraic subset of $X_m$ of dimension $m$. One can check readily that for each $0\leq\ell< k$ the prime ideal $\gtm_{k,\ell}^*:=\Speca({\tt h}_{k,\ell})(\gtm^*)$ has a singleton fiber with respect to $\Speca({\tt j})$ and that $\gtm_{k,\ell}\neq\gtm_{k,\ell'}$ if $\ell\neq\ell'$. Thus, there exist infinitely maximal ideals with singleton fibers contained in $\cl_{\Speca(M)}(Y)$.

\begin{figure}[ht]
\centering
\begin{tikzpicture}
%\draw[thin,color=black,step=.5cm,dashed] (0,0) grid (13,6);

\draw (0,0) -- (5,0) -- (5,5) -- (0,5) -- (0,0);
\draw[fill=black!30!white] (0,0) -- (5,0) -- (5,5) -- (0,5) -- (0,0);

\draw (0,0) arc (270:360:2.5cm) -- (2.5,0) -- (0,0);
\draw[fill=black!45!white] (0,0) arc (270:360:2.5cm) -- (2.5,0) -- (0,0); 

\draw[->,thick] (5.5,3.3) -- (7.5,3.3);

\draw (6.5,3.6) node{${\tt h}_{5,4}$};

\draw[->,thick] (5.5,2.6) -- (7.5,2.6);

\draw (6.5,2.9) node{${\tt h}_{5,3}$};

\draw[->,thick] (5.5,1.9) -- (7.5,1.9);

\draw (6.5,2.1) node{${\tt h}_{5,2}$};

\draw[->,thick] (5.5,1.2) -- (7.5,1.2);

\draw (6.5,1.5) node{${\tt h}_{5,1}$};

\draw[->,thick] (5.5,0.5) -- (7.5,0.5);

\draw (6.5,0.8) node{${\tt h}_{5,0}$};

\draw (8,0) -- (13,0) -- (13,5) -- (8,5) -- (8,0);
\draw[fill=black!15!white] (8,0) -- (13,0) -- (13,5) -- (8,5) -- (8,0);

\draw (8,0) -- (9,0) -- (9,1) -- (8,1) -- (8,0);
\draw[fill=black!30!white] (8,0) -- (9,0) -- (9,1) -- (8,1) -- (8,0);
\draw (8,0) arc (270:360:0.5cm) -- (8.5,0) -- (8,0);
\draw[fill=black!45!white] (8,0) arc (270:360:0.5cm) -- (8.5,0) -- (8,0);

\draw (9,0) -- (10,0) -- (10,1) -- (9,1) -- (9,0);
\draw[fill=black!30!white] (9,0) -- (10,0) -- (10,1) -- (9,1) -- (9,0);
\draw (9,0) arc (270:360:0.5cm) -- (9.5,0) -- (9,0);
\draw[fill=black!45!white] (9,0) arc (270:360:0.5cm) -- (9.5,0) -- (9,0);

\draw (10,0) -- (11,0) -- (11,1) -- (10,1) -- (10,0);
\draw[fill=black!30!white] (10,0) -- (11,0) -- (11,1) -- (10,1) -- (10,0);
\draw (10,0) arc (270:360:0.5cm) -- (10.5,0) -- (10,0);
\draw[fill=black!45!white] (10,0) arc (270:360:0.5cm) -- (10.5,0) -- (10,0); 

\draw (11,0) -- (12,0) -- (12,1) -- (11,1) -- (11,0);
\draw[fill=black!30!white] (11,0) -- (12,0) -- (12,1) -- (11,1) -- (11,0);
\draw (11,0) arc (270:360:0.5cm) -- (11.5,0) -- (11,0);
\draw[fill=black!45!white] (11,0) arc (270:360:0.5cm) -- (11.5,0) -- (11,0);

\draw (12,0) -- (13,0) -- (13,1) -- (12,1) -- (12,0);
\draw[fill=black!30!white] (12,0) -- (13,0) -- (13,1) -- (12,1) -- (12,0);
\draw (12,0) arc (270:360:0.5cm) -- (12.5,0) -- (12,0);
\draw[fill=black!45!white] (12,0) arc (270:360:0.5cm) -- (12.5,0) -- (12,0);

\draw (4.5,4.5) node{$X_2$};
\draw (12.5,4.5) node{$X_2$};
\draw (1.5,1.5) node{$\gtm$};
\draw (8.5,0.75) node{$\gtm_{5,0}$};
\draw (9.5,0.75) node{$\gtm_{5,1}$};
\draw (10.5,0.75) node{$\gtm_{5,2}$};
\draw (11.5,0.75) node{$\gtm_{5,3}$};
\draw (12.5,0.75) node{$\gtm_{5,4}$};

\end{tikzpicture}
\caption{Construction of the maximal ideals $\gtm_{k,\ell}$ for $m=2$ (and $k=5$).}
\end{figure}
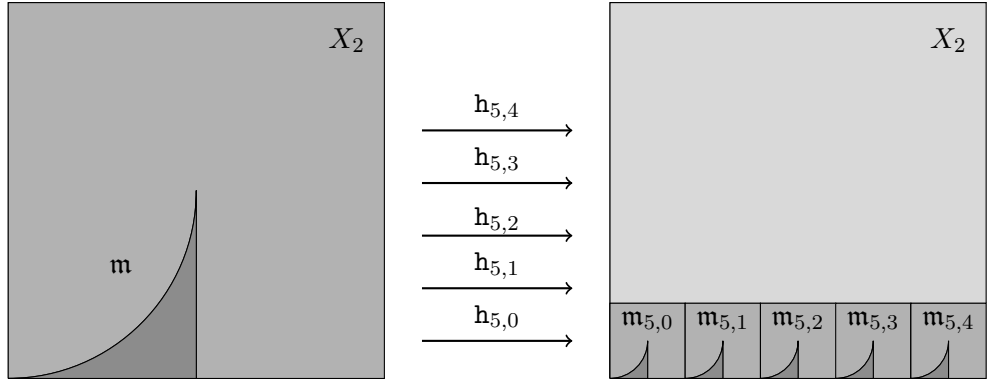

(v) Fix $1\leq\ell\leq m-2$. We construct next an inclusion ${\tt j}':N\hookrightarrow X_m$ and a non-refinable chain or prime ideals \em $\gtq_1'\subsetneq\cdots\subsetneq\gtq_{m-\ell}$ in ${\mathcal S}(X_m)$ such that the fiber $\Speca({\tt j}')^{-1}(\gtq_1')$ is a singleton while $\Speca({\tt j}')^{-1}(\gtq_k')$ is an infinite set for $k\geq2$ \em (see Theorem \ref{fiberspec} and Lemmas \ref{pmchain} and \ref{fiberspec3}).

Write $y_{(\ell)}:=(x_1,\ldots,x_{\ell-1})$ and $z_{(\ell)}:=(x_{\ell+1},\ldots,x_m)$ for $1\leq\ell\leq m-2$. Consider the semialgebraic map $\psi_\ell:X_m\to X_m$ given by
$$
(x_1,\ldots,x_m)\mapsto 
\begin{cases}
(y_{(\ell)},x_\ell(x_{m-1}+x_m),z_{(\ell)})&\text{ if $0\leq x_\ell<\tfrac{1}{2},x_{m-1}+x_m\leq1$,}\\
(y_{(\ell)},(1-x_\ell)(x_{m-1}+x_m-2)+1,z_{(\ell)})&\text{ if $\tfrac{1}{2}\leq x_\ell\leq 1,x_{m-1}+x_m\leq1$,}\\
(x_1,\ldots,x_m)&\text{ if $x_{m-1}+x_m\geq1$.}
\end{cases}
$$

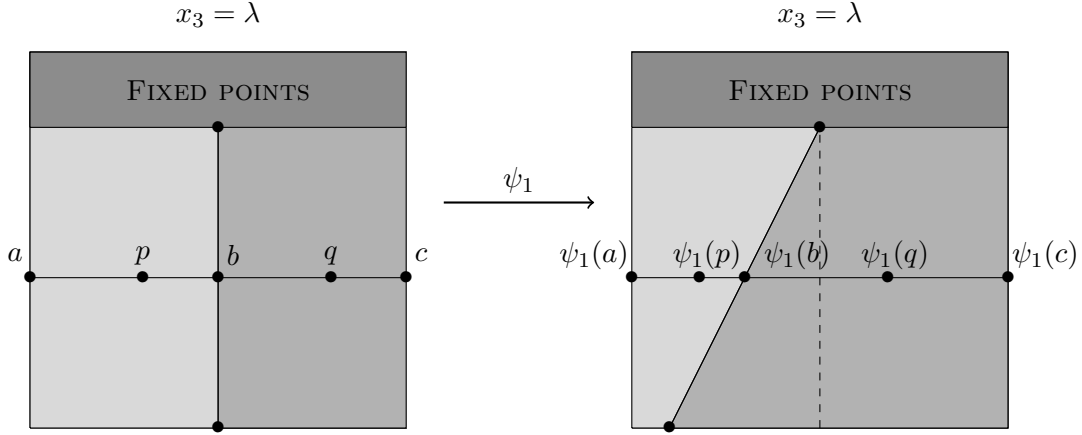
\begin{figure}[ht]
\centering
\begin{tikzpicture}
%\draw[thin,color=black,step=.5cm,dashed] (0,0) grid (13,6);

\draw (0,0) -- (5,0) -- (5,5) -- (0,5) -- (0,0);
\draw[fill=black!15!white] (0,0) -- (5,0) -- (5,5) -- (0,5) -- (0,0);
\draw[fill=black!30!white] (2.5,0) -- (5,0) -- (5,4) -- (2.5,4) -- (2.5,0);

\draw (2.5,0) -- (2.5,4);

\draw (0,2) -- (5,2);

\draw (0,4) -- (5,4) -- (5,5) -- (0,5) -- (0,4);
\draw[fill=black!45!white] (0,4) -- (5,4) -- (5,5) -- (0,5) -- (0,4);

\draw (2.5,4) node{$\bullet$};
\draw (5,2) node{$\bullet$};
\draw (0,2) node{$\bullet$};
\draw (2.5,2) node{$\bullet$};
\draw (1.5,2) node{$\bullet$};
\draw (4,2) node{$\bullet$};
\draw (2.5,0) node{$\bullet$};

\draw (2.5,4.5) node{\sc Fixed points};
\draw (-0.2,2.3) node{$a$};
\draw (1.5,2.3) node{$p$};
\draw (2.7,2.3) node{$b$};
\draw (4,2.3) node{$q$};
\draw (5.2,2.3) node{$c$};
\draw (2.5,5.5) node{$x_3=\lambda$};

\draw[->,thick] (5.5,3) -- (7.5,3);

\draw (6.5,3.3) node{$\psi_1$};

\draw (8,0) -- (13,0) -- (13,5) -- (8,5) -- (8,0);
\draw[fill=black!15!white] (8,0) -- (13,0) -- (13,5) -- (8,5) -- (8,0);
\draw[fill=black!30!white] (8.5,0) -- (13,0) -- (13,4) -- (10.5,4) -- (8.5,0);

\draw (8.5,0) -- (10.5,4);
\draw[dashed] (10.5,0) -- (10.5,4);

\draw (8,2) -- (13,2);

\draw (8,4) -- (13,4) -- (13,5) -- (8,5) -- (8,4);
\draw[fill=black!45!white] (8,4) -- (13,4) -- (13,5) -- (8,5) -- (8,4);

\draw (10.5,4) node{$\bullet$};
\draw (13,2) node{$\bullet$};
\draw (8,2) node{$\bullet$};
\draw (8.9,2) node{$\bullet$};
\draw (9.5,2) node{$\bullet$};
\draw (11.4,2) node{$\bullet$};
\draw (8.5,0) node{$\bullet$};

\draw (10.5,4.5) node{\sc Fixed points};
\draw (7.5,2.3) node{$\psi_{1}(a)$};
\draw (9,2.3) node{$\psi_{1}(p)$};
\draw (10.2,2.3) node{$\psi_{1}(b)$};
\draw (11.5,2.3) node{$\psi_{1}(q)$};
\draw (13.5,2.3) node{$\psi_{1}(c)$};
\draw (10.5,5.5) node{$x_3=\lambda$};

\end{tikzpicture}
\caption{Restriction of the map $\psi_{1}$ to the plane $x_3=\lambda$ for $m=3$.}
\end{figure}

Note that $\psi_\ell|_M:M\to M$ is a semialgebraic homeomorphism. Thus, the same happens to the restriction of the composition $\psi_{(\ell)}=\psi_\ell\circ\cdots\circ\psi_1:X_m\to X_m$ to $M$. Observe
\begin{equation}\label{final}
\psi_{(\ell)}(\{0\leq x_1\leq\tfrac{1}{2},\cdots,0\leq x_\ell\leq\tfrac{1}{2},x_{m-1}=0,x_m=0\})=\{0\}
\end{equation}
and consider the semialgebraic map ${\tt k}':=\psi_{(\ell)}\circ{\tt k}:M\to X_m$ where ${\tt k}:M\hookrightarrow X_m$ is the inclusion.

Denote $\gtq_k':=\Speca(\psi_{(\ell)})(\gtq_k)$ for $k=1,\ldots,m-\ell$. Observe $\gtq_0'\subsetneq\cdots\subsetneq\gtq_{m-\ell}'$ but for $k=m-\ell+1,\ldots,m$ it holds by \eqref{final} that $\Speca(\psi_{(\ell)})(\gtq_k)=\gtq_{m-\ell}'$, which is the maximal ideal of ${\mathcal S}(X_m)$ of all semialgebraic functions on $X_m$ vanishing at the origin. In addition by \eqref{final} we have
\begin{equation}\label{valueds}
{\tt d}_{X_m}(\gtq_k')=\begin{cases}
{\tt d}_{X_m}(\gtq_k)=m-k&\text{if $0\leq k\leq m-\ell-1$},\\
0&\text{if $k=m-\ell$.}
\end{cases}
\end{equation}
As the chain $\gtq_0\subsetneq\cdots\subsetneq\gtq_m$ is non-refinable, the same happens  by Theorem \ref{minq} to the chain $\gtq_0'\subsetneq\cdots\subsetneq\gtq_{m-\ell}'$. As $\Speca({\tt k})(\gtp_k)=\gtq_k$, it follows that $\Speca({\tt k}')$ maps the non-refinable chain $\gtp_0\subsetneq\cdots\subsetneq\gtp_m$ onto $\gtq_0'\subsetneq\cdots\subsetneq\gtq_{m-\ell}'$, so $\Speca({\tt k}')(\gtp_k)=\gtq_{m-\ell}'$ for $k=m-\ell,\ldots,m$. 

Define ${\tt j}':={\tt k}'\circ{\tt j}:N\hookrightarrow M':=X_m$ and recall $\Speca({\tt j})^{-1}(\gtp_k)=\{\gtp_k'\}$ for $k=0,\ldots,m$. Consequently, 
$$
\Speca({\tt j}')(\gtp_k')=\begin{cases}
\gtq_k'&\text{if $0\leq k\leq m-\ell-1$},\\
\gtq_{m-\ell}'&\text{if $m-\ell\leq k\leq m$.}
\end{cases}
$$
As ${\mathcal W}_{M'}={\mathcal E}_{M'}=\varnothing$, we have $\widehat{\gtq_i'}=\gtq_i'$. Thus, by \eqref{valueds} and Theorem \ref{fiberspec} (or Lemma \ref{fiberspec3}) \em the fiber $\Speca({\tt j}')^{-1}(\gtq_1')$ is a singleton while $\Speca({\tt j}')^{-1}(\gtq_k')$ is an infinite set for $k\geq2$\em. In particular, \em the fiber of $\gtq_{m-\ell}'$ contains the subchain $\gtp_{m-\ell}'\subsetneq\ldots\subsetneq\gtp_m'$\em. Compare this fact with Lemma \ref{pmchain}.
\qed
\end{examples}

\section{Proofs of Lemma \ref{main0}, Theorems \ref{main} and \ref{main2} and some consequences}\label{s4}
\renewcommand{\theparagraph}{\thesubsection.\arabic{paragraph}}
\setcounter{secnumdepth}{4}

The main purpose of this section is to prove Lemma \ref{main0} and Theorems \ref{main} and \ref{main2}. We also show how to reduce the computation of the size of the fibers of spectral maps induced by (dense) semialgebraic embeddings to the case of pure dimensional suitably arranged {\tt sa}-tuples. \setcounter{paragraph}{0}

\subsection{Proof of Lemma \ref{main0}} Let us see how we can reduce the computation of the size of the fibers of $\Specs({\tt j}):\Specs(N)\to\Specs(M)$ to analyze the fibers of $\Speca({\tt j}):\Speca(N)\to\Speca(M)$.

\begin{proof}[\sl Proof of Lemma \em\ref{main0}]
(i) First, let $\gtq\in\Specs(N)$ and $\gtp=\Specs({\tt j})(\gtq)=\gtq\cap{\mathcal S}(M)$. As ${\mathcal W}_N$ is contained in the units of ${\mathcal S}(N)$, it is clear that $\gtq\cap{\mathcal W}_N=\varnothing$, so $\gtp\cap{\mathcal W}_N=\varnothing$. Conversely, let $\gtp$ be a prime ideal of ${\mathcal S}(M)$ such that $\gtp\cap{\mathcal W}_N=\varnothing$. Consider the diagram
$$
\xymatrix{
\Specs(N)\ar[rr]^{\Specs({\tt j})}\ar@{^{(}->}[d]_{{\mathfrak i}_N}&&\Specs(M)\ar@{^{(}->}[d]^{{\mathfrak i}_M}\\
\Speca(N)\ar[rr]^{\Speca({\tt j})}&&\Speca(M)\\}
$$
and let $\gtp':=\gtp\cap{\mathcal S}^*(M)$. By Theorem \ref{main}(i), whose proof does not use Lemma \ref{main0}(i), the map $\Speca({\tt j}):\Speca(N)\to\Speca(M)$ is surjective. We claim:

\paragraph{}\label{previi} $\hspace{-1mm}$
\em If $\gtq'\in\Speca(N)$ satisfies $\Speca(\gtq')=\gtp'$, then $\gtq'\cap{\mathcal W}_N=\varnothing$\em. 

Suppose by contradiction that $\gtq'\cap{\mathcal W}_N\neq\varnothing$ and let $f\in\gtq'\cap{\mathcal W}_N$. Assume that $M$ is bounded and let $X_2:=\cl_{\R^{m+1}}({\rm graph}(f))$ and $X_1:=\cl_{\R^m}(N)=\cl_{\R^m}(M)$. Let $\rho:X_2\to X_1,\ (x,y)\mapsto x$ and ${\tt k}_2:N\hookrightarrow X_2,\ x\mapsto(x,f(x))$. By \ref{remainder} we know that $\rho^{-1}(X_1\setminus N)=X_2\setminus{\tt k}_2(N)$ and $\rho(X_2\setminus{\tt k}_2(N))=X_1\setminus N$. Let $\widehat{f}:X_2\to\R, (x,y)\mapsto y$. Observe that $\widehat{f}\circ{\tt k}_2=f$ and let $T:=Z(\widehat{f})$. As $f$ does not vanish in $N$, we have $T\cap{\tt k}_2(N)=\varnothing$. Then $\rho(T)\subset X_1\setminus N$, so $\rho^{-1}(\rho(T))\subset X_2\setminus{\tt k}_2(N)$. Let $g\in{\mathcal S}(X_1)$ be such that $Z(g)=\rho(T)$. Observe $Z(\widehat{f})=T\subset Z(g\circ\rho)$ and, as $\gtq'\cap{\mathcal S}(X_2)$ is a $z$-ideal and $\widehat{f}\in\gtq'\cap{\mathcal S}(X_2)$, we deduce $g\circ\rho\in\gtq'\cap{\mathcal S}(X_2)$. Thus, $g\in\gtq'$, so $g\in\gtp'=\gtq'\cap{\mathcal S}^*(M)$. On the other hand, $Z(g)=\rho(T)\subset X_1\setminus N$, so $g\in\gtp'\cap{\mathcal W}_N$, which is a contradiction. We conclude $\gtq'\cap{\mathcal W}_N=\varnothing$.

By \ref{inequality} the image of ${\mathfrak i}_N$ is the collection of all prime ideals of ${\mathcal S}^*(N)$ that do not intersect ${\mathcal W}_N$, so $\gtp=\Spec({\tt j})(\gtq'{\mathcal S}(N))$ belongs to the image of $\Specs({\tt j})$. 

(ii) The first part of the statement has been already proved in \ref{previi}. Once this is proved, the rest of the statement follows straightforwardly from \ref{inequality}.

(iii) We have to show that the restriction map 
\begin{equation}\label{hm}
\Specs({\tt j})|:\Specs(N)\setminus\Specs({\tt j})^{-1}(\Zz_{\Specs(M)}(h))\to\Speca(M)\setminus\Zz_{\Specs(M)}(h)
\end{equation}
is a homeomorphism where $Z(h)=\cl_{\R^m}(\cl_{\R^m}(M)\setminus N)$. Consider the locally compact semialgebraic set $N_{\lc}:=\cl_{\R^m}(M)\setminus Z(h)$. By Proposition \ref{rho} we know that $N_{\lc}$ is dense in $N$, so it is also dense in $M$. Consider the inclusions ${\tt j_1}:N_{\lc}\hookrightarrow N$ and ${\tt j_2}:N_{\lc}\hookrightarrow M$. It holds ${\tt j}_2={\tt j}\circ{\tt j}_1$. By Theorem \ref{mins} we know that the maps
\begin{multline*}
\Specs({\tt j}_1):\Specs(N_{\lc})\to\Specs(N)\setminus\Zz_{\Specs(N)}(h)\\
\text{and}\quad\Specs({\tt j}_2):\Specs(N_{\lc})\to\Specs(M)\setminus\Zz_{\Specs(M)}(h)
\end{multline*}
are homeomorphisms. Notice that $\Specs({\tt j})^{-1}(\Zz_{\Specs(M)}(h))=\Zz_{\Specs(N)}(h)$. As the following diagrams are commutative,
$$\enlargethispage{3mm}
\xymatrix{
N\ar@{^{(}->}[r]^{\tt j}&M\\
N_{\lc}\ar@{^{(}->}[u]^{{\tt j}_1}\ar@{^{(}->}[ur]_{{\tt j}_2}
}\quad\begin{array}{c} \\[20pt]\leadsto\end{array}\quad
\xymatrix{
\Specs(N)\setminus\Zz_{\Specs(N)}(h)\ar[rr]^{\Specs({\tt j})|}&&\Specs(M)\setminus\Zz_{\Specs(M)}(h)
\\
\Specs(N_{\lc})\ar[u]^{\Specs({\tt j}_1)}_\cong\ar[urr]_(0.45){\Specs({\tt j}_2)}^(0.45)\cong
}
$$
we conclude that also the map $\Specs({\tt j})|$ in \eqref{hm} is a homeomorphism, as required.
\end{proof}

\subsection{Proof of Theorem \ref{main}}\setcounter{paragraph}{0} 
The proof of this result follows mainly from Theorem \ref{minq} where it is partially approached for a suitably arranged {\tt sa}-tuple.

\begin{proof}[\sl Proof of Theorem \em\ref{main}]
(i) Consider the auxiliary suitably arranged {\tt sa}-tuple 
$$
(M,N_{\lc}, M\setminus N_{\lc},{\tt j}_1,{\tt i}_1)
$$ 
and the inclusion ${\tt j}_2:N_{\lc}\hookrightarrow N$. Observe ${\tt j}_1={\tt j}\circ{\tt j}_2$. Thus, $\Speca({\tt j}_1)=\Speca({\tt j})\circ\Speca({\tt j}_2)$ and since $\Speca({\tt j}_1)$ is by Theorem \ref{minq}(iii) surjective, also $\Speca({\tt j})$ is surjective, as required. 

(ii) To prove the first equality observe first that by \ref{closure} (ii) $\Speca(C)\cong\cl_{\Speca(N)}(C)$ and 
$$
\Speca(\cl_M(C))\cong\cl_{\Speca(M)}(\cl_M(C))=\cl_{\Speca(M)}(C). 
$$
The spectral map $\Speca({\tt j}_0):\Speca(C)\to\Speca(\cl_M(C))$ induced by the inclusion ${\tt j_0}:C\hookrightarrow\cl_M(C)$ is by (i) surjective, so $\Speca({\tt j})(\cl_{\Speca(N)}(C))=\cl_{\Speca(M)}(C)$.

To prove the second equality, note first that the inclusion
$$
\cl_{\Speca(N)}(C)\subset\Speca({\tt j})^{-1}(\cl_{\Speca(M)}(C))
$$
is clear, so
$$
\cl_{\Speca(N)}(C)\setminus\Speca({\tt j})^{-1}(Z)\subset\Speca({\tt j})^{-1}(\cl_{\Speca(M)}(C)\setminus Z).
$$
To prove the converse, we show
$$
\Speca({\tt j})\Big(\Speca(N)\setminus(\cl_{\Speca(N)}(C)\cup\Speca({\tt j})^{-1}(Z))\Big)
\subset\Speca(M)\setminus(\cl_{\Speca(M)}(C)\cup Z).
$$
Let $\gtq\in\Speca(N)\setminus(\cl_{\Speca(N)}(C)\cup\Speca({\tt j})^{-1}(Z))$. Then there exist $h\in{\mathcal S}^*(N)$ and $g\in{\mathcal S}^*(M)$ such that $C\subset Z_N(h)$, $h\not\in\gtq$, $Y\subset Z_M(g)$ and $g\not\in\gtq\cap{\mathcal S}^*(M)$. As $h$ is bounded and $g|_{M\setminus N}=0$, we deduce that $hg$ defines an element of ${\mathcal S}^*(M)$ such that $C\cup Y\subset Z_M(hg)$. As $hg\not\in\gtq\cap{\mathcal S}^*(M)=\Speca({\tt j})(\gtq)$, we deduce by \ref{closure}
$$
\Speca({\tt j})(\gtq)\not\in\cl_{\Speca(M)}(C\cup Y)=\cl_{\Speca(M)}(C)\cup Z,
$$
as required.

(iii) Let us check first: 

\paragraph{}\label{3a1}$\hspace{-1.5mm}$\em The restriction map $\Speca({\tt j})|:\Speca(N)\setminus\Speca({\tt j})^{-1}(Z)\to\Speca(M)\setminus Z$ is bijective\em.

We proceed by induction on the dimension of $M$. By Theorem \ref{minq}(iii) the result is true if $N$ is locally compact. In particular this holds if $\dim(M)=\dim(N)\leq1$. Assume the result true if $\dim(M)\leq d-1$ and let us show that it is also true if $\dim(M)=d$.

Consider the auxiliary suitably arranged {\tt sa}-tuples $(M_i,N_i,Y_i,{\tt j}_i,{\tt i}_i)$ for $i=1,2$ where
$$
\begin{cases}
M_1:=M,\ N_1:=N_{\lc},\ Y_1:=M\setminus N_{\lc},\\
M_2:=N,\ N_2:=N_{\lc},\ Y_2:=N\setminus N_{\lc}.\
\end{cases}
$$
As ${\tt j}_1={\tt j}\circ{\tt j}_2$, we infer $\Speca({\tt j}_1)=\Speca({\tt j})\circ\Speca({\tt j}_2)$. Write $Z_i:=\cl_{\Speca(M_i)}(Y_i)$. As $N_i$ is locally compact, the restriction map
\begin{equation}\label{88}
\Speca({\tt j}_i)|:\Speca(N_i)\setminus\Speca({\tt j}_i)^{-1}(Z_i)\to\Speca(M_i)\setminus Z_i
\end{equation}
is a homeomorphism. Observe $Y_2=N\setminus N_{\lc}\subset M\setminus N_{\lc}=Y_1$; hence,
$$
\cl_{\Speca(M)}(Y_2)\subset\cl_{\Speca(M)}(Y_1)=Z_1.
$$
By (ii) we get $\Speca({\tt j})(\cl_{\Speca(N)}(Y_2))=\cl_{\Speca(M)}(Y_2)$, so $\Speca({\tt j})(Z_2)\subset Z_1$. As $\Speca({\tt j})$ is surjective, $Z_2\subset\Speca({\tt j})^{-1}(Z_1)$ and
$$
\Speca({\tt j}_2)^{-1}(Z_2)\subset\Speca({\tt j}_2)^{-1}(\Speca({\tt j})^{-1}(Z_1))=\Speca({\tt j}_1)^{-1}(Z_1).
$$
Consequently, the restriction map
$$
\Speca({\tt j})|=\Speca({\tt j}_1)|\circ(\Speca({\tt j}_2)|)^{-1}:\Speca(N)\setminus\Speca({\tt j})^{-1}(Z_1)\to\Speca(M)\setminus Z_1
$$
is by equation \eqref{88} a homeomorphism. As $Y_1=M\setminus N_{\lc}=(M\setminus N)\cup(N\setminus N_{\lc})=Y\cup Y_2$, 
$$
Z_1=\cl_{\Speca(M)}(Y_1)=\cl_{\Speca(M)}(Y)\cup\cl_{\Speca(M)}(Y_2)=Z\cup\cl_{\Speca(M)}(Y_2).
$$
By (ii) we know
$$
\Speca({\tt j})^{-1}(\cl_{\Speca(M)}(Y_2)\setminus Z)=\cl_{\Speca(N)}(Y_2)\setminus\Speca({\tt j})^{-1}(Z)=Z_2\setminus\Speca({\tt j})^{-1}(Z)
$$
and to finish this part we must show: 

\paragraph{}\label{3a2}$\hspace{-1.5mm}$\em The restriction map $\Speca({\tt j})|:\cl_{\Speca(N)}(Y_2)\setminus\Speca({\tt j})^{-1}(Z)\to\cl_{\Speca(M)}(Y_2)\setminus Z$ is bijective.\em 

Indeed, by \ref{closure}(iii)
\begin{multline}\label{7688}
\cl_{\Speca(M)}(Y_2)\cap Z=\cl_{\Speca(M)}(\cl_M(Y_2))\cap\cl_{\Speca(M)}(\cl_M(Y))\\=\cl_{\Speca(M)}(\cl_M(Y_2)\cap\cl_M(Y)).
\end{multline}
As $Y_2=N\setminus N_{\lc}$ is closed in $N$, we have $\cl_M(Y_2)\cap N_{\lc}=\varnothing$, so
\begin{multline}\label{76}
\cl_M(Y_2)\setminus Y_2=\cl_M(Y_2)\cap(M\setminus(N\setminus N_{\lc}))=\cl_M(Y_2)\cap((M\setminus N)\cup N_{\lc}))\\
=(\cl_M(Y_2)\cap Y)\cup(\cl_M(Y_2)\cap N_{\lc})=(\cl_M(Y_2)\cap Y)\subset\cl_M(Y_2)\cap\cl_M(Y).
\end{multline}
Let ${\tt k}:\cl_M(Y_2)\hookrightarrow M$ be the inclusion map. By \ref{closure}(ii) the maps
\begin{align*}
&\Speca({\tt k}):\Speca(\cl_M(Y_2))\to\cl_{\Speca(M)}(\cl_M(Y_2)),\\
&\Speca({\tt i}_2):\Speca(Y_2)\to\cl_{\Speca(N)}(Y_2)
\end{align*}
are homeomorphisms. Consider the {\tt sa}-tuple $(M_3,N_3,Y_3,{\tt j}_3,{\tt i}_3)$ where $M_3:=\cl_M(Y_2)$ and $N_3:=Y_2$. Write $Z_3:=\cl_{\Speca(M_3)}(Y_3)$. By \eqref{7688} and \eqref{76} we get
\begin{multline}\label{incl00}
\Speca({\tt k})(Z_3)=\cl_{\Speca(M)}(Y_3)=\cl_{\Speca(M)}(\cl_M(Y_2)\setminus Y_2)\\
\subset\cl_{\Speca(M)}(\cl_M(Y_2)\cap\cl_M(Y))=\cl_{\Speca(M)}(Y_2)\cap Z.
\end{multline}
Consider the commutative diagrams
$$
\xymatrix{
Y_2\ar@{^{(}->}[rr]^{{\tt j}_3}\ar@{^{(}->}[d]_{{\tt i}_2}&&\cl_M(Y_2)\ar@{^{(}->}[d]^{{\tt k}}\\
N\ar@{^{(}->}[rr]^{{\tt j}}&&M
}
\quad\begin{array}{c} \\[20pt]\leadsto\end{array}\quad
\xymatrix{
\Speca(N_3)\ar@{->}[rr]^{\Speca({\tt j}_3)}\ar[d]_{\Speca({\tt i}_2)}^{\cong}&&\Speca(M_3)\ar[d]^{\Speca({\tt k})}_{\cong}\\
\cl_{\Speca(N)}(Y_2)\ar[rr]^{\Speca({\tt j})|}&&\cl_{\Speca(M)}(Y_2).
}
$$

Thus, by \eqref{incl00} it is enough to prove: 

\paragraph{}\label{3a3}$\hspace{-1.5mm}$\em The restriction map $\Speca({\tt j}_3)|:\Speca(N_3)\setminus\Speca({\tt j})^{-1}(Z_3)\to\Speca(M_3)\setminus Z_3$ is bijective\em.

As $\dim(M_3)=\dim(N_3)<\dim(N)=\dim(M)$, statement \ref{3a3} follows by induction, so \ref{3a2} and consequently \ref{3a1} also hold. 

Since $\Speca({\tt j})|$ is continuous and bijective, to finish the proof of (iii) we show: 

\paragraph{}$\hspace{-1.5mm}$\em The restriction map $\Speca({\tt j})|:\Speca(N)\setminus\Speca({\tt j})^{-1}(Z)\to\Speca(M)\setminus Z$ is open\em.

It is sufficient to show that given $g\in{\mathcal S}^*(N)$, the following straightforward equality holds:
\begin{multline*}
\Speca({\tt j})(\Dd_{\Speca(N)}(g)\cap(\Speca(N)\setminus\Speca({\tt j})^{-1}(Z)))\\
=\bigcup_{a\in\ker\phi}\Dd_{\Speca(M)}(ag)\cap(\Speca(M)\setminus Z)
\end{multline*}
where $\phi:{\mathcal S}^*(M)\to{\mathcal S}^*(Y),\ f\to f|_Y$ is the restriction homomorphism. 

(iv) This follows from the previous statements using $\Speca({\tt h})(\betaa N)=\betaa M$.
\end{proof}

\subsection{Proof of Theorem \ref{main2}}\setcounter{paragraph}{0} 
The proof of this result requires some preparation. In the following $(M,N,Y,{\tt j},{\tt i})$ denotes an {\tt sa}-tuple. Let us find first sufficient conditions to guarantee that the fibers of certain points of $\cl_{\Speca(M)}(Y)$ under $\Speca({\tt j})$ contain infinitely many points. 

\begin{lem}[Fibers of infinite size]\label{crucialstep}
Assume $M$ is pure dimensional of dimension $d$. Let $C\subset Y$ be a semialgebraic subset of $Y$ whose codimension in $M$ is $\geq 2$ and let $\gtp\in\cl_{\Speca(M)}(C)$. Then 
\begin{itemize}
\item[(i)] For each $r\geq1$ there exists a subset $\{\gtq_i\}_{i=1}^r\subset\Speca({\tt j})^{-1}(\gtp)$ such that $\gtq_i\not\subset\gtq_j$ and $\gtq_j\not\subset\gtq_i$ if $i\neq j$. In particular, the fiber $\Speca({\tt j})^{-1}(\gtp)$ is an infinite set.
\item[(ii)] If $\gtp=\gtm^*$ is a maximal ideal, the fiber $\Speca({\tt j})^{-1}(\gtm^*)$ contains infinitely many maximal ideals of ${\mathcal S}^*(N)$.
\end{itemize}
\end{lem}
Before proving this lemma, we need a preliminary result concerning triangulations.
\begin{lem}\label{trian}
Let $(K,\Phi)$ be a triangulation of a closed and bounded semialgebraic set $X$ compatible with a finite family ${\mathcal F}=\{T_1,\ldots,T_r\}$ of semialgebraic subsets of $X$. Let $(L,\Psi)$ be the first barycentric subdivision of $K$ and let $\sigma\in L$. Suppose $\sigma^0\cap\Psi^{-1}(T_k)=\varnothing$. Then either $\sigma\cap\Psi^{-1}(T_k)=\varnothing$ or there exists a proper face $\tau_k$ of $\sigma$ such that $\tau_k^0\subset\sigma\cap\Psi^{-1}(T_k)\subset\tau_k$.
\end{lem}
\begin{proof}[\sl Proof]
Write $\sigma:=[b_{\epsilon_0},\ldots,b_{\epsilon_d}]$ where $b_{\epsilon_i}$ is the barycenter of the simplex $\epsilon_i$ of $K$ and $\epsilon _{i+1}$ is a proper face of $\epsilon_i$ (see \cite[p.123]{sp}). Notice that $[b_{\epsilon_0},\ldots,b_{\epsilon_d}]\setminus\epsilon_{0}^0=[b_{\epsilon_1},\ldots,b_{\epsilon_d}]$ and so on. Assume $\sigma\cap\Psi^{-1}(T_k)\neq\varnothing$. As $\Psi^{-1}(T_k)$ is a finite union of open simplices of $K$ and the vertices of $\sigma$ are barycenters of simplices of $K$, there exists a first index $0\leq i\leq d$ such that $b_{\epsilon_i}\in\Psi^{-1}(T_k)$. Observe that $i\neq0$ because otherwise $\sigma\subset\epsilon_0\subset\Psi^{-1}(T_k)$, against the hypothesis $\sigma^0\cap\Psi^{-1}(T_k)=\varnothing$. Thus, $\tau_k:=[b_{\epsilon_i},\ldots,b_{\epsilon_d}]$ is a proper face of $\sigma$. We claim: \em $\tau_k$ satisfies $\tau_k^0\subset\sigma\cap\Psi^{-1}(T_k)\subset\tau_k$\em. 

Indeed, as $b_{\epsilon_j}\not\in T_k$ for $j=0,\ldots,i-1$, we deduce $\epsilon_j^0\cap\Psi^{-1}(T_k)=\varnothing$, so
$$
\sigma\cap\Psi^{-1}(T_k)\subset[b_{\epsilon_0},\ldots,b_{\epsilon_d}]\setminus\bigcup_{j=0}^{i-1}\epsilon_j^0=[b_{\epsilon_i},\ldots,b_{\epsilon_d}]=\tau_k.
$$
On the other hand, as $b_{\epsilon_i}\in T_k$, we have $\tau_k^0\subset\epsilon_i^0\subset T_k$, so $\tau_k^0\subset\sigma\cap\Psi^{-1}(T_k)$, as required.
\end{proof}

\begin{proof}[\sl Proof of Lemma \em\ref{crucialstep}]
The proof is conducted in several steps. We begin by proving the following:

\paragraph{}\label{required1}$\hspace{-1.5mm}$
\em For each $r\geq 1$ there exist pure dimensional closed semialgebraic subsets $M_1,\ldots,M_r$ of $M$ of dimension $d$ and a semialgebraic subset $C'$ of $C$ such that:
\begin{itemize}
\item[(1)] $M_i\cap Y=C'$ for each $i$ and $M_i\cap M_j=C'$ if $i\neq j$.
\item[(2)] $M_i\setminus Y=M_i\setminus C'$ is connected and dense in $M_i$.
\item[(3)] $\gtp\in\cl_{\Speca(M)}(C')$.
\end{itemize}\em

Indeed, as commented above, $Y$ is a semialgebraic subset of $M$ of dimension $\leq d-1$. Assume $M$ bounded and let $X:=\cl_{\R^m}(M)$. By Theorem \cite[9.2.1]{bcr} applied to $X$ and the family of semialgebraic sets ${\mathcal F}=\{T_1:=M,T_2:=Y,T_3:=C\}$ there exists a semialgebraic triangulation $(K,\Phi)$ of $X$ compatible with ${\mathcal F}$. After a barycentric subdivision, we may assume by Lemma \ref{trian} that for each $d$-dimensional simplex $\sigma$ of $K$ either $\sigma\cap T_k=\varnothing$ or there exists a proper face $\tau$ of $\sigma$ satisfying $\tau^0\subset\sigma\cap\Phi^{-1}(T_k)\subset\tau$ for $k=2,3$. We identify $|K|$ with $X$ and $\Phi^{-1}(T_k)$ with $T_k$. 

Let $\sigma_1,\ldots,\sigma_k$ be the $d$-dimensional simplices of $K$. Write $S_{\ell}:=\sigma_\ell\cap M$, which is a closed subset of $M$. As $M$ is pure dimensional, $M=\bigcup_{\ell=1}^kS_\ell$. Moreover, for each $\ell=1,\ldots,k$ either $C_{\ell}:=\sigma_\ell\cap C=S_{\ell}\cap C\subset S_{\ell}$ is empty or there exists a proper face $\tau_\ell$ of $\sigma_\ell$ such that $\tau_\ell^0\subset C_\ell\subset\tau_\ell$. In this latter case, $C_\ell=\tau_\ell\cap M$ is a closed subset of $M$. On the other hand 
$$
\text{$\Speca(M)=\bigcup_{\ell=1}^k\cl_{\Speca(M)}(S_\ell)$\quad and\quad $\cl_{\Speca(M)}(C)=\bigcup_{\ell=1}^k\cl_{\Speca(M)}(C_\ell)$.} 
$$
Assume $\gtp\in \cl_{\Speca(M)}(C_1)\subset \cl_{\Speca(M)}(S_1)$. Observe
$$
\dim(C_1)\leq\dim(C)\leq d-2=\dim(S_1)-2.
$$ 
Note that $\sigma_1^0\cap Y=\varnothing$ because $Y\in{\mathcal F}$ has dimension $\leq d-1$ and $(K,\Phi)$ is compatible with ${\mathcal F}$. Let $\upsilon_1$ be the proper face of $\sigma_1$ satisfying $\upsilon_1^0\subset\sigma_1\cap Y\subset\upsilon_1$; clearly, $\tau_1\subset\upsilon_1$. \em Let us construct $r$ simplices $\epsilon_1,\ldots,\epsilon_r\subset\sigma_1^0\cup\tau_1$ of dimension $d$ such that $\tau_1$ is a face of $\epsilon_i$, $\epsilon_i\cap\epsilon_j=\tau_1$ if $i\neq j$ and $\epsilon_i\cap Y=\epsilon_i\cap C_1=\tau_1\cap C_1$\em.

Indeed, let $\eta$ be the face of $\sigma_1$ generated by the vertices of $\sigma_1$ not contained in its face $\tau_1$. As $\dim(\tau_1)\leq\dim(\sigma_1)-2$, we have $e:=\dim(\eta)\geq 1$. We claim: $Y\cap\eta^0=\varnothing$. 

Otherwise, $\eta^0\subset Y\cap\sigma\subset\upsilon_1$ and as $\tau_1^0\subset C_1\subset Y\cap\sigma\subset\upsilon_1$ and $\upsilon_1$ is convex, we deduce $\upsilon_1\cap\sigma^0\neq\varnothing$, so $\sigma^0\subset\upsilon_1^0\subset Y$. This is a contradiction because $\dim(Y)\leq d-1$ and $\dim(\sigma)=d$.

Consider any collection $\{\eta_1,\ldots,\eta_r\}$ of pairwise disjoint simplices of dimension $e$ contained in $\eta^0$. A straightforward computation shows that the $d$-dimensional simplices $\epsilon_i$ generated by the vertices of $\tau_1$ and $\eta_i$ satisfy the desired conditions.

Now one proves readily that the semialgebraic sets $C':=C_1$ and $M_i:=\epsilon_i\cap M\subset S_1$ for $i=1,\ldots,r$ satisfy the required conditions in \ref{required1}. 

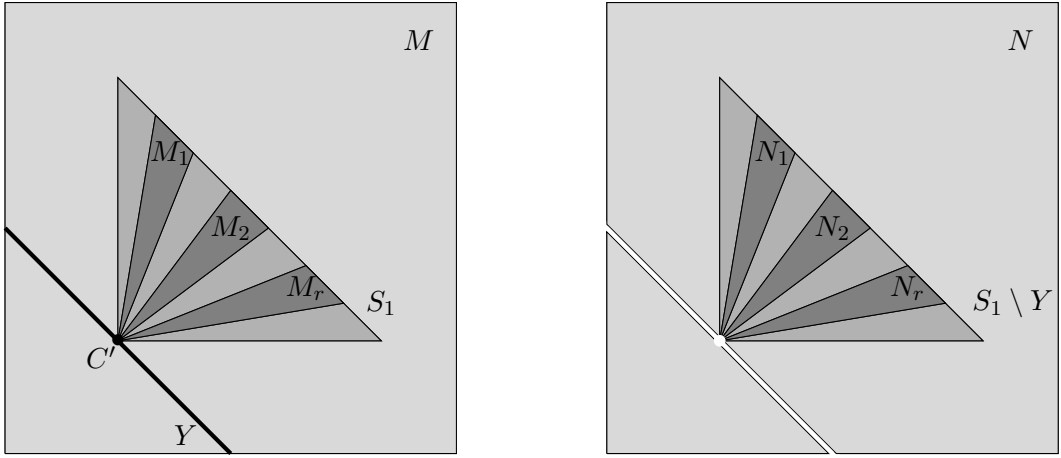
\begin{figure}[ht]
\centering
\begin{tikzpicture}
%\draw[thin,color=black,step=.5cm,dashed] (0,0) grid (14,6);

\draw (0,0) -- (6,0) -- (6,6) -- (0,6) -- (0,0);
\draw[fill=black!15!white] (0,0) -- (6,0) -- (6,6) -- (0,6) -- (0,0);

\draw (1.5,1.5) -- (1.5,5) -- (5,1.5) -- (1.5,1.5);
\draw[fill=black!30!white] (1.5,1.5) -- (1.5,5) -- (5,1.5) -- (1.5,1.5);

\draw (1.5,1.5) -- (2,4.5) -- (2.5,4) -- (1.5,1.5);
\draw[fill=black!50!white] (1.5,1.5) -- (2,4.5) -- (2.5,4) -- (1.5,1.5);

\draw (1.5,1.5) -- (3,3.5) -- (3.5,3) -- (1.5,1.5);
\draw[fill=black!50!white] (1.5,1.5) -- (3,3.5) -- (3.5,3) -- (1.5,1.5);

\draw (1.5,1.5) -- (4,2.5) -- (4.5,2) -- (1.5,1.5);
\draw[fill=black!50!white] (1.5,1.5) -- (4,2.5) -- (4.5,2) -- (1.5,1.5);

\draw[ultra thick] (3,0) -- (0,3);

\draw (1.5,1.5) node{$\bullet$};

\draw (5.5,5.5) node{$M$};
\draw (2.2,4) node{$M_1$};
\draw (3,3) node{$M_2$};
\draw (4,2.2) node{$M_r$};

\draw (5,2) node{$S_1$};
\draw (1.28,1.28) node{$C'$};
\draw (2.4,0.25) node{$Y$};

\draw (8,0) -- (10.95,0) -- (8,2.95) -- (8,0);
\draw[fill=black!15!white] (8,0) -- (10.95,0) -- (8,2.95) -- (8,0);

\draw (11.1,0) -- (14,0) -- (14,6) -- (8,6) -- (8,3.1) -- (11.1,0);
\draw[fill=black!15!white] (11.05,0) -- (14,0) -- (14,6) -- (8,6) -- (8,3.05) -- (11.05,0);

\draw (9.5,1.5) -- (9.5,5) -- (13,1.5) -- (9.5,1.5);
\draw[fill=black!30!white] (9.5,1.5) -- (9.5,5) -- (13,1.5) -- (9.5,1.5);

\draw (9.5,1.5) -- (10,4.5) -- (10.5,4) -- (9.5,1.5);
\draw[fill=black!50!white] (9.5,1.5) -- (10,4.5) -- (10.5,4) -- (9.5,1.5);

\draw (9.5,1.5) -- (11,3.5) -- (11.5,3) -- (9.5,1.5);
\draw[fill=black!50!white] (9.5,1.5) -- (11,3.5) -- (11.5,3) -- (9.5,1.5);

\draw (9.5,1.5) -- (12,2.5) -- (12.5,2) -- (9.5,1.5);
\draw[fill=black!50!white] (9.5,1.5) -- (12,2.5) -- (12.5,2) -- (9.5,1.5);

\draw [fill=white,white,ultra thick] (9.5,1.5) circle [radius=0.05];
\draw[white,ultra thick] (11,0) -- (8,3);

\draw (13.5,5.5) node{$N$};
\draw (10.2,4) node{$N_1$};
\draw (11,3) node{$N_2$};
\draw (12,2.2) node{$N_r$};

\draw (13.4,2) node{$S_1\setminus Y$};

\end{tikzpicture}
\caption{Construction of the semialgebraic sets $M_i$ and $N_i$.}
\end{figure}

\paragraph{}\label{required2}
Write $N_i:=N\cap M_i=M_i\setminus Y=M_i\setminus C'$ and ${\tt j}_i:N_i\hookrightarrow M_i$. It holds: \em $N_i$ is dense in $M_i$ and $N_i\cap N_j=\varnothing$ if $i\neq j$\em. Moreover, \em $N_i$ is closed in $N$ \em because $M_i$ is closed in $M$. By \ref{closure} $\Speca(N_i)\cong\cl_{\Speca(N)}(N_i)$, $\Speca(M_i)\cong\cl_{\Speca(M)}(M_i)$ for $i=1,\ldots,r$ and
$$
\Speca\Big(\bigsqcup_{i=1}^rN_i\Big)\cong\cl_{\Speca(N)}\Big(\bigsqcup_{i=1}^rN_i\Big).
$$
As the semialgebraic sets $N_i$ are pairwise disjoint closed connected subsets of $N$, the connected components of $\bigsqcup_{i=1}^rN_i$ are $N_1,\ldots,N_r$. By \ref{closure} \em the sets $\cl_{\Speca(N)}(N_i)$ are the connected components of 
$$
\cl_{\Speca(N)}\Big(\bigsqcup_{i=1}^rN_i\Big)
$$ 
\em and in particular they are disjoint. 

\paragraph{}
After the previous preparation we are ready to prove the statement:

(i) By Theorem \ref{main}(i) each map $\Speca({\tt j}_i):\Speca(N_i)\to\Speca(M_i)$ is surjective. Thus, the same happens to
$$
\Speca({\tt j})|_{\cl_{\Speca(N)}(N_i)}:\cl_{\Speca(N)}(N_i)\to\cl_{\Speca(M)}(M_i)\subset\Speca(M).
$$
Since $\gtp\in\cl_{\Speca(M)}(C')\subset\bigcap_{i=1}^r\cl_{\Speca(M)}(M_i)$, there exists a prime ideal $\gtq_i\in\cl_{\Speca(N)}(N_i)$ such that $\Speca({\tt j})(\gtq_i)=\gtp$ for each $i=1,\ldots,r$. Since the sets $\cl_{\Speca(M)}(N_i)$ are pairwise disjoint, $\gtq_i\not\subset\gtq_j$ and $\gtq_j\not\subset\gtq_i$ for $i\neq j$. As this holds for each $r\geq 1$, the fiber $\Speca({\tt j})^{-1}(\gtp)$ has infinitely many elements.

(ii) If $\gtp:=\gtm^*$ is a maximal ideal, let $\gtn_i^*$ be the unique maximal ideal of ${\mathcal S}^*(N)$ containing the prime ideal $\gtq_i$ constructed in (i) for $\gtm^*$. Note that $\gtn_i^*\in\cl_{\Speca(M)}(\{\gtq_i\})\subset\cl_{\Speca(M)}(N_i)$ and since $\cl_{\Speca(M)}(N_i)\cap \cl_{\Speca(M)}(N_j)=\varnothing$ if $i\neq j$, we conclude $\gtn_i^*\neq\gtn_j^*$ if $i\neq j$. As $\gtm^*$ is maximal and $\Speca({\tt j})(\gtq_i)=\gtm^*$, we deduce $\Speca({\tt j})(\gtn^*_i)=\gtm^*$. Thus, the fiber $\Speca({\tt j})^{-1}(\gtm^*)$ contains infinitely many maximal ideals.
\end{proof}
\begin{cor}\label{fiberspec2}
Let $(M,N,Y,{\tt j},{\tt i})$ be a {\tt sa}-tuple such that $M$ is pure dimensional of dimension $d$ and let $\gtP$ be a prime $z$-ideal of ${\mathcal S}(M)$ such that ${\tt d}_M(\gtP)\leq d-2$. Let $\gtq$ be a prime ideal of ${\mathcal S}^*(M)$ that contains $\gtp:=\gtP\cap{\mathcal S}^*(M)$. Then the fiber $\Speca({\tt j})^{-1}(\gtq)$ is an infinite set.
\end{cor}
\begin{proof}[\sl Proof]
Let $f\in\gtP$ be such that $\dim(Z(f))={\tt d}_M(\gtP)$ and let $C:=Z(f)$. Since $\gtP$ is a prime $z$-ideal, we deduce by \ref{closure}(i) that $\gtP\in\cl_{\Specs(M)}(C)$; hence, $\gtp\in\cl_{\Speca(M)}(C)$ by \ref{inequality}. Thus, also $\gtq\in\cl_{\Speca(M)}(C)$ and one can apply Lemma \ref{crucialstep}.
\end{proof}

\begin{proof}[\sl Proof of Theorem \em\ref{main2}] 
(i) By Theorem \ref{main}(iii) the proof of this statement (and its counterpart in (iv)) is reduced to prove the following:

\paragraph{}$\hspace{-1.5mm}$ \em Let $p\in Y$ be such that $\dim_p(M)\geq 2$. Then the fiber $\Speca({\tt j})^{-1}(\gtm_p^*)$ contains infinitely many maximal ideals of ${\mathcal S}^*(N)$.\em

Fix $s\geq 1$. Since $\dim_p(M)\geq 2$, there exist by the curve selection lemma \cite[2.5.5]{bcr} semialgebraic paths $\alpha_i:[0,1]\to\R^{n}$ for $i=1,\ldots,s$ such that $\alpha_i(0)=p$, $\alpha_i((0,1])\subset N$ and $\alpha_i((0,1])\cap\alpha_j((0,1])=\varnothing$ if $i\neq j$. Thus, the maximal ideals of ${\mathcal S}^*(N)$ given by $\gtn_i^*:=\{f\in{\mathcal S}^*(N):\,\lim_{t\to0}(f\circ\alpha_i)(t)=0\}$ are different. 

Note that $\gtn_i^*\cap{\mathcal S}^*(M)=\gtm^*_p$ because $f(p)=\lim_{t\to0}(f\circ\alpha_i)(t)=0$ for every $f\in\gtn_i^*\cap{\mathcal S}^*(M)$, so $\gtn_i^*\in \Speca({\tt j})^{-1}(\gtm_p^*)$ for $i=1,\ldots,s$. As this holds for each $s\geq1$, the fiber $\Speca({\tt j})^{-1}(\gtm_p^*)$ contains infinitely many maximal ideals.

(ii) The proof of this statement (and its counterpart in (iv)) is reduced to prove the following:

\paragraph{}$\hspace{-1.5mm}$ \em Suppose that $\dim_p(Y)\leq\dim_p(M)-2$ for all $p\in Y$ and let $\gtp\in\cl_{\Speca(M)}(Y)$. Then the fiber $\Speca({\tt j})^{-1}(\gtp)$ is an infinite set. Moreover, if $\gtp=\gtm^*$ is a maximal ideal of ${\mathcal S}^*(M)$, then its fiber $\Speca({\tt j})^{-1}(\gtm^*)$ contains infinitely many maximal ideals of ${\mathcal S}^*(N)$\em.

Let $\bs_M:=\{\bs_i(M)\}_{i=1}^r$ be the family of bricks of $M$, see \ref{bricks}. Denote $Y_i:=\bs_i(M)\cap Y$. As $Y=\bigcup_{i=1}^rY_i$, it holds 
$$
\cl_{\Speca(M)}(Y)=\bigcup_{i=1}^r\cl_{\Speca(M)}(Y_i).
$$ 

Define $I:=\{1,\ldots,r\}$ and $J:=\{j\in I:\,\gtp\in\cl_{\Speca(M)}(\bs_j(M))\}$. For every $i\in I\setminus J$ there exists by \ref{closure}(i) $f_i\in{\mathcal S}^*(M)\setminus\gtp$ such that $Z(f_i)=\bs_i(M)$. Let $f:=\prod_{i\in I\setminus J}f_i\in{\mathcal S}^*(M)\setminus\gtp$. Then
\begin{equation*}
\gtp\in\Dd_{\Speca(M)}(f)\cap\bigcap_{j\in J}\cl_{\Speca(M)}(Y_j)=\Dd_{\Speca(M)}(f)\cap\bigcap_{j\in J}\cl_{\Speca(M)}(\cl_M(D(f)\cap Y_j)).
\end{equation*}

Denote $j_0:=\min(J)$ and $C:=\cl_M(D(f)\cap Y_{j_0})$. We claim: $\dim(C)\leq\dim(\bs_{j_0}(M))-2.$

Indeed, observe $D(f)=D(f)\cap\bigcup_{j\in J}\bs_j(M)$. Since $\dim_p(Y)\leq \dim_p(M)-2$ for all $p\in Y$ and $\dim(\bs_{j_0}(M))\geq\dim(\bs_j(M))$ for all $j\in J$, we deduce for all $p\in D(f)\cap Y$
$$
\dim_p(Y_{j_0})\leq\dim_p(Y)\leq\dim_p(M)-2=\dim_p\Big(\bigcup_{j\in J}\bs_j(M)\Big)-2=\dim_p(\bs_{j_0}(M))-2. 
$$
We conclude $\dim(C)=\dim(D(f)\cap Y_{j_0})\leq\dim(\bs_{j_0}(M))-2$.

Now, since $\gtp\in\cl_{\Speca(M)}(C)\subset\cl_{\Speca(M)}(\bs_{j_0}(M))\cong\Speca(\bs_{j_0}(M))$, we deduce that $\Speca({\tt j})^{-1}(\gtp)$ is by Lemma \ref{crucialstep} an infinite set. Moreover, if $\gtp=\gtm^*$ is in addition a maximal ideal of ${\mathcal S}^*(M)$, its fiber contains by Lemma \ref{crucialstep} infinitely many maximal ideals, as required.

(iii) (and the remaining part of (iv)) Since $N$ is dense in $M$, we have
$$
\dim(Y)=\dim(\cl_M(N)\setminus N)<\dim(M)=1.
$$ 
Thus, $Y$ is a finite set and $\cl_{\Speca(M)}(Y)=Y$. Moreover, $\partial N$ is by \cite[5.17]{fg5} also a finite set. To finish we must show $\Speca({\tt j})^{-1}(Y)\subset\partial N$. 

Let $p\in Y$ and $\gtq\in\Speca({\tt j})^{-1}(\gtm_p^*)$. Notice that \em $\gtq$ is not a minimal prime ideal of ${\mathcal S}^*(N)$ \em because otherwise $\gtm^*_p$ would be by Theorem \ref{minq}(i) a minimal prime ideal of ${\mathcal S}^*(M)$, against Theorem \ref{minimalnlc}. Since $N$ is one dimensional, each prime ideal of ${\mathcal S}^*(N)$ is either minimal or maximal (but not both, see \cite[7.3]{fe1}). Thus, $\gtq$ is a maximal ideal of ${\mathcal S}^*(N)$ and it only remains to check that it is free. Otherwise $\gtq=\gtm^*_q$ for some point $q\in N$, so $\gtm^*_p=\Speca({\tt j})(\gtq)=\Speca({\tt j})(\gtm^*_q)=\gtm^*_{{\tt j}(q)}=\gtm^*_q$, which is wrong because $p\in Y=M\setminus N$. 
\end{proof}

\subsection{Size of the fibers of a {\tt sa}-tuple}
Let $(M,N,Y,{\tt j},{\tt i})$ be a {\tt sa}-tuple and $\gtp$ a prime ideal of ${\mathcal S}^*(M)$. To compute the size of the fiber $\Speca({\tt j})^{-1}(\gtp)$ we proceed as follows. 

\subsubsection{Reduction to the case in which $M$ is pure dimensional}
Let $\bs_N$ and $\bs_M$ be the families of bricks of $N$ and $M$. By \ref{bricks} we know
\begin{itemize}
\item[(i)] $\bs_M:=\{\bs_i(M)=\cl_M(\bs_i(N))\}_i$,
\item[(ii)] $\bs_N:=\{\bs_i(N)=\bs_i(M)\cap N\}_i$.
\end{itemize}
Thus, $N_i:=\bs_i(N)$ is dense in $M_i:=\bs_i(M)$, so $(N_i,M_i,Y_i,{\tt j}_i,{\tt i}_i)$ is a {\tt sa}-tuple.

Moreover, since $\Speca({\tt j})$ is continuous, 
$$
\Speca({\tt j})(\cl_{\Speca(N)}(N_i))\subset\cl_{\Speca(M)}({\tt j}(N_i))=\cl_{\Speca(M)}(\cl_M(N_i))=\cl_{\Speca(M)}(M_i).
$$
Moreover, $\Speca(N)=\bigcup_{i=1}^r\cl_{\Speca(N)}(N_i)$ and $\Speca(M)=\bigcup_{i=1}^r\cl_{\Speca(M)}(M_i)$. In addition, by \ref{closure}(ii) $\cl_{\Speca(N)}(N_i)\cong\Speca(N_i)$ (because $N_i$ is closed in $N$) and $\cl_{\Speca(M)}(M_i)\cong\Speca(M_i)$ (because $M_i$ is closed in $M$). Thus, for our purposes it is enough to compute the size of the fibers of the spectral maps $\Speca({\tt j}_i):\Speca(N_i)\to\Speca(M_i)$ corresponding to the suitably arranged {\tt sa}-tuples $(N_i,M_i,Y_i,{\tt j}_i,{\tt i}_i)$.

So we assume in the following that $M$ is pure dimensional.

\subsubsection{Reduction to the case in which $N$ is locally compact}\label{red1} 

By Corollary \ref{rho2} it holds that $\cl_M(\rho_1(N))$ is a semialgebraic subset of $M$ of (local) codimension $\geq2$; hence, $C:=\cl_M(\rho_1(N))\cap\cl_M(Y)$ is a closed semialgebraic subset of $\cl_M(Y)$ that has (local) codimension $\geq2$ in $M$. Denote $Z_1:=\cl_{\Speca(M)}(Y)$ and $T:=\cl_{\Speca(M)}(\rho_1(N))$. By \ref{closure}
$$
\cl_{\Speca(M)}(C)=\cl_{\Speca(M)}(\cl_M(\rho_1(N)))\cap\cl_{\Speca(M)}(\cl_M(Y))
=T\cap Z_1.
$$
By Theorem \ref{main}(ii) it holds
$$
\Speca({\tt j})^{-1}(T\setminus Z_1)=\cl_{\Speca(N)}(\rho_1(N))\setminus\Speca({\tt j})^{-1}(Z_1).
$$
If $\gtp\in\cl_{\Speca(M)}(C)$, we know by Lemma \ref{crucialstep} that $\Spec({\tt j})^{-1}(\gtp)$ is an infinite set. Thus, by Theorem \ref{main2}(ii) we conclude that if $\gtp\in T$, the fiber
\begin{equation}\label{size0}
\Speca({\tt j})^{-1}(\gtp)\text{ is }
\left\{
\begin{array}{ll}
\!\!\text{a singleton}&\text{if $\gtp\in T\setminus Z_1$,}\\[4pt]
\!\!\text{an infinite set}&\text{if $\gtp\in T\cap Z_1$.}
\end{array}
\right.
\end{equation}
So it remains to determine the size of the fiber of a prime ideal $\gtp\in\Speca(M)\setminus T$. Consider the {\tt sa}-tuple $(N,N_{\lc},\rho_1(N),{\tt j}_2,{\tt i}_2)$ and denote $Z_2:=\cl_{\Speca(N)}(\rho_1(N))$. By Theorem \ref{main}(iii) the restriction map
$$
\Speca({\tt j}_2)|:\Speca(N_{\lc})\setminus\Speca({\tt j}_2)^{-1}(Z_2)\to\Speca(N)\setminus Z_2
$$
is a homeomorphism. By Theorem \ref{main}(ii) $\Speca({\tt j})(Z_2)=T$, so $Z_2\subset\Speca({\tt j})^{-1}(T)$ and consequently $\Speca({\tt j}_2)^{-1}(Z_2)\subset\Speca({\tt j}\circ{\tt j}_2)^{-1}(T)$. Thus, the restriction map
\begin{equation}\label{rest2}
\Speca({\tt j}_2)|:\Speca(N_{\lc})\setminus\Speca({\tt j}\circ{\tt j}_2)^{-1}(T)\to\Speca(N)\setminus\Speca({\tt j})^{-1}(T)
\end{equation} 
is also a homeomorphism. We have the following commutative diagrams
\begin{equation}\label{remate}
\xymatrix{
N\ar@{^{(}->}[r]^{{\tt j}}&M\\
N_{\lc}\ar@{^{(}->}[u]^{{\tt j}_2}\ar@{^{(}->}[ur]_{{\tt j}\circ{\tt j}_2}
}
\quad\begin{array}{c} \\[20pt]\leadsto\end{array}\quad
\xymatrix{
\Speca(N)\setminus\Speca({\tt j})^{-1}(T)\ar[rr]^(0.575){\Speca({\tt j})|}&&\Speca(M)\setminus T\\
\Speca(N_{\lc})\setminus\Speca({\tt j}\circ{\tt j}_2)^{-1}(T)\ar[u]^{{\Speca(\tt j}_2)|}_{\cong}\ar[urr]_{\Speca({\tt j}\circ{\tt j}_2)|}
}
\end{equation}
Therefore, for our purposes it is enough to determine the size of the fibers of the spectral map induced by the suitably arranged {\tt sa}-tuple $(M,N_{\lc},Y':=M\setminus N_{\lc},{\tt j}_3:={\tt j}\circ{\tt j}_2,{\tt i}_3)$. Indeed, for the prime ideals $\gtp\in T$ we have already computed the size of the fiber $\Speca({\tt j})^{-1}(\gtp)$ in \eqref{size0} and if $\gtp\in\Speca(M)\setminus T$, we know, as the restriction map in \eqref{rest2} is a homeomorphism, that the fiber $\Speca({\tt j})^{-1}(\gtp)$ is homeomorphic to the fiber $\Speca({\tt j}_3)^{-1}(\gtp)$ (see diagram \eqref{remate}).

So we assume that $N$ is locally compact and we are reduced to study the case of a suitably arranged {\tt sa}-tuple $(M,N,Y,{\tt j},{\tt i})$ such that $M$ is pure dimensional. This case is fully studied in Theorem \ref{fiberspec} that we prove in the next section.

\section{Proof of Theorem \ref{fiberspec}}\label{s5}

In this section we prove Theorem \ref{fiberspec}. Its proof is quite involved and requires some preliminary results. In the following $(M,N,Y,{\tt j},{\tt i})$ denotes a suitably arranged {\tt sa}-tuple such that $M$ is pure dimensional of dimension $d$. In particular, in the following $Y$ is a closed subset of $M$.

\subsection{Preliminary results}
Recall that ${\mathcal W}_M$ is the multiplicative set of those functions $f\in{\mathcal S}^*(M)$ such that $Z(f)=\varnothing$ and ${\mathcal E}_M$ is the multiplicative set of those $f\in{\mathcal S}(M)$ such that $Z(f)=M\setminus M_{\lc}$. Denote $Z_1:=\cl_{\Speca(M)}(\rho_1(M))$. 

\begin{lem}\label{e}
Let $\gtp\in\cl_{\Speca(M)}(C)\setminus Z_1$ where $C$ is a closed semialgebraic subset of $M$. Then
\begin{itemize}
\item[(i)] The threshold $\widehat{\gtp}$ of $\gtp$ in ${\mathcal S}^*(M)$ defined in \eqref{gtQ} is univocally determined by $\gtp$. In addition, if $\gtp\cap{\mathcal W}_M=\varnothing$ but $\gtp\cap{\mathcal E}_M\neq\varnothing$, there exists a maximal ideal $\gtm_1$ of ${\mathcal S}(M_{\lc})$ such that $\widehat{\gtp}=\gtm_1\cap{\mathcal S}^*(M)$.
\item[(ii)] $\widehat{\gtp}\in\cl_{\Speca(M)}(C)$ and $\widehat{\gtp}{\mathcal S}(M)\in\cl_{\Specs(M)}(C)$.
\item[(iii)] Every non-refinable chain of prime ideals of ${\mathcal S}^*(M)$ through $\gtp$ contains also $\widehat{\gtp}$.
\item[(iv)] $\widehat{\gtp}{\mathcal S}(M)$ is a $z$-ideal. 
\item[(v)] If ${\tt d}_M(\widehat{\gtp}{\mathcal S}(M))\leq d-2$, the fiber $\Speca({\tt j})^{-1}(\gtp)$ is an infinite set.
\end{itemize}
\end{lem}
\begin{proof}[\sl Proof]
Consider the auxiliary suitably arranged {\tt sa}-tuples 
$$
(M,M_{\lc},\rho_1(M),{\tt j}_1,{\tt i}_1)\quad\text{ and }\quad(M_{\lc},N,Y_2:=M_{\lc}\setminus N,{\tt j}_2,{\tt i}_2).
$$
Note that $N\subset M_{\lc}$ because $N$ is locally compact and dense in $M$, $\rho_1(M)\subset Y$ and ${\tt j}={\tt j}_1\circ{\tt j}_2$. By Theorem \ref{main}(iii) the restriction map
$$
\Speca({\tt j}_1)|:\Speca(M_{\lc})\setminus\Speca({\tt j})^{-1}(Z_1)\to\Speca(M)\setminus Z_1
$$
is a homeomorphism. Write $\Speca({\tt j}_1)^{-1}(\gtp)=\{\gtp_1\}$. The size of the fiber $\Speca({\tt j})^{-1}(\gtp)$ coincides with the one of $\Speca({\tt j}_2)^{-1}(\gtp_1)$ because they are homeomorphic sets. Thus, to prove statement (v) we are reduced to prove that $\Speca({\tt j}_2)^{-1}(\gtp_1)$ is an infinite set.

We prove all statements simultaneously by distinguishing two cases:

\noindent{\bf Case 1.} If $\gtp\cap{\mathcal W}_M\neq\varnothing$, it is clear that $\widehat{\gtp}:=\gtm\cap{\mathcal S}^*(M)$ is univocally determined by $\gtp$. By \ref{chainp}(iii) $\gtm\in\cl_{\Specs(M)}(C)$ and $\gtm\cap{\mathcal S}^*(M)\in\cl_{\Speca(M)}(C)$ and by \ref{chainp}(ii) every non-refinable chain of prime ideals of ${\mathcal S}^*(M)$ containing $\gtp$ contains also $\widehat{\gtp}$. Moreover, $\widehat{\gtp}{\mathcal S}(M)=\gtm$ is a prime $z$-ideal because it is maximal. It only remains to prove (v).
 
\paragraph{}\label{4a1} We claim: \em $\gtm{\mathcal S}(M_{\lc})$ is a prime ideal of ${\mathcal S}(M_{\lc})$ that satisfies $\gtm{\mathcal S}(M_{\lc})\cap{\mathcal S}(M)=\gtm$\em. 

Let us prove first that $\gtm\cap{\mathcal E}_M=\varnothing$. Indeed, as $\gtp\not\in Z_1$, we deduce $\gtm\cap{\mathcal S}^*(M)\not\in Z_1$; hence, by \ref{closure}(i) $\gtm\not\in\cl_{\Specs(M)}(\rho_1(M))$. As $\gtm$ is a prime $z$-ideal (because it is maximal), $\gtm\cap{\mathcal E}_M=\varnothing$. Now our claim follows from Theorem \ref{mins}.

\paragraph{}As $\widehat{\gtp}\not\in Z_1$, the fiber $\Speca({\tt j}_1)^{-1}(\widehat{\gtp})$ is a singleton $\{\widehat{\gtp}_1\}$. As $\widehat{\gtp}\subset\gtp$, we deduce by Theorem \ref{minq}(iv) that $\widehat{\gtp}_1\subset\gtp_1$. We claim $\widehat{\gtp}_1=\gtm{\mathcal S}(M_{\lc})\cap{\mathcal S}^*(M_{\lc})$. 

Indeed,
$$
\gtm{\mathcal S}(M_{\lc})\cap{\mathcal S}^*(M_{\lc})\cap{\mathcal S}^*(M)=\gtm{\mathcal S}(M_{\lc})\cap{\mathcal S}(M)\cap{\mathcal S}^*(M)=\gtm\cap{\mathcal S}^*(M)=\widehat{\gtp}.
$$

As ${\tt d}_{M_{\lc}}(\gtm{\mathcal S}(M_{\lc}))\leq{\tt d}_M(\gtm)\leq d-2$, we conclude by Corollary \ref{fiberspec2} that $\Speca({\tt j}_2)^{-1}(\gtp_1)$ is an infinite set, as wanted.

\noindent{\bf Case 2.} If $\gtp\cap{\mathcal W}_M=\varnothing$ and $\gtp\cap{\mathcal E}_M=\varnothing$, we have $\widehat{\gtp}=\gtp$, so it is univocally determined by $\gtp$ (and it also holds (iii)). By Theorem \ref{mins} and \ref{inequality} $\gtp{\mathcal S}(M_{\lc})$ is a prime ideal of ${\mathcal S}(M_{\lc})$ that satisfies $\gtp{\mathcal S}(M_{\lc})\cap{\mathcal S}(M)=\gtp{\mathcal S}(M)$. As $M_{\lc}$ is locally compact, $\gtp{\mathcal S}(M_{\lc})$ is a prime $z$-ideal, so by \ref{zideal3} $\gtp{\mathcal S}(M)$ is also a prime $z$-ideal. Finally, if ${\tt d}_M(\gtp{\mathcal S}(M))={\tt d}_M(\widehat{\gtp}{\mathcal S}(M))\leq d-2$, the fiber $\Speca({\tt j})^{-1}(\gtp)$ is by Corollary \ref{fiberspec2} an infinite set, so this situation is completely approached.

Assume next $\gtp\cap{\mathcal W}_M=\varnothing$ and $\gtp\cap{\mathcal E}_M\neq\varnothing$. As ${\mathcal E}_M\subset{\mathcal W}_{M_{\lc}}$ and $\gtp_1\cap{\mathcal S}^*(M)=\gtp$, we have $\gtp_1\cap{\mathcal W}_{M_{\lc}}\neq\varnothing$. Let ${\gtm_1}^*$ be the unique maximal ideal of ${\mathcal S}^*(M_{\lc})$ that contains $\gtp_1$ and let $\gtm_1$ be the unique maximal ideal of ${\mathcal S}(M_{\lc})$ such that $\gtm_1\cap{\mathcal S}^*(M_{\lc})\subset{\gtm_1}^*$. By \ref{chainp}(ii) we know
\begin{equation}\label{subchain}
\gtm_1\cap{\mathcal S}^*(M_{\lc})\subset\gtp_1\subset\gtm_1^*.
\end{equation}

\paragraph{}We claim: $\widehat{\gtp}=\gtm_1\cap{\mathcal S}^*(M)$. Assume this proved for a while. As $\gtm_1$ is univocally determined by $\gtp_1$, we conclude that $\widehat{\gtp}$ is univocally determined by $\gtp$ (and this proves (i)).

Indeed, let $\gtq$ be a prime ideal of ${\mathcal S}^*(M)$ contained in $\gtp$. As $\gtp\not\in Z_1$, we have $\gtq\not\in Z_1$, so $\Speca({\tt j}_1)^{-1}(\gtq)$ is a singleton $\{\gtq_1\}$. By Theorem \ref{minq}(iv) it holds $\gtq_1\subset\gtp_1$. Let us check: \em $\gtq\cap{\mathcal E}_M\neq\varnothing$ if and only if $\gtq_1\cap{\mathcal W}_{M_{\lc}}\neq\varnothing$\em. 

If $\gtq_1\cap{\mathcal W}_{M_{\lc}}\neq\varnothing$, pick $g\in\gtq_1\cap{\mathcal W}_{M_{\lc}}$ and $h\in{\mathcal S}^*(M)$ such that $Z(h)=\rho_1(M)$. Observe $gh\in\gtq\cap{\mathcal E}_M$. The converse follows because ${\mathcal E}_M\subset{\mathcal W}_{M_{\lc}}$.

By \ref{chainp}(i) we have $\gtq_1\subset\gtm_1\cap{\mathcal S}^*(M_{\lc})$ if $\gtq\cap{\mathcal E}_M=\varnothing$ and $\gtm_1\cap{\mathcal S}^*(M_{\lc})\subset\gtq_1$ if $\gtq\cap{\mathcal E}_M\neq\varnothing$. By Theorem \ref{minq}(iv), the definition of $\widehat{\gtp}$ and the equality $\Speca({\tt j}_1)^{-1}(\gtm_1\cap{\mathcal S}^*(M))=\{\gtm_1\cap{\mathcal S}^*(M_{\lc})\}$ we deduce $\gtm_1\cap{\mathcal S}^*(M)=\widehat{\gtp}$.

The fact that $\Speca({\tt j}_1)^{-1}(\gtq)$ is a singleton for each prime ideal $\gtq\subset\gtp$ together with $\gtm_1\cap{\mathcal S}^*(M)=\widehat{\gtp}$ and equation \eqref{subchain} imply by \ref{chainp}(ii) that statement (iii) holds.

\paragraph{}Next we claim: $\gtp_1\in\cl_{\Speca(M_{\lc})}(\cl_M(C\setminus\rho_1(M)))=\cl_{\Speca(M_{\lc})}(C\setminus\rho_1(M))$.

Indeed, by \ref{closure}(i) we have to show that if $Z(g)=\cl_M(C\setminus\rho_1(M))$, then $g\in\gtp_1$. As $\gtp\not\in Z_1$, there exists $h\in{\mathcal S}^*(M)\setminus\gtp$ such that $Z(h)=\rho_1(M)$. As $C\subset Z(gh)$ and $\gtp\in\cl_{\Speca(M)}(C)$, we have $hg\in\gtp\subset\gtp_1$. As $h\not\in\gtp$ and $\gtp=\gtp_1\cap{\mathcal S}^*(M)$, we conclude $h\not\in\gtp_1$, so $g\in\gtp_1$. Consequently, $\gtp_1\in\cl_{\Speca(M_{\lc})}(C\setminus\rho_1(M))$.

\paragraph{}By \ref{chainp}(iii) and equation \eqref{subchain} $\gtm_1\in\cl_{\Specs(M_{\lc})}(C\setminus\rho_1(M))$ and $\gtm_1\cap{\mathcal S}^*(M_{\lc})\in\cl_{\Speca(M_{\lc})}(C\setminus\rho_1(M))$. By the continuity of $\Speca({\tt j}_1)$
$$
\{\widehat{\gtp}\}=\Speca({\tt j}_1)(\{\gtm_1\cap{\mathcal S}^*(M_{\lc})\})\subset\Speca({\tt j}_1)(\cl_{\Speca(M_{\lc})}(C\setminus\rho_1(M)))\subset\cl_{\Speca(M)}(C),
$$
so $\widehat{\gtp}{\mathcal S}(M)\in\cl_{\Specs(M)}(C)$ (and this proves (ii)).

Notice that $\widehat{\gtp}{\mathcal S}(M)=\gtm_1\cap{\mathcal S}(M)$, so ${\tt d}_{M_{\lc}}(\gtm_1)\leq{\tt d}_M(\widehat{\gtp}{\mathcal S}(M))\leq d-2$ and by \ref{zideal3} $\widehat{\gtp}{\mathcal S}(M)$ is a prime $z$-ideal, so statement (iv) holds. By Corollary \ref{fiberspec2} and equation \eqref{subchain} we deduce that $\Speca({\tt j}_2)^{-1}(\gtp_1)$ is an infinite set, which proves (v), as required.
\end{proof}

\begin{remark}
We have proved that if ${\tt d}_M(\widehat{\gtp}{\mathcal S}(M))\leq d-2$, the fiber $\Speca({\tt j})^{-1}(\gtp)$ is an infinite set. In \S\ref{s} we will prove the converse of this fact, namely: \em If ${\tt d}_M(\widehat{\gtp}{\mathcal S}(M))=d-1$, the fiber $\Speca({\tt j})^{-1}(\gtp)$ is a finite set\em.
\end{remark}

We are ready to prove Lemmas \ref{pmchain} and \ref{fiberspec3}.

\begin{proof}[\sl Proof of Lemma \ref{pmchain}]
The proof is conducted in several steps.

\noindent{\bf Step 1.} Assume that $M$ is bounded. Let $f_i\in\gtp_i\setminus\gtp_{i-1}$ for $i=2,\ldots,r$ and $f_1\in\gtp_1$ be such that $\dim(Z(f_1))={\tt d}_M(\gtm)$ and $Z(f_1)=Y$. After substituting $M$ with ${\rm graph}(f_1,\ldots,f_r)$, we may assume that each $f_i$ can be extended continuously to $X_1:=\cl_{\R^m}(M)=\cl_{\R^m}(N)$. Consider the inclusion ${\tt k}_1:M\hookrightarrow X_1$ and denote $\gtP_i:=\gtp_i\cap{\mathcal S}(X_1)$. Observe $f_i\in\gtP_i\setminus\gtP_{i-1}$. 

\paragraph{}Let $g_j\in\gtq_j\setminus\gtq_{j-1}$ for $j=2,\ldots,s$ and consider the semialgebraic compactification of $N$
$$
{\tt k}_2:N\hookrightarrow X_2:=\cl_{\R^{m+{s-1}}}({\rm graph}(g_2,\ldots,g_s)), x\mapsto(x,g_2(x),\ldots,g_s(x)).
$$
Denote $\gtQ_j:=\gtq_j\cap{\mathcal S}(X_2)$ and observe $g_j\in\gtQ_j\setminus\gtQ_{j-1}$. Consider the (surjective) projection $\pi:X_2\to X_1,\ (x,y)\mapsto x$. Let us prove : ${\tt d}_{X_1}(\gtP_1)={\tt d}_{X_2}(\gtQ_1)=d-1$. 

Indeed, observe first
$$
d-1={\tt d}_M(\gtp_1)\leq{\tt d}_{X_1}(\gtP_1)\leq\max\{\dim(Z(f_1)),\dim(X_1\setminus M)\}\leq d-1;
$$
hence, ${\tt d}_{X_1}(\gtP_1)=d-1$. On the other hand, 
\begin{equation}\label{p0q0}
\gtQ_1\cap{\mathcal S}(X_1)=\gtq_1\cap{\mathcal S}(X_2)\cap{\mathcal S}(X_1)=\gtq_1\cap{\mathcal S}^*(M)\cap{\mathcal S}(X_1)=\gtp_1\cap{\mathcal S}(X_1)=\gtP_1.
\end{equation}
Consequently, we have the following commutative diagram 
$$
\xymatrix{
{\mathcal S}(X_1)/\gtP_1\ar@{^{(}->}[r]\ar@{^{(}->}[d]&{\mathcal S}(X_2)/\gtQ_1\ar@{^{(}->}[d]\\
\qf({\mathcal S}(X_1)/\gtP_1)\ar@{^{(}->}[r]&\qf({\mathcal S}(X_2)/\gtQ_1).
}
$$
As $f_1\circ\pi\in\gtQ_1$ and $Z(f_1)=Y$, we deduce by \ref{sd}(iii)
\begin{multline*}
d-1={\tt d}_{X_1}(\gtP_1)=\tr\deg_\R(\qf({\mathcal S}(X_1)/\gtP_1))\leq\tr\deg_\R(\qf({\mathcal S}(X_2)/\gtQ_1))\\
={\tt d}_{X_2}(\gtQ_1)\leq\dim(Z_{X_2}(f_1\circ\pi))\leq\dim(X_2\setminus N)=d-1.
\end{multline*}
Thus, ${\tt d}_{X_1}(\gtP_1)={\tt d}_{X_2}(\gtQ_1)=d-1$.

\noindent{\bf Step 2.} As ${\tt k}_2(N)$ and $N$ are respectively dense in the $d$-dimensional semialgebraic sets $X_2$ and $X_1$ and $\dim(Y)=d-1$, the dimension of $\pi^{-1}(Y)$ is $d-1$. Consider the restriction map $\pi|_{\pi^{-1}(Y)}:\pi^{-1}(Y)\to Y$, which is surjective. By \cite[9.3.3]{bcr} there exists a closed semialgebraic subset $V$ of $Y$ of dimension $\dim(V)<\dim(Y)$ such that $\pi|_{\pi^{-1}(Y)}$ has a semialgebraic trivialization over each connected component of $Y\setminus V$. We may further assume that $Y\setminus V$ is pure dimensional and locally compact. In our case, the trivialization property means that for each connected component $Y_\ell$ of $Y\setminus V$ there exists a finite set $F_\ell$ and a semialgebraic homeomorphism $\theta_\ell:Y_\ell\times F_\ell\to(\pi|_{\pi^{-1}(Y)})^{-1}(Y_\ell)=:T_\ell$ such that $\pi|_{T_\ell}\circ\theta_\ell$ is the projection map $Y_\ell\times F_\ell\to Y_\ell$. Note that the connected components of $T_\ell$ are $\theta_\ell(Y_\ell\times\{p\})$ for $p\in F_\ell$ and that each connected component of $T_\ell$ is homeomorphic to $Y_\ell$. Define $\ol{Y\setminus V}:=\cl_{X_1}(Y\setminus V)$, $T:=\pi^{-1}(Y\setminus V)$ and $\ol{T}:=\cl_{X_2}(T)$. Notice that \em $T$ is locally compact\em. Indeed, as $Y\setminus V$ is locally compact and dense in $\ol{Y\setminus V}$, it is an open subset, so $T$ is an open subset of $\ol{T}$. As $\ol{T}$ is compact, $T$ is locally compact. 

\noindent{\bf Step 3.} Write $\gtp:=\gtp_i$ for $i=1,\ldots,r$ and let us prove: \em $\gtp\not\in\cl_{\Speca(M)}(V)$ and consequently $\gtp\in\cl_{\Speca(M)}(Y\setminus V)$\em.

Suppose first by contradiction that $\gtp\in\cl_{\Speca(M)}(V)$. By \ref{closure} we have a homeomorphism $\Speca(V)\cong\cl_{\Speca(M)}(V)$ induced by the inclusion ${\tt j}':V\hookrightarrow M$. Let $\gta$ be a minimal prime ideal of ${\mathcal S}(V)$ such that $\gta':=\Speca({\tt j}')(\gta\cap{\mathcal S}^*(V))\subset\gtp$. Note
$$
{\tt d}_M(\Specs({\tt j}')(\gta))={\tt d}_V(\gta)\leq\dim(V)\leq d-2.
$$ 
The subchain $\gtp_1=\gtm\cap{\mathcal S}^*(M)\subsetneq\cdots\subsetneq\gtp_r=\gtm^*$ is by \ref{chainp} the same for every non-refinable chain of prime ideals in ${\mathcal S}^*(M)$ ending at $\gtm^*$. As $\gtp_j\cap{\mathcal W}_M\neq\varnothing$ for $j\geq2$, we deduce $\gta'\subset\gtp_1$ because $\gta$ is a minimal prime ideal of ${\mathcal S}(V)$, so $\gta'\cap{\mathcal W}_M=\varnothing$. Thus, $\Specs({\tt j}')(\gta)=\gta'{\mathcal S}(M)$ satisfies
$$
d-1>{\tt d}_M(\Specs({\tt j}')(\gta))={\tt d}_M(\gta'{\mathcal S}(M))\geq{\tt d}_M(\gtp_1{\mathcal S}(M))=d-1, 
$$
which is a contradiction. Therefore $\gtp\not\in\cl_{\Speca(M)}(V)$.

As $Y=(Y\setminus V)\cup V$, we now have 
$$
\cl_{\Speca(M)}(Y)=\cl_{\Speca(M)}(Y\setminus V)\cup\cl_{\Speca(M)}(V);
$$
hence, $\gtp\in\cl_{\Speca(M)}(Y\setminus V)$. 

\noindent{\bf Step 4.} We claim: \em $\gtP_1\in\cl_{\Speca(X_1)}(\ol{Y\setminus V})$ and $\gtQ_1\in\cl_{\Speca(X_2)}(\ol{T})$. Consequently, 
$$
\gtP_i\in\cl_{\Speca(X_1)}(\ol{Y\setminus V})\quad\text{for $i=1,\ldots,r$}\quad\text{and}\quad \gtQ_j\in\cl_{\Speca(X_2)}(\ol{T})\quad\text{for $j=1,\ldots,s$.}
$$\em

First, as $\gtp_1\in\cl_{\Speca(M)}(Y\setminus V)$, we deduce by Theorem \ref{main}(ii)
$$
\gtP_1=\Speca({\tt k}_1)(\gtp_1)\in\Speca({\tt k}_1)(\cl_{\Speca(M)}(Y\setminus V))=\cl_{\Speca(X_1)}(\ol{Y\setminus V}).
$$

\paragraph{}\label{fp0}
Let $f\in\gtP_1$ be such that $Z_{X_1}(f)=\ol{Y\setminus V}$ and define $g:=f\circ\pi\in\gtQ_1$, which satisfies $Z_{X_2}(g)=\pi^{-1}(\ol{Y\setminus V})$. 

As $\gtQ_1$ is a prime $z$-ideal, we deduce by \ref{closure}(i) that $\gtQ_1\in\cl_{\Speca(X_2)}(\pi^{-1}(\ol{Y\setminus V}))$. On the other hand, let $C:=\ol{Y\setminus V}\setminus(Y\setminus V)$, which is a closed subset of $X_1$ because $Y\setminus V$ is locally compact. As $\pi^{-1}(\ol{Y\setminus V})=\pi^{-1}(Y\setminus V)\cup\pi^{-1}(C)=\ol{T}\cup\pi^{-1}(C)$, 
$$
\cl_{\Speca(X_2)}(\pi^{-1}(\ol{Y\setminus V}))=\cl_{\Speca(X_2)}(\ol{T})\cup\cl_{\Speca(X_2)}(\pi^{-1}(C)).
$$
Suppose $\gtQ_1\in\cl_{\Speca(X_2)}(\pi^{-1}(C))$. As $\Speca(\pi):\Speca(X_2)\to\Speca(X_1)$ is continuous,
$$
\gtP_1=\Speca(\pi)(\gtQ_1)\in\Speca(\pi)(\cl_{\Speca(X_2)}(\pi^{-1}(C)))\subset\cl_{\Speca(X_2)}(C).
$$
But this contradicts \ref{closure}(i) because $\dim(C)<\dim(Y\setminus V)=d-1$ and ${\tt d}_{X_1}(\gtP_1)=d-1$. Thus, $\gtQ_1\in\cl_{\Speca(X_2)}(\ol{T})$.

\noindent{\bf Step 5.} Consider the commutative diagram
$$
\xymatrix{
T\ar@{^{(}->}[rr]^{{\tt j}_2}\ar@{->>}[d]^{\pi|_T}&&\ol{T}\ar@{^{(}->}[r]^(0.475){{\tt j}_1}\ar@{->>}[d]^{\pi|_{\ol{T}}}&X_2\ar@{->>}[d]^\pi\\
Y\setminus V\ar@{^{(}->}[r]^(0.375){{\tt i}_3}&(\ol{Y\setminus V})\cap M\ar@{^{(}->}[r]^(0.575){{\tt i}_2}\ar@{^{(}->}[d]^{{\tt i}_0}&\ol{Y\setminus V}\ar@{^{(}->}[r]^(0.55){{\tt i}_1}&X_1\\
&Y\ar@{^{(}->}[r]^{{\tt i}}&M\ar@{^{(}->}[ru]^{{\tt k}_1}\\
}\\
$$
that induces the following commutative one
{\footnotesize$$
\xymatrix{
\Speca(T)\ar@{->>}[rr]\ar[d]&&\Speca(\ol{T})\ar@{<->}[r]^(0.45){\cong}\ar[d]&\cl_{\Speca(X_2)}(\ol{T})\ar@{^{(}->}[r]\ar[d]&\Speca(X_2)\ar[d]\\
\Speca(Y\setminus V)\ar@{->>}[r]&\Speca((\ol{Y\setminus V})\cap M)\ar@{->>}[r]\ar@{<->}[d]_\cong&\Speca(\ol{Y\setminus V})\ar@{<->}[r]^(0.45)\cong&\cl_{\Speca(X_1)}(\ol{Y\setminus V})\ar@{^{(}->}[r]&\Speca(X_1)\\
&\cl_{\Speca(M)}((\ol{Y\setminus V})\cap M)\ar@{->>}[rru]\ar@{^{(}->}[r]&\cl_{\Speca(M)}(Y)\ar@{^{(}->}[r]&\Speca(M)\ar@{->>}[ru]\\
}
$$}

\paragraph{} As $\gtp_i\in\cl_{\Speca(M)}(Y\setminus V)\subset\cl_{\Speca(M)}((\ol{Y\setminus V})\cap M)$ (see Step 3), there exists a unique prime ideal $\gtp_i'\in\Speca((\ol{Y\setminus V})\cap M)$ such that $\Speca({\tt i}\circ{\tt i}_0)(\gtp_i')=\gtp_i$. As the chain $\gtp_1\subsetneq\cdots\subsetneq\gtp_r$ is non-refinable, the same happens to the chain $\gtp_1'\subsetneq\cdots\subsetneq\gtp_r'$. As $\gtp_i\not\in\cl_{\Speca(M)}(V)$ and $\Speca({\tt i}\circ{\tt i}_0)$ is continuous, $\gtp_i'\not\in \cl_{\Speca(\ol{Y\setminus V}\cap M)}(\ol{Y\setminus V}\cap V)$. 

\paragraph{} It holds: $Z:=\cl_{\Speca(\ol{Y\setminus V}\cap M)}(\ol{Y\setminus V}\cap M)\setminus(Y\setminus V)=\cl_{\Speca(\ol{Y\setminus V}\cap M)}(\ol{Y\setminus V}\cap V)$. To prove this, we show $(\ol{Y\setminus V}\cap M)\setminus(Y\setminus V)=\ol{Y\setminus V}\cap V$.

Indeed,
\begin{multline*}
(\ol{Y\setminus V}\cap M)\setminus(Y\setminus V)=\cl_M(Y\setminus V)\setminus(Y\setminus V)=(\cl_M(Y\setminus V)\cap(M\setminus Y))\\
\cup(\cl_M(Y\setminus V)\cap V)
=(\cl_Y(Y\setminus V)\setminus Y)\cup(\ol{Y\setminus V}\cap V)=\ol{Y\setminus V}\cap V.
\end{multline*}

\paragraph{}By Theorem \ref{minq}(iv) there exists a chain of prime ideals $\gtp_1''\subsetneq\cdots\subsetneq\gtp_r''$ in $\Speca(Y\setminus V)$ such that $\Speca({\tt i_3})(\gtp_i'')=\gtp_i'$. By Theorem \ref{main}(iii) the restriction
$$
\Speca({\tt i}_3)|:\Speca(Y\setminus V)\setminus\Speca({\tt i}_3)^{-1}(Z)\to\Speca(\ol{Y\setminus V}\cap M)\setminus Z
$$
is a homeomorphism. As the chain $\gtp_1'\subsetneq\cdots\subsetneq\gtp_r'$ is non-refinable and each $\gtp_i'\not\in Z$, the chain $\gtp_1''\subsetneq\cdots\subsetneq\gtp_r''$ is non-refinable.

\paragraph{}Let $\gtP_i'$ be the (unique) prime ideal of $\Speca(\ol{Y\setminus V})$ that satisfies 
$$
\Speca({\tt i}_1)(\gtP_i')=\gtP_i\in\cl_{\Speca(X_1)}(\ol{Y\setminus V}). 
$$
We claim: $\Speca({\tt i}_2)(\gtp_i')=\gtP_i'$. 

Indeed, as ${\tt i}_1\circ{\tt i}_2={\tt k}_1\circ{\tt i}\circ{\tt i}_0$, we have 
$$
\Speca({\tt i}_1)(\Speca({\tt i}_2)(\gtp_i'))=\Speca({\tt k}_1)(\Speca({\tt i}\circ{\tt i}_0)(\gtp_i'))=\Speca({\tt k}_1)(\gtp_i)=\gtP_i.
$$
Consequently, $\Speca({\tt i}_2)(\gtp_i')=\gtP_i'$.

\paragraph{}\label{minimalp0}Let us check now: $\Speca({\tt i}_2\circ{\tt i}_3)^{-1}(\gtP_1')=\{\gtp_1''\}$. To that end, we show first: \em $\gtP_1'$ is a minimal prime ideal of ${\mathcal S}(\ol{Y\setminus V})$\em. Once this is proved, its fiber under $\Speca({\tt i}_2\circ{\tt i}_3)$ is by Theorem \ref{minq}(i) a singleton; hence, $\Speca({\tt i}_2\circ{\tt i}_3)^{-1}(\gtP_1')=\{\gtp_1''\}$ because $\Speca({\tt i}_2\circ{\tt i}_3)(\gtp_1'')=\Speca({\tt i}_2)(\gtp_1')=\gtP_1'$.

Indeed, as $\ol{Y\setminus V}$ is pure dimensional, it is enough to show by Theorem \ref{minimalnlc} that ${\tt d}_{\ol{Y\setminus V}}(\gtP_1')=d-1$. Since the homomorphism ${\mathcal S}(X)\to{\mathcal S}(\ol{Y\setminus V})$ is surjective, $f\in\gtP_1$ satisfies $Z_{X_1}(f)=\ol{Y\setminus V}$ (see \ref{fp0}) and $\Speca({\tt i}_1)(\gtP_1')=\gtP_1$, it holds ${\tt d}_{\ol{Y\setminus V}}(\gtP_1')={\tt d}_{X_1}(\gtP_1)=d-1$.

\paragraph{} As $\Speca({\tt j})$ maps the chain $\gtq_1\subsetneq\cdots\subsetneq\gtq_s$ onto the chain $\gtp_1\subsetneq\cdots\subsetneq\gtp_r$, $\Speca(\pi)(\gtQ_1)=\gtP_1$ (by \eqref{p0q0}) and the following diagrams are commutative
\vspace*{-5mm}
$$
\xymatrix{
N\ar@{^{(}->}[r]\ar@{^{(}->}[d]&X_2\ar@{->>}[d]\\
M\ar@{^{(}->}[r]&X_1
}\quad\begin{array}{c} \\[20pt]\leadsto\end{array}\quad
\xymatrix{
\Speca(N)\ar@{->>}[r]\ar@{->>}[d]\ar@{->>}[rd]&\Speca(X_2)\ar@{->>}[d]\\
\Speca(M)\ar@{->>}[r]&\Speca(X_1),
}
$$

\noindent we conclude that $\Speca(\pi)$ maps the chain $\gtQ_1\subsetneq\cdots\subsetneq\gtQ_s$ onto the chain $\gtP_1\subsetneq\cdots\subsetneq\gtP_r$. Let $\gtQ_j'$ be the unique prime ideal of $\Speca(\ol{T})$ such that $\Speca({\tt j}_1)(\gtQ_j')=\gtQ_j\in\cl_{\Speca(X_2)}(\ol{T})$. Notice that $\Speca(\pi|_{\ol{T}})$ maps the chain $\gtQ_1'\subsetneq\cdots\subsetneq\gtQ_s'$ onto the chain $\gtP_1'\subsetneq\cdots\subsetneq\gtP_r'$. 

\paragraph{}\label{q0p0}
By Theorem \ref{minq}(iv) there exists a chain of prime ideals $\gtq_1''\subsetneq\cdots\subsetneq\gtq_s''$ in $\Speca(T)$ such that $\Speca({\tt j_2})(\gtq_j'')=\gtQ_j'$. We claim: \em $\Speca(\pi|_T)$ maps the chain $\gtq_1''\subsetneq\cdots\subsetneq\gtq_s''$ onto the chain $\gtp_1''\subsetneq\cdots\subsetneq\gtp_r''$\em. As the chain $\gtp_1''\subsetneq\cdots\subsetneq\gtp_r''$ is non-refinable, it is enough to show by Theorem \ref{minq}(iv) that $\Speca(\pi|_T)(\gtq_1'')=\gtp_1''$.

Since the subdiagram
$$
\xymatrix{
\Speca(T)\ar@{->>}[rr]^{\Speca({\tt j}_2)}\ar[d]_{\Speca(\pi|_{T})}&&\Speca(\ol{T})\ar[d]^{\Speca(\pi|_{\ol{T}})}\\
\Speca(Y\setminus V)\ar@{->>}[rr]^{\Speca({\tt i}_2\circ{\tt i}_3)}&&\Speca(\ol{Y\setminus V})
}
$$
is commutative, we have
$$
\Speca({\tt i}_2\circ{\tt i}_3)(\Speca(\pi|_T)(\gtq_1''))=\Speca(\pi|_{\ol{T}})(\Speca({\tt j}_2)(\gtq_1''))=\Speca(\pi|_{\ol{T}})(\gtQ_1')=\gtP_1'
$$ 
and by \ref{minimalp0} $\Speca(\pi|_T)(\gtq_1'')=\gtp_1''$, as required.

\paragraph{}For the sake of clearness let us summarize all the previous information:
{\tiny$$
\xymatrix{
\Speca(T)\ar@{->>}[rr]\ar[d]&&\Speca(\ol{T})\ar@{^{(}->}[r]\ar[d]&\Speca(X_2)\ar[d]\\
\Speca(Y\setminus V)\ar@{->>}[r]&\Speca((\ol{Y\setminus V})\cap M)\ar@{->>}[r]\ar@{^{(}->}[d]&\Speca(\ol{Y\setminus V})\ar@{^{(}->}[r]&\Speca(X_1)\\
&\Speca(M)\ar@{->>}[rru]\\
}
\quad\xymatrix{\gtq_j''\ar@{|->}[d]\ar@{|->}[rr]&&\gtQ_j'\ar@{|->}[r]\ar@{|->}[d]&\gtQ_j\ar@{|->}[d]\\
\gtp_i''\ar@{|->}[r]&\gtp_i'\ar@{|->}[d]\ar@{|->}[r]&\gtP_i'\ar@{|->}[r]&\gtP_i\\
&\gtp_i\ar@{|->}[rru]}
$$}

\noindent{\bf Step 6.} Let $S_1,\ldots,S_\ell$ be the connected components of $T$. By \ref{closure}(iv) the connected components of $\Speca(T)$ are $\cl_{\Speca(T)}(S_k)\cong\Speca(S_k)$ for $k=1,\ldots,\ell$. We may assume $\gtq_1''\in\cl_{\Speca(T)}(S_1)\cong\Speca(S_1)$ and $\pi(S_1)=Y_1$; hence, as $\Speca(\pi|_T)(\gtq_1'')=\gtp_1''$ (see \ref{q0p0}) and $\Speca(\pi|_T)$ is continuous, it holds $\gtp_1''\in\cl_{\Speca(Y\setminus V)}(Y_1)\cong\Speca(Y_1)$. As we have proved in Step 2, the map $\pi|_{S_1}:S_1\to Y_1$ is a semialgebraic homeomorphism; hence, 
$$
\Speca(\pi|_{S_1}):\Speca(S_1)\to\Speca(Y_1)
$$
is a homeomorphism. Thus, the restriction map
$$
\Speca(\pi|_T)|_{\cl_{\Speca(T)}(S_1)}:\cl_{\Speca(T)}(S_1)\to\cl_{\Speca(Y\setminus V)}(Y_1)
$$
is also a homeomorphism. In particular, $\Speca(\pi|_T)|_{\cl_{\Speca(T)}(S_1)}$ maps the chain $\gtq_1''\subsetneq\cdots\subsetneq\gtq_s''$ bijectively onto the chain $\gtp_1''\subsetneq\cdots\subsetneq\gtp_r''$, so $r=s$, as required.
\end{proof}

\begin{proof}[\sl Proof of Lemma \ref{fiberspec3}]
Note first that $\gtp$ is not a minimal prime ideal of ${\mathcal S}^*(M)$. Otherwise, since $\gtp\cap{\mathcal W}_M=\varnothing$, $\gtP:=\gtp{\mathcal S}(M)$ would be a minimal prime ideal of ${\mathcal S}(M)$. As $M$ is pure dimensional of dimension $d$, it follows from Theorem \ref{minimalnlc} that ${\tt d}_M(\gtP)=d$, against the hypotheses.
 
\paragraph{}Now we prove: \em $\gtP:=\gtp{\mathcal S}(M)$ is a $z$-ideal\em. 

Indeed, by \ref{sd}(i) there exists a (unique) prime $z$-ideal $\gtP^z$ of ${\mathcal S}(M)$ such that $\gtP\subset\gtP^z$ and ${\tt d}_M(\gtP^z)={\tt d}_M(\gtP)=d-1$. By assumption $\gtp$ contains only one minimal prime ideal $\gta$ of ${\mathcal S}^*(M)$ properly. Since $\gtp\cap{\mathcal W}_M=\varnothing$, also $\gtP$ contains by \ref{inequality} a unique minimal prime ideal $\gtA$ of ${\mathcal S}(M)$ properly. By \ref{zideal3} $\gtA$ is a $z$-ideal, so by Theorem \ref{minimalnlc} ${\tt d}_M(\gtA)=d$. As ${\tt d}_M(\gtP^z)=d-1$, by \ref{sd}(ii) there does not exist any prime ideal between $\gtA$ and $\gtP^z$; hence, $\gtP=\gtP^z$ is a prime $z$-ideal.

\paragraph{}As ${\mathcal S}(M)={\mathcal S}^*(M)_{{\mathcal W}_M}$, the quotient fields $\qf({\mathcal S}(M)/\gtP)$ and $\qf({\mathcal S}^*(M)/\gtp)$ are equal. Thus, by \ref{sd}(iii)
$$
\tr\deg_\R(\qf({\mathcal S}^*(M)/\gtp))=\tr\deg_\R(\qf({\mathcal S}(M)/\gtP))={\tt d}_M(\gtP)=d-1.
$$
On the other hand, if $\gtq$ is a prime ideal of ${\mathcal S}^*(N)$ such that $\gtp=\gtq\cap{\mathcal S}^*(M)=\Speca({\tt j})(\gtq)$, it holds: $\tr\deg_\R(\qf({\mathcal S}^*(N)/\gtq))=d-1$.

Indeed, we have the inclusions
$$
{\mathcal S}^*(M)/\gtp\hookrightarrow{\mathcal S}^*(N)/\gtq\quad\leadsto\quad\qf({\mathcal S}^*(M)/\gtp)\hookrightarrow\qf({\mathcal S}^*(N)/\gtq).
$$
Thus, by \ref{sd}(iii) 
$$
d-1=\tr\deg_\R(\qf({\mathcal S}^*(M)/\gtp))\leq\tr\deg_\R(\qf({\mathcal S}^*(N)/\gtq))\leq\dim(N)=d.
$$
Suppose by contradiction that $\tr\deg_\R(\qf({\mathcal S}^*(N)/\gtq))=d$. Then $\gtq$ is a minimal prime ideal of ${\mathcal S}^*(N)$ and by Theorem \ref{minq}(iii) $\gtp=\Speca({\tt j})(\gtq)$ is a minimal prime ideal of ${\mathcal S}^*(M)$, which is a contradiction. Thus, $\tr\deg_\R(\qf({\mathcal S}^*(N)/\gtq))=d-1$.

\paragraph{}Finally we prove: \em $\Speca({\tt j})^{-1}(\gtp)$ is a singleton\em. 

Suppose by contradiction that $\Speca({\tt j})^{-1}(\gtp)$ is not a singleton. Then there exist two distinct prime ideals $\gtq_1,\gtq_2\in\Speca(N)$ such that $\Speca({\tt j})(\gtq_i)=\gtp$. In particular $\tr\deg_\R(\qf({\mathcal S}^*(N)/\gtq_i))=d-1$. Let $\gtb_i\subset\gtq_i$ be a minimal prime ideal of ${\mathcal S}^*(N)$. By Theorem \ref{minq}(iii) $\Speca({\tt j})(\gtb_i)$ is a minimal prime ideal of ${\mathcal S}^*(M)$ contained in $\gtp$; hence, $\Speca({\tt j})(\gtb_i)=\gta$. By Theorem \ref{minq}(i) the fiber $\Speca({\tt j})^{-1}(\gta)$ is a singleton, so $\gtb_1=\gtb_2$. Thus, we may assume by \ref{chainp} that $\gtb_1\subset\gtq_1\subsetneq\gtq_2$. 

By \ref{longitud} and \ref{brimming} there exists a brimming semialgebraic compactification $(X,{\tt k})$ of $N$ such that $\gtq_1\cap{\mathcal S}(X)\subsetneq\gtq_2\cap{\mathcal S}(X)$ and
$$
\qf({\mathcal S}(X)/(\gtq_i\cap{\mathcal S}(X)))=\qf({\mathcal S}^*(N)/\gtq_i)
$$
for $i=1,2$. By \ref{sd}(ii) we deduce, as $X$ is locally compact, that
\begin{multline*}
d-1=\tr\deg_\R(\qf({\mathcal S}^*(N)/\gtq_1))={\tt d}_X(\gtq_1\cap{\mathcal S}(X))\\
>{\tt d}_X(\gtq_2\cap{\mathcal S}(X))=\tr\deg_\R(\qf({\mathcal S}^*(N)/\gtq_2))=d-1,
\end{multline*}
which is a contradiction, as required.
\end{proof}

\subsection{Proof of the quantitative part of Theorem \ref{fiberspec} for singleton fibers}\label{fiberspec3b}\setcounter{paragraph}{0} 
Our purpose here is to prove the following: \em Let $\gtp\in\cl_{\Speca(M)}(Y)\setminus\cl_{\Speca(M)}(\rho_1(M))$ be a prime ideal such that ${\tt d}_M(\widehat{\gtp}{\mathcal S}(M))=d-1$. Then the fiber $\Speca({\tt j})^{-1}(\gtp)$ is a singleton if and only if $\widehat{\gtp}$ contains exactly one minimal prime ideal of ${\mathcal S}^*(M)$\em.

\begin{proof}[\sl Proof of the quantitative part of Theorem \em\ref{fiberspec} \em for singleton fibers]
Assume first that the threshold $\widehat{\gtp}$ contains only one minimal prime ideal. As $\widehat{\gtp}\cap{\mathcal E}_M=\varnothing$, also $\widehat{\gtp}\cap{\mathcal W}_M=\varnothing$. By Lemma \ref{e} $\widehat{\gtp}\in\cl_{\Speca(M)}(Y)$, so by Lemma \ref{fiberspec3} $\Speca({\tt j})^{-1}(\widehat{\gtp})$ is a singleton. Let us check that also $\Speca({\tt j})^{-1}(\gtp)$ is a singleton. If $\gtp\cap{\mathcal E}_M=\varnothing$, it holds $\widehat{\gtp}=\gtp$ and we are done, so we assume $\gtp\cap{\mathcal E}_M\neq\varnothing$. 

\paragraph{}\label{hypo}$\hspace{-1.5mm}$\em We may assume $\gtp\cap{\mathcal W}_M\neq\varnothing$ and $\widehat{\gtp}=\gtm\cap{\mathcal S}^*(M)$ for some maximal ideal $\gtm$ of ${\mathcal S}(M)$\em.

Indeed, if $\gtp\cap{\mathcal W}_M\neq\varnothing$, there is nothing to prove, so we assume $\gtp\cap{\mathcal W}_M=\varnothing$. Consider the suitably arranged {\tt sa}-tuples $(M,M_{\lc},\rho_1(M),{\tt j}_1,{\tt i}_1)$ and $(M_{\lc},N,M_{\lc}\setminus N,{\tt j}_2,{\tt i}_2)$. Denote $Z:=\cl_{\Speca(M)}(\rho_1(M))$ and recall that by Theorem \ref{main}(iii) the restriction map
\begin{equation}\label{homeo88}
\Speca({\tt j}_1)|:\Speca(M_{\lc})\setminus\Speca({\tt j}_1)^{-1}(Z)\to\Speca(M)\setminus Z
\end{equation}
is a homeomorphism. As $\gtp\not\in Z$ (by hypothesis), the fiber $\Speca({\tt j}_1)^{-1}(\gtp)$ is a singleton whose unique element is denoted by $\gtp_1\in\Speca(M_{\lc})\setminus\Speca({\tt j}_1)^{-1}(Z)$. As ${\tt j}={\tt j}_1\circ{\tt j}_2$, the sizes of $\Speca({\tt j})^{-1}(\gtp)$ and $\Speca({\tt j}_2)^{-1}(\gtp_1)$ coincide. As $\widehat{\gtp}\subset\gtp$ and $\gtp\not\in Z$, we deduce $\widehat{\gtp}\not\in Z$, so $\Speca({\tt j}')^{-1}(\widehat{\gtp})$ is a singleton whose unique element $\gtq\in\Speca(M_{\lc})\setminus\Speca({\tt j}_1)^{-1}(Z)$. On the other hand, as ${\mathcal E}_M\subset{\mathcal W}_{M_{\lc}}$ and $\gtp_1\cap{\mathcal S}^*(M)=\gtp$, we have $\gtp_1\cap{\mathcal W}_{M_{\lc}}\neq\varnothing$. By Lemma \ref{e}(i) there exists a maximal ideal $\gtm_1$ of ${\mathcal S}(M_{\lc})$ such that $\gtm_1\cap{\mathcal S}^*(M_{\lc})\subset\gtp_1\subset\gtm_1^*$ and $\widehat{\gtp}=\gtm_1\cap{\mathcal S}^*(M)$; hence, $\gtq=\gtm_1\cap{\mathcal S}(M_{\lc})=\widehat{\gtp}_1$. Notice that $\widehat{\gtp}_1$ contains only one minimal prime ideal of ${\mathcal S}^*(M_{\lc})$ because \eqref{homeo88} is a homeomorphism, $\widehat{\gtp}\not\in Z$ and $\widehat{\gtp}$ contains only one minimal prime ideal of ${\mathcal S}^*(M)$. In addition, it holds ${\tt d}_{M_{\lc}}(\widehat{\gtp}_1{\mathcal S}(M_{\lc}))={\tt d}_{M_{\lc}}(\gtm_1)={\tt d}_M(\widehat{\gtp}{\mathcal S}(M))=d-1$ because $\widehat{\gtp}{\mathcal S}(M)=\gtm_1\cap{\mathcal S}(M)$, $\widehat{\gtp}_1{\mathcal S}(M_{\lc})=\gtm_1$ and $\dim(\rho_1(M))\leq d-2$ (see Corollary \ref{rho2}). 

Thus, substituting $M$ by $M_{\lc}$ and $\gtp$ by $\gtp_1$, we are under the hypotheses of \ref{hypo} (together with those in the statement) and the sizes of $\Speca({\tt j})^{-1}(\gtp)$ and $\Speca({\tt j}_2)^{-1}(\gtp_1)$ coincide.

\paragraph{} Assume in the following $\gtp\cap{\mathcal W}_M\neq\varnothing$ and $\widehat{\gtp}=\gtm\cap{\mathcal S}^*(M)$ for some maximal ideal $\gtm$ of ${\mathcal S}(M)$. As $\widehat{\gtp}$ contains only one minimal prime ideal, also $\gtm$ contains only one minimal prime ideal that we denote with $\gta_0$. In particular, as $M$ is pure dimensional, it holds by Theorem \ref{minimalnlc} that ${\tt d}_M(\gta_0)=d$. As ${\tt d}_M(\gtm)=d-1$, we deduce by \ref{sd}(ii) that there does not exist any prime ideal between $\gta_0$ and $\gtm$. By \ref{inequality} $\gtp_0:=\gta_0\cap{\mathcal S}^*(M)$ is the unique minimal ideal of ${\mathcal S}^*(M)$ contained in $\widehat{\gtp}$. By \ref{chainp}(ii) we conclude that the collection of all prime ideals of ${\mathcal S}^*(M)$ containing $\gtp_0$ is
\begin{equation}\label{eschain}
\gtp_0\subsetneq\gtp_1:=\widehat{\gtp}=\gtm\cap{\mathcal S}^*(M)\subsetneq\cdots\subsetneq\gtp_\ell:=\gtp\subsetneq\cdots\subsetneq\gtp_r:=\gtm^*.
\end{equation}
It follows from Theorem \ref{minq}(i) that there exists only one minimal prime ideal $\gtq_0$ of ${\mathcal S}^*(N)$ such that 
$$
\Speca({\tt j})(\gtq_0)=\gtp_0. 
$$
Let $\gtq_0\subsetneq\gtq_1\subsetneq\cdots\subsetneq\gtq_s$ be the collection of all prime ideals of ${\mathcal S}^*(N)$ that contain $\gtq_0$. Observe $\Speca({\tt j})(\gtq_1)=\gtp_1$ and by Lemma \ref{pmchain} we conclude $r=s$. Summarizing, we conclude by Theorem \ref{minq} that the fibers of all prime ideals in the chain \eqref{eschain} are singletons; hence, in particular, $\Speca({\tt j})^{-1}(\gtp)=\{\gtq_\ell\}$ is a singleton.

\paragraph{}Assume next that the fiber $\Speca({\tt j})^{-1}(\gtp)$ is a singleton. Let us prove that $\widehat{\gtp}$ contains a unique minimal prime ideal of ${\mathcal S}^*(M)$. As $\widehat{\gtp}\cap{\mathcal W}_M=\varnothing$, by \ref{inequality} it is enough to check that the prime ideal $\widehat{\gtP}=\widehat{\gtp}{\mathcal S}(M)$ of ${\mathcal S}(M)$ contains only one minimal prime ideal of ${\mathcal S}(M)$. Suppose by contradiction that $\widehat{\gtP}$ contains two different minimal prime ideals $\gtQ_1$ and $\gtQ_2$ of ${\mathcal S}(M)$. Fix $g\in{\mathcal S}(M)$ such that $Z(g)=Y$. As $\widehat{\gtP}\in\cl_{\Specs(M)}(Y)$, we have $g\in\widehat{\gtP}$ and by Lemma \ref{intcomp1} there exist $f_i\in\gtQ_i\setminus\gtQ_j$ if $i\neq j$ such that $Z(f_1^2+f_2^2)\subset Z(g)$. Define $Z_i:=Z(f_i)$ and $N_i:=Z_i\cap N$. As $N$ is locally compact, so are $N_1$ and $N_2$. Moreover, $N_1$ and $N_2$ are disjoint because $Z_1\cap Z_2\cap N\subset Y\cap N=\varnothing$. By \ref{closure}(ii) $\cl_{\Speca(N)}(N_1)\cap\cl_{\Speca(N)}(N_2)=\varnothing$.

Let ${\tt j}_i:N_i\hookrightarrow Z_i$ be the inclusion map. By Theorem \ref{minq}(iii) $\Speca({\tt j}_i):\Speca(N_i)\to\Speca(Z_i)$ is surjective. Thus, after identifying $\Speca(N_i)\equiv\cl_{\Speca(N)}(N_i)$ and $\Speca(Z_i)\equiv\cl_{\Speca(M)}(Z_i)$, the map $\Speca({\tt j})|:\cl_{\Speca(N)}(N_i)\to\cl_{\Speca(M)}(Z_i)$ is surjective.
 
Observe that $\gtq_i:=\gtQ_i\cap{\mathcal S}^*(M)\in\cl_{\Speca(M)}(Z_i)$ because $\gtQ_i\in\cl_{\Specs(M)}(Z_i)$. As 
$$
\gtp\in\cl_{\Speca(M)}(\{\widehat{\gtp}\})\subset\cl_{\Speca(M)}(\{\gtq_1\})\cap\cl_{\Speca(M)}(\{\gtq_2\}),
$$
we conclude $\gtp\in\cl_{\Speca(M)}(Z_1)\cap\cl_{\Speca(M)}(Z_2)$.
Thus, there exists $\gtb_i\in\cl_{\Speca(N)}(N_i)$ such that 
$$
\Speca({\tt j})(\gtb_i)=\gtp
$$ 
and $\gtb_1\neq\gtb_2$ because $\cl_{\Speca(N)}(N_1)\cap\cl_{\Speca(N)}(N_2)=\varnothing$. Consequently, $\Speca{j}^{-1}(\gtp)$ is not a singleton, which is a contradiction. Thus, $\widehat{\gtP}$ contains only one minimal prime ideal of ${\mathcal S}(M)$, as required.
\end{proof}

\subsection{Proof of the remaining part of Theorem \ref{fiberspec}}\label{s}\setcounter{paragraph}{0} 
Our purpose here is to prove: \em If $\gtp\in\cl_{\Speca(M)}(Y)\setminus\cl_{\Speca(M)}(\rho_1(M))$ and ${\tt d}_M(\widehat{\gtp}{\mathcal S}(M))=d-1$, then the fiber of $\gtp$ is a finite set and whose size equals the number of minimal prime ideals of ${\mathcal S}^*(M)$ contained in $\widehat{\gtp}$\em.

\begin{proof}[\sl Proof of the remaining part of Theorem \em\ref{fiberspec}]
We may assume that $M$ is bounded and denote $X:=\cl_{\R^m}(M)$. By Theorem \cite[9.2.1]{bcr} applied to $X$ and the family of semialgebraic sets ${\mathcal F}:=\{M,N,Y\}$ there exists a semialgebraic triangulation $(K,\Phi)$ of $X$ compatible with ${\mathcal F}$. For simplicity we identify all involved objects with their images under $\Phi^{-1}$ and denote $\widehat{\gtP}:=\widehat{\gtp}{\mathcal S}(M)$, which is a proper prime ideal of ${\mathcal S}(M)$ because $\widehat{\gtp}\cap{\mathcal W}_M=\varnothing$.

\paragraph{}Let $\tau_1^0,\ldots,\tau_r^0$ be all simplices of $K$ contained in $Y$. We know by the compatibility property of the semialgebraic triangulation $(K,\Phi)$ that $Y=\bigcup_{i=1}^r\tau_i^0$. Let $T_i:=\cl_M(\tau_i^0)=\tau_i\cap Y$ and $h_i\in{\mathcal S}^*(M)$ be such that $Z(h_i)=T_i$. As the zero set of $h:=\prod_{i=1}^rh_i$ equals $Y$ and $\widehat{\gtp}\in\cl_{\Speca(M)}(Y)$, we may assume $h_1\in\widehat{\gtp}$ and write $T:=T_1$ and $\tau:=\tau_1$. Note in particular that $\widehat{\gtP}\in\cl_{\Specs(M)}(T)$ as $\widehat{\gtP}$ is a prime $z$-ideal; hence, $\widehat{\gtp}\in\cl_{\Speca(M)}(T)$. 

As ${\tt d}_M(\widehat{\gtP})=d-1$, we deduce $\dim(T)=\dim(Z(g_1))=d-1$. Let $\sigma_1,\ldots,\sigma_s\in K$ be the collection of all simplices of dimension $d$ that contains the $(d-1)$-dimensional simplex $\tau$; clearly, $\sigma_i\cap\sigma_j=\tau$ if $i\neq j$. Denote $M_i:=\sigma_i\cap M$ and observe $M_i\cap M_j=T=\tau\cap Y$ if $i\neq j$ and $\widehat{\gtp}\in\cl_{\Speca(M)}(M_i)$. 

The semialgebraic set $U:=\bigcup_{i=1}^s\sigma_i^0\cup\tau^0$ is an open neighborhood of $\tau^0$ in $M$ (as it is the star of $\tau^0$). Thus, $M_0:=M\setminus U$ is a closed semialgebraic subset of $M$ that satisfies
$$
M_0\cap T=(M\setminus U)\cap T\subset T\setminus\tau^0\subset\tau\setminus\tau^0,
$$
which has dimension $<\dim(T)=d-1$. As $\widehat{\gtP}\in\cl_{\Specs(M)}(T)$ and ${\tt d}_M(\widehat{\gtP})=d-1$, we deduce $\widehat{\gtP}\not\in\cl_{\Specs(M)}(M_0)$ because otherwise by \ref{closure}(iii) 
$$
\widehat{\gtP}\in\cl_{\Specs(M)}(M_0)\cap\cl_{\Specs(M)}(T)=\cl_{\Specs(M)}(M_0\cap T),
$$
so $d-1={\tt d}_M(\widehat{\gtP})\leq\dim(M_0\cap T)\leq d-2$, which is a contradiction. Thus, $\widehat{\gtp}\not\in\cl_{\Speca(M)}(M_0)$ and by Lemma \ref{e}(ii) we deduce $\gtp\not\in\cl_{\Speca(M)}(M_0)$. 

\paragraph{}Write $N_i:=M_i\setminus Y$ for $i=1,\ldots,s$ and notice that $N_i$ is dense in $M_i$ and $N_i\cap N_j=\varnothing$ if $i\neq j$. As $N_i$ is closed in $N$, each $N_i$ is locally compact. If we denote the inclusions with ${\tt j}_i:N_i\hookrightarrow M_i$ and ${\tt i}_i:Y_i:=M_i\setminus N_i\hookrightarrow M_i$, we get a suitably arranged {\tt sa}-tuple $(N_i,M_i,Y_i,{\tt j}_i,{\tt i}_i)$.

By Theorem \ref{main}(ii) it holds $\Speca({\tt j})(\cl_{\Speca(N)}(N_i))=\cl_{\Speca(M)}(M_i)$. By \ref{closure}(ii) 
$$
\Specd(N_i)\cong\cl_{\Specd(N)}(N_i)\quad\text{and}\quad\Specd(M_i)\cong\cl_{\Specd(M)}(M_i)
$$ 
via the inclusions ${\tt k}_i:N_i\hookrightarrow N$ and ${\tt l}_i:M_i\hookrightarrow M$. Denote the unique prime ideal of ${\mathcal S}^*(M_i)$ whose image under $\Speca({\tt l}_i)$ is $\gtp$ with $\gtp_i$ for $i=1,\ldots,s$. 

As the semialgebraic sets $N_i$ are pairwise disjoint closed connected subsets of $N$, the connected components of $\bigsqcup_{i=1}^sN_i$ are $N_1,\ldots,N_s$. By \ref{closure} \em the sets $\cl_{\Speca(N)}(N_i)$ are the connected components of 
$$
\cl_{\Speca(N)}\Big(\bigsqcup_{i=1}^sN_i\Big)
$$ 
\em and in particular they are disjoint.

As $\Speca(M)=\bigcup_{i=0}^s\cl_{\Speca(M)}(M_i)$ and $\gtp\not\in\cl_{\Speca(M)}(M_0)$, it holds $$
\gtp\in\bigcup_{i=1}^s\cl_{\Speca(M)}(M_i)\quad\text{and so}\quad\Speca({\tt j})^{-1}(\gtp)\subset\bigsqcup_{i=1}^s\cl_{\Speca(N)}(N_i).
$$
Consequently, the size of the fiber $\Speca({\tt j})^{-1}(\gtp)$ coincides with the sum of the sizes of the fibers 
$$
\Speca({\tt j}_i)^{-1}(\gtp_i)
$$ 
for $i=1,\ldots,s$. Denote the unique prime ideal of ${\mathcal S}(M_i)$ whose image under $\Speca({\tt l}_i)$ is $\widehat{\gtp}$ with $\widehat{\gtp}_i$ for $i=1,\ldots,s$. As $\widehat{\gtp}\cap{\mathcal W}_M=\varnothing$ and $\widehat{\gtp}\in\cl_{\Speca(M)}(M_i)$, one can check that $\widehat{\gtp}_i\cap{\mathcal W}_{M_i}=\varnothing$. By Lemma \ref{main0}(ii) $\{\widehat{\gtP}_i:=\widehat{\gtp}_i{\mathcal S}(M_i)\}=\Specs({\tt l}_i)^{-1}(\widehat{\gtP})$ and, as $\widehat{\gtP}\in\cl_{\Specs(M)}(M_i)$, it holds ${\tt d}_{M_i}(\widehat{\gtP}_i)={\tt d}_M(\widehat{\gtP})=d-1$. In addition, as $\widehat{\gtP}\in\cl_{\Speca(M)}(T)$ and $T\subset M_i$, we deduce $\widehat{\gtP}_i\in\cl_{\Speca(M_i)}(T)$.

\paragraph{}\label{cardcrucial}We claim: \em $\widehat{\gtp}_i$ contains exactly one minimal prime ideal of ${\mathcal S}^*(M_i)$\em. 

Indeed, fix $i=1,\ldots,s$ and suppose that there are two minimal prime ideals $\gta_1,\gta_2$ of ${\mathcal S}^*(M_i)$ contained in $\widehat{\gtp}_i$. By Theorem \ref{minimalnlc} ${\tt d}_{M_i}(\gta_j{\mathcal S}(M_i))=d$ for $j=1,2$. Let $g\in\widehat{\gtp}_i$ be such that $Z_{M_i}(g)=T$. By Lemma \ref{intcomp1} there exists $f_j\in\gta_j\setminus\gta_k$ if $j\neq k$ such that $Z_{M_i}(f_1^2+f_2^2)\subset Z_{M_i}(g)$ and $Z_{M_i}(f_j)$ is pure dimensional for $j=1,2$. Substituting $g$ with $g^2+f_1^2+f_2^2$, we may assume $Z_{M_i}(f_1^2+f_2^2)=Z_{M_i}(g)\subset T$. Note that $\dim(Z_{M_i}(g))=d-1$ because ${\tt d}_{M_i}(\widehat{\gtP}_i)=d-1$.

Let $(K_i,\Phi_i)$ be a semialgebraic triangulation of $\sigma_i$ compatible with all its faces, $Z_{M_i}(f_1)$ and $Z_{M_i}(f_2)$. Let $C$ be the image under $\Phi_i$ of the union of all simplices of $K_i$ of dimension $d-1$ contained in $\sigma_i\setminus T$ and all simplices of dimension $\leq d-2$. As $Z_{M_i}(f_j)$ is pure dimensional of dimension $d$, we conclude $Z_{M_i}(f_j)\setminus C=\bigcup_{\ell=1}^mS_\ell$ where $S_\ell$ is either 
\begin{itemize}
\item[(1)] the image under $\Phi_i$ of an open simplex of dimension $d$ or 
\item[(2)] the image under $\Phi_i$ of the union of an open simplex $\upsilon$ of dimension $d$ and an open simplex $\epsilon$ of dimension $d-1$ adherent to $\upsilon$ and contained in $\Phi_i^{-1}(\tau^0)$.
\end{itemize}
Thus, $Z_{M_i}(f_j)\setminus C$ is an open subset of $\sigma_i$. In particular, 
$$
Z_{M_i}(g)\setminus C=Z_{M_i}(f_1^2+f_2^2)\setminus C=(Z_{M_i}(f_1)\setminus C)\cap(Z_{M_i}(f_2)\setminus C)
$$ 
is an open subset of $\sigma_i$. As $\dim(Z_{M_i}(g))=d-1$ and $\dim(\sigma_i)=d$, we deduce $Z_{M_i}(g)\setminus C=\varnothing$, so $Z_{M_i}(g)\subset T\cap C$. But this is impossible because $T\cap C$ has dimension $\leq d-2$ and $Z_{M_i}(g)$ has dimension $d-1$. We conclude that $\widehat{\gtp}_i$ contains only one minimal prime ideal, as required.

\paragraph{}$\hspace{-1.5mm}$ \em There exist exactly $s$ minimal prime ideals of ${\mathcal S}(M)$ contained in $\widehat{\gtp}$\em.

Let $\gta_i$ be the unique minimal prime of ${\mathcal S}^*(M_i)$ such that $\gta_i\subset\widehat{\gtp}_i$. It holds $\gtq_i:=\Speca({\tt l}_i)(\gta_i)\subsetneq\widehat{\gtp}$, so $\gta_i\cap{\mathcal W}_M=\varnothing$. As $\gta_i{\mathcal S}(M_i)$ is a minimal prime ideal of ${\mathcal S}(M_i)$, it is a $z$-ideal, so $\gtQ_i:=\gtq_i{\mathcal S}(M)$ is by \ref{zideal3} also a $z$-ideal. Consequently, by \ref{sd}(ii) $d\geq{\tt d}_M(\gtQ_i)>{\tt d}_M(\widehat{\gtP})=d-1$, so $\gtQ_i$ is a minimal prime ideal of ${\mathcal S}(M)$ by Theorem \ref{minimalnlc}. Thus, $\gtq_i$ is a minimal prime ideal of ${\mathcal S}^*(M)$. Of course, $\gtq_i\neq\gtq_j$ if $i\neq j$ because otherwise 
$$
\gtp_i\in\cl_{\Speca(M)}(M_i)\cap\cl_{\Speca(M)}(M_j)=\cl_{\Speca(M)}(M_i\cap M_j)=\cl_{\Speca(M)}(T)
$$
and this is impossible because $\dim(T)=d-1$ and $\dim(Z(f))=d$ for each $f\in\gtp_i$. 

Conversely, let $\gtq$ be a minimal prime ideal of ${\mathcal S}(M)$ contained in $\widehat{\gtp}$. Then $\gtq\not\in\cl_{\Speca(M)}(M_0)$, so $\gtq\in\cl_{\Speca(M)}(M_i)$ for some $i=1,\ldots,s$. Since $\cl_{\Speca(M)}(M_i)\cong\Speca(M_i)$ and $\widehat{\gtp}_i$ contains exactly one minimal prime of ${\mathcal S}(M_i)$, we deduce $\gtq=\gtq_i$, so there are exactly $s$ minimal prime ideals of ${\mathcal S}(M)$ contained in $\widehat{\gtp}$.

\paragraph{} Finally, as $\widehat{\gtp}_i$ contains exactly one minimal prime ideal of ${\mathcal S}(M_i)$ and ${\tt d}_{M_i}(\widehat{\gtP}_i)=d-1$, we deduce by \ref{fiberspec3b} that $\Speca({\tt j}_i)^{-1}(\gtp_i)$ is a singleton. Thus, the size of $\Speca({\tt j})^{-1}(\gtp)$ is equal to $s$, so it coincides with the number of minimal prime ideals of ${\mathcal S}(M)$ contained in $\widehat{\gtp}$, as required.
\end{proof}


\begin{thebibliography}{22}

\bibitem{bf} J. Bingener, H. Flenner: On the fibers of analytic mappings. {\em Complex analysis and geometry. Univ. Ser. Math.} (1993), 45--101.

\bibitem{bcr} J. Bochnak, M. Coste, M.-F. Roy: Real algebraic geometry. {\em Ergeb. Math. } {\bf 36}, Springer-Verlag, Berlin: 1998.

\bibitem{br} G. Brumfiel: The real spectrum compactification of Teichm\"uller space. {\em Contemporary Math.} {\bf 74} (1988), 51--75.

\bibitem{cc} M. Carral, M. Coste: Normal spectral spaces and their dimensions. {\em J. Pure Appl. Algebra} {\bf 30} (1983), 227--235.

\bibitem{cd1} G.L. Cherlin, M.A. Dickmann: Real closed rings I. Residue rings of rings of continuous functions. {\em Fund. Math.} {\bf126} (1986), no. 2, 147--183.

\bibitem{cd2} G.L. Cherlin, M.A. Dickmann: Real closed rings II. Model theory. {\em Ann. Pure Appl. Logic} {\bf25} (1983), no. 3, 213--231.

\bibitem{dk} H. Delfs, M. Knebusch: Separation, Retractions and homotopy extension in semialgebraic spaces. {\em Pacific J. Math.} {\bf114} (1984), no. 1, 47--71.

\bibitem{dk2} H. Delfs, M. Knebusch: Locally semialgebraic spaces. {\em Lecture Notes in Mathematics}, {\bf1173}. Springer-Verlag, Berlin: 1985.

\bibitem{fe1} J.F. Fernando: On chains of prime ideals in rings of semialgebraic functions. {\em Q. J. Math.} {\bf 65} (2014), no. 3, 893--930. 

\bibitem{fg1} J.F. Fernando, J.M. Gamboa: On \L ojasiewicz's inequality and the Nullstellensatz for rings of semialgebraic functions. {\em J. Algebra} {\bf 399} (2014), 475--488.

\bibitem{fg2} J.F. Fernando, J.M. Gamboa: On the Krull dimension of rings of semialgebraic functions. {\em Rev. Mat. Iberoam.} {\bf 31} (2015), no. 3, 753--766.

\bibitem{fg3} J.F. Fernando, J.M. Gamboa: On the spectra of rings of semialgebraic functions. \em Collect. Math. \em {\bf 63} (2012), no. 3, 299--331.

\bibitem{fg5} J.F. Fernando, J.M. Gamboa: On the semialgebraic Stone--\v{C}ech compactification of a semialgebraic set. \em Trans. AMS \em {\bf 364} (2012), no. 7, 3479--3511.

\bibitem{fg6} J.F. Fernando, J.M. Gamboa: On open and closed morphisms between semialgebraic sets. \em PAMS \em {\bf 140} (2012), no. 4, 1207--1219.

\bibitem{g} J.M. Gamboa: On prime ideals in rings of semialgebraic functions. {\em Proc. Amer. Math. Soc.} {\bf 118} (1993), no. 4, 1034--1041.

\bibitem{gr} J.M. Gamboa, J.M. Ruiz: On rings of semialgebraic functions. {\em Math. Z.} {\bf 206} (1991) no. 4, 527--532.

\bibitem{gj} L. Gillman, M. Jerison: Rings of continuous functions. {\em The Univ. Series in Higher Nathematics} {\bf 1}, D. Van Nostrand Company, Inc.: 1960.

\bibitem{grm} H. Grauert, R. Remmert: Coherent analytic sheaves. {\em Grundlehren der mathematischen Wissenschaften.} {\bf 265}, Springer-Verlag, Berlin: 1984.

\bibitem{gd} A. Grothendieck: \'El\'ements de g\'eom\'etrie alg\'ebrique. IV. \'Etude locale des sch\'emas et des morphismes de sch\'emas. II. \em Inst. Hautes \'Etudes Sci. Publ. Math. \em {\bf24} (1965) 231 pp.

\bibitem{ps} A. Prestel, N. Schwartz: Model theory of real closed rings. Valuation theory and its applications, Vol. I (Saskatoon, SK, 1999), 261--290, \em Fields Inst. Commun.\em, {\bf32}, Amer. Math. Soc., Providence, RI, 2002. 

\bibitem{sp} E. H. Spanier: Algebraic Topology. McGraw-Hill Book Co., New York-Toronto, Ont.-London: 1966.

\bibitem{s0} N. Schwartz: Real closed rings. \em Habilitationsschrift\em, M\"unchen: 1984.

\bibitem{s1} N. Schwartz: Real closed rings. Algebra and order (Luminy-Marseille, 1984), 175--194, {\em Res. Exp. Math.}, {\bf14}, Heldermann, Berlin: 1986.

\bibitem{s2} N. Schwartz: The basic theory of real closed spaces. {\em Mem. Amer. Math. Soc.} {\bf77} (1989), no. 397.

\bibitem{s3} N. Schwartz: Rings of continuous functions as real closed rings. Ordered algebraic structures (Curaçao, 1995), 277--313, Kluwer Acad. Publ., Dordrecht: 1997.

\bibitem{s4} N. Schwartz: Epimorphic extensions and Prü\"ufer extensions of partially ordered rings. \em Manuscripta Math. \em {\bf102} (2000), 347--381.

\bibitem{s5} N. Schwartz: Open locally semi-algebraic maps. \em J. Pure Appl. Algebra \em {\bf53} (1988), no. 1-2, 139?169.

\bibitem{sm} N. Schwartz, J.J. Madden: Semi-algebraic function rings and reflectors of partially ordered rings. {\em Lecture Notes in Mathematics}, {\bf1712}. Springer-Verlag, Berlin: 1999.

\bibitem{scht} N. Schwartz, M. Tressl: Elementary properties of minimal and maximal points in Zariski spectra. {\em J. Algebra} {\bf323} (2010), no. 3, 698--728.

\bibitem{t0} M. Tressl: The real spectrum of continuous definable functions in o-minimal structures. {\em S\'eminaire de Structures Alg\'ebriques Ordonn\'ees} 1997-1998, {\bf68}, Mars 1999, p. 1--15.

\bibitem{t1} M. Tressl: Super real closed rings. {\em Fund. Math.} {\bf194} (2007), no. 2, 121--177.

\bibitem{t2} M. Tressl: Bounded super real closed rings. Logic Colloquium 2007, 220--237, \em Lect. Notes Log.\em, {\bf35}, Assoc. Symbol. Logic, La Jolla, CA, 2010.


\end{thebibliography}
\end{document}